\documentclass{amsart}
\usepackage{stmaryrd}
\usepackage{amssymb}
\usepackage{mathrsfs}
\usepackage{cases}
\usepackage{amsmath}
\usepackage{amsfonts}
\usepackage{pifont}
\usepackage{graphicx,amssymb,mathrsfs,amsmath}

 \newtheorem{thm}{Theorem}[section]
 \newtheorem{cor}[thm]{Corollary}
 \newtheorem{lem}[thm]{Lemma}
 \newtheorem{prop}[thm]{Proposition}
 \theoremstyle{definition}
 \newtheorem{defn}[thm]{Definition}
 \newtheorem{example}[thm]{Example}
 
 \theoremstyle{remark}
 \newtheorem{rem}[thm]{Remark}
 
 \numberwithin{equation}{section}
 \DeclareMathOperator{\RE}{Re}
 \DeclareMathOperator{\IM}{Im}

\allowdisplaybreaks

\begin{document}

\title[Besov spaces associated with operators]{Besov spaces associated with non-negative operators on Banach spaces}
\thanks{2020 Mathematics Subject Classification: 30H25; 46E40; 47B12; 47A60}
\thanks{Keywords: Non-negative operators, Fractional powers, Besov spaces, Interpolation spaces}
\thanks{This work was partially supported by the 
	China Scholarship Council (Grant No. 201906240015) 
	and the 
NNSF of China (Grant No. 11971327)}
\date{}

\author{ Charles Batty }
\address{St. John's College \\
University of Oxford \\ Oxford OX1 3JP \\ UK}
\email{charles.batty@sjc.ox.ac.uk}

\author{ Chuang Chen }
\address{Department of Mathematics \\
Sichuan University \\ Chengdu 610064 \\ P. R. China 
} 
\email{ccmath@scu.edu.cn}

\maketitle

\begin{abstract}
Motivated by a variety of representations of fractional powers of operators, we develop the theory of abstract Besov spaces $B^{ s, A }_{ q, X }$ for 
non-negative operators $A$ on Banach spaces $X$ with a full range of indices $s \in \mathbb{R}$ and $0 < q \leq \infty$. The approach we use is the dyadic decomposition of resolvents for non-negative operators, an analogue of the Littlewood-Paley decomposition in the construction of the classical Besov spaces. 
In particular, by using the 
reproducing formulas for fractional powers of operators and explicit quasi-norms estimates for
Besov spaces 
we discuss
the connections between the smoothness of Besov spaces associated with
operators and the boundedness of fractional powers of the underlying operators.

\end{abstract}

\tableofcontents

\medskip


\section{Introduction}\label{S:Introduction}

This paper is devoted to the theory of Besov spaces from the operator-theoretic point of view and extending some recent work on Besov spaces associated with operators (see, for example, Besov spaces associated with heat kernels \cite{BDY2012} 
or sectorial operators \cite{Haase2005} or, in particular, sectorial operators of angle zero \cite{Kriegler and Weis2016}). A new point in our development of the Besov space $B^{ s, A }_{ q, X }$ associated with an operator $A$ is that we allow the operator $A$ to be a general non-negative operator (not necessarily defined densely or with dense range) on a Banach space $X$ and allow the space $B^{ s, A }_{ q, X }$ to admit a full range of the smoothness index $s \in \mathbb{R}$ and size index $0 < q \le \infty$. 

The story originates from the theory of fractional powers of operators.
Roughly speaking, the concept of fractional powers of operators is a continuous analogue of the discrete scale of regularity associated with operators in the theory of abstract Cauchy problems and PDEs 
(see \cite[Chapters 3 and 5]{MartinezM2001} for more information on the history and devolopment of fractional powers of operators).
In general, it is difficult to describe accurately the domains of fractional powers (i.e., the fractional domains for short) for non-negative operators on Banach spaces. To some extent, an effective substitute for the fractional domain is 
the so-called
interpolation space, especially in the study of regularity questions (see \cite{Bergh and Lofstrom M} for the theory of
interpolation spaces and \cite{LunardiM1995,YagiM2010} for the application of interpolation spaces in the regularity theory of abstract Cauchy problems).

In a series of papers \cite{Komatsu1966}-\cite{Komatsu1972-73}, H. Komatsu developed the theory of fractional powers of non-negative operators on Banach spaces.
In particular, in \cite
{Komatsu1967} H. Komatsu revealed that the real interpolation space $( X, D ( A^\alpha ))_{ \theta, q }$ between the fractional domain $( D ( A^\alpha ), \left\| A^\alpha \cdot \right\|  )$ and the underlying space $(X, \left\| \cdot \right\| )$ can be characterized via a novel space $D^s_q ( A )$.
More precisely, let $A$ be a non-negative operator defined densely on a Banach space $(X, \left\| \cdot \right\| )$. For $s > 0$ and $1 \le q \le \infty$, the so-called Komatsu space $D^s_q ( A )$ is given by 
\begin{align}\label{d:Komatsu spaces}
D^s_q ( A ) :=  \left\{ x \in X: \int_{0}^{\infty} \left\| t^{ s } A^n ( t + A )^{ - n } x \right\|^q \, \frac{ d t }{ t } < \infty \right\}
\end{align}
endowed with the norm 
\begin{align*}
\left\| x \right\|_{ D^s_q ( A ) } := \| x \| + \left\{ \int_{0}^{\infty} \left\| t^{ s } A^n ( t + A )^{ - n } x \right\|^q \, \frac{ d t }{ t } \right\}^{ 1 / q }  
\end{align*}
(with the usual modification if $ q = \infty $), where $n$ is an integer greater than $s$. It can be verified that the norm $\left\| \cdot \right\|_{ D^s_q ( A ) }$ is independent of integers $n > s$ in the sense of equivalent norms \cite[Proposition 1.2]{Komatsu1967}, so that $D^s_q ( A )$ is well-defined as a Banach space (in the sense of equivalent norms). 
Furthermore, it also can be verified that the interpolation space $( X, D ( A^\alpha ))_{ \theta, q }$ coincides with $D^{ \theta  \alpha }_q ( A )$ for $\alpha > 0$ and $0 < \theta < 1$ (see \cite[Theorems 3.1 and 3.2]{Komatsu1967}).


From the operator-theoretic point of view, the Komatsu spaces $D^s_q ( A )$ yield a pioneering style of abstract (inhomogeneous) Besov spaces.
Subsequently, T. Muramatu \cite{MuramatuM1985} discussed the inhomogeneous Besov spaces $D^\varphi_q ( A )$ with $0 < q \leq \infty$ and $\varphi ( s )$, a weight function on $\mathbb{R}_+$, for non-negative operators $A$ on quasi-Banach spaces $X$, where the Besov quasi-norm was essentially the same as that of the Komatsu space $D_q^s ( A )$ with $s > 0$. An alternative improvement of the abstract inhomogeneous Besov spaces $D^s_q ( A )$ with the smoothness index extended from $s > 0$ to 
$s \in \mathbb{R}$ was given by T. Kobayashi and T. Muramatu \cite{Kobayashi and Muramatu1992} while the size index is still restricted in the range of $1 \leq q \leq \infty$. Furthermore, abstract (homogeneous) Besov spaces $\dot{B}_q^\varphi ( A )$ with $1 \leq q \leq \infty$ and $\varphi$, the weight function on $\mathbb{R}_+$ mentioned above, were discussed 
by T. Matsumoto and T. Ogawa \cite{Matsumoto and Ogawa2010} for non-negative operators $A	$ on Banach spaces. 

Recall that a closed linear operator on a Banach space is non-negative if and only if it is sectorial (see \cite[Proposition 1.2.1]{MartinezM2001}). 
Thanks to sectoriality, an alternative approach to fractional powers of non-negative operators on Banach spaces is the so-called functional calculus (see, for example, \cite[Chapter 3]{HaaseM2006}). Using the language of functional calculus, M. Haase \cite[Section 7]{Haase2005} revealed that the Komatsu spaces $D^s_q ( A )$ mentioned above follow a common pattern, i.e.,
\begin{align*}
D^s_q ( A ) = \left\{ x \in X: \int_{0}^{\infty} \left\| t^{ - s } \psi ( t A ) x \right\|^q \, \frac{ d t }{ t } < \infty \right\} 
\end{align*}
in the sense of equivalent norms 
\begin{align*}
\left\| x \right\|_{ D^s_q ( A ) } \simeq \left\| x \right\| + \left\{ \int_{0}^{\infty} \left\| t^{ - s } \psi ( t A ) x \right\|^q \, \frac{ d t }{ t } \right\}^{ 1 / q }
\end{align*}
(with the usual modification if $q = \infty$) for $s > 0$ and $1 \leq q \leq \infty$ , where $\psi$ is a bounded holomorphic function satisfying an appropriate decay estimate on some sector relating to $A$. In other words, the Komatsu spaces can be characterized by using the natural functional calculus for sectorial operators on Banach spaces. In particular, applying $\psi ( z ) = z^\alpha ( 1 + z )^{ - \alpha }$ with $0 < s < \RE \alpha$ to the functional calculus of $A$ yields the Komatsu space $D^s_q ( A )$, immediately (also, see \cite[Definition 11.3.1]{MartinezM2001}). We refer the reader to \cite[Chapter 6]{HaaseM2006} for more information on this topic. 

More recent work due to C. Kriegler and L. Weis \cite[Section 5]{Kriegler and Weis2016} reveals that the Komatsu spaces can also be characterized via the Mihlin functional calculus for sectorial operators with angle zero. More precisely, let $A$ be a sectorial operator of angle zero on a Banach space $X$ as defined in \cite[Section 5]{Kriegler and Weis2016}, where an additional assumption that $\overline{ R ( A ) } = X$ is posed in the definition, so that $A$ is injective whenever it is sectorial of angle zero. Then the Komatsu spaces $D^s_q ( A )$ admit the so-called Littlewood-Paley decomposition, i.e., 
\begin{align*}
D^s_q ( A ) = B^s_q ( A ) := \left\{ x \in \dot{X}_{ - N } + \dot{X}_N: \sum_{ n = 0 }^\infty \left\| 2^{ n s } \varphi_n ( A ) x \right\|^q < \infty \right\}  
\end{align*}
in the sense of equivalent norms 
\begin{align*}
\left\| x \right\|_{ D^s_q ( A ) } \simeq \left\| x \right\|_{ B^s_q ( A ) } := \left\{ \sum_{ n = 0 }^\infty \left\| 2^{ n s } \varphi_n  ( A ) x \right\|^q \right\}^{ 1 / q }  
\end{align*}
(with the usual modification if $q = \infty$) for $s \in \mathbb{R}$ and $1 \leq q \leq \infty$, where $\{ \varphi_n \}$ is an inhomogeneous partition of unity on $\mathbb{R}_+$ and $\dot{X}_{ \pm N }$, with $N > | s |$, are the extropolation spaces given by the completions of $D ( A^{ \pm N } )$ with respect to $\left\| A^{ \pm N } \cdot \right\|$. 
A historical account of the classical theory of Besov spaces on $\mathbb{R}^n$ is beyond the scope of this work, and we refer the reader to, for instance, \cite{TriebelM1983-2010,
Han1994,
YuanM2010,SawanoM2018}. We shall, however, recall briefly the main techniques used to develop the theory of Besov spaces which are relevant to our work. Recall that the classical Besov spaces $B^{s}_{p, q} ( \mathbb{R}^n )$ are related closely to the negative Laplacian
or its square root
on the Euclidean space $\mathbb{R}^n$. Indeed, applying $A = - \Delta_p$, i.e., the negative Laplacian on $L^p ( \mathbb{R}^n )$ with $1 < p < \infty$, to the Komatsu space $D^s_q ( A )$ yields the classical Besov space $B^s_{ p, q } ( \mathbb{R}^n )$ for $s > 0$ and $1 \le q \le \infty$ in the sense of equivalent norms. So far, compared with the classical approaches to Besov spaces (for example, the initial method of moduli of continuity due to O. V. Besov \cite{Besov1961}, the Hardy--Littlewood maximal operator characterizations associated with the Poisson kernel and Gaussian kernel due to H. Taibleson \cite{Taibleson1964}, the Fourier analytic method and the real variable techniques due to J. Peetre \cite{PeetreM1976} and the approach via the Calder\'{o}n reproducing formulas due to H.-Q. Bui, M. Palusz\'{y}nski and M. H. Taibleson \cite{BuiPT1997}),
the dyadic decomposition of resolvents provides an alternative approach to the Besov spaces from the operator-theoretic point of view.

In addition to the real-variable techniques developed in the framework of the classical harmonic analysis, there is some recent work on abstract Besov spaces associated with the negative generators of strongly continuous (or, in particular, bounded analytic) semigroups. More precisely, the work connected with Peetre's approach includes, for instance, \cite{Dancona and Pierfelice2005,
Benedetto and Zheng2010,Hu2014,Mayeli2016,Bui2020}, where Besov spaces associated with self-adjoint positive operators or the Schr\"{o}dinger operators are discussed. The work in \cite{BDY2012} is devoted to a class of
homogeneous Besov spaces $\dot{B}^{s, L}_{p, q}$ with $-1 < s < 1$ and $1 \leq p, q \leq \infty$ 
for the negative generators $L$ of analytic semigroups with heat kernels satisfying an upper bound of Poisson type and H\"{o}lder continuity. And Besov spaces $B^s_{ p, q } (L)$ with a full range of indices $s \in \mathbb{R}$ and $0 < p, q \leq \infty$ are systematically studied 
in \cite{Kerkyacharian and Petrushev2015} for the negative generators $L$ of semigroups with heat kernels satisfying an upper bound of Gauss type, H\"{o}lder continuity and a Markov property. 
Moreover, a quite recent work on Lipschitz spaces $\Lambda^s$ (i.e., $B^s_{ \infty, \infty }$) associated with the Schr\"{o}dinger operators is given in \cite{DeLeon and Torrea2019} by using the language of semigroups.


Our work starts from a unified representation of fractional powers of non-negative operators on Banach spaces. Let $A$ be a non-negative operator on a Banach space $X$ and let $z \in \mathbb{C}_+$. It can be verified that 
\begin{align}\label{e:representation of fractional powers-all}
A^z x = \frac{ \Gamma ( \alpha + \beta ) }{ \Gamma ( \alpha + z ) \Gamma ( \beta - z ) } \int_{0}^{\infty} \lambda^{ z + \alpha } A^\beta ( \lambda + A )^{ - \alpha - \beta } x \, \frac{ d \lambda }{ \lambda }, \quad x \in D ( A^{ z + \epsilon } ), 
\end{align}
where $\epsilon > 0$ and $\alpha, \beta \in \mathbb{C}_+$ satisfying $- \RE \alpha < \RE z < \RE \beta$ (see Proposition \ref{p:representation of fractional powers-all} below). It is the representation (\ref{e:representation of fractional powers-all}) that motivates us to deal with a dyadic decomposition of the integrand and define 
the inhomogeneous Besov space $B^{ s, A }_{ q, X }$ associated with $A$ to be the completion of the subspace 
\begin{align*}
\left\{ x \in \overline{ D ( A ) }: \sum_{ i = k }^{ \infty } \big\| 2^{ i ( s + \alpha ) } A^\beta ( 2^i + A )^{ - \alpha - \beta } x \big\|^q < \infty \right\} \subset X
\end{align*}
with respect to the quasi-norm 
\begin{align*}
\left\| x \right\|_{ B^{ s, A }_{ q, X } } := \left\| ( 2^k + A )^{ - \alpha } x \right\| + \left\{ \sum_{ i = k }^{ \infty } \big\| 2^{ i ( s + \alpha ) } A^\beta ( 2^i + A )^{ - \alpha - \beta } x \big\|^q \right\}^{ 1 / q } 
\end{align*}
(with the usual modification if $q = \infty$) for $s \in \mathbb{R}$ and $0 < q \le \infty$, where $k \in \mathbb{Z}$ and $\alpha, \beta \in \mathbb{C}_+^*$ satisfying $- \RE \alpha < s < \RE \beta$ (see Definition \ref{d:inhomogeneous Besov spaces} below).

Indeed, the dyadic decomposition of resolvents of non-negative operators is an analogue of the standard frequency restriction in the construction of the classical Besov spaces on $\mathbb{R}^n$ and it allows us to construct a variety of Besov spaces 
with a full range of the smoothness index $s \in \mathbb{R}$ and size index $0 < q \leq \infty$ (see Section \ref{S:Besov spaces} below). 
Moreover, the continuous scale of indices $\alpha$ and $\beta$ in the representation (\ref{e:representation of fractional powers-all}) allows us to describe more exactly the lifting property as well as other properties of interest for fractional powers of non-negative operators on Besov spaces $B^{ s, A }_{ q, X }$. Some results are novel even in the case $1 \le q \le \infty$. More precisely, thanks to the representation (\ref{e:representation of fractional powers-all}) with continuous scale of indices $\alpha$ and $\beta$
and reproducing formula (\ref{e:reproducing formula-ab}) of fractional powers of operators, Theorems \ref{t:lifting property-negative} and \ref{p:interpolation}  
provide equivalent quasi-norms for the lifting and interpolation of Besov spaces,
which improves \cite[Theorem 3.1]{Komatsu1967} and \cite[Corollary 7.3 (b)]{Haase2005} in the sense of equivalent (quasi-)norms even in the case $1 \le q \le \infty$ (see Remark \ref{r:interpolations} below for more information).


The paper is organized as follows. In Section \ref{S:Fractional powers of operators} we provide the reader with a concise introduction to the concepts of non-negative operators and their fractional powers and present some estimates and representations of fractional powers of operators used in this paper. 
Section \ref{S:Besov spaces} is devoted to the construction of (inhomogeneous and homogeneous) Besov spaces associated with non-negative operators (i.e., abstract Besov spaces for short). 
Section \ref{S:Basic properties} contains some basic properties of abstract Besov spaces, such as quasi-norm equivalence, continuous embedding and translation invariance. In Section \ref{S:Fractional powers} we discuss the connections between the smoothness of abstract Besov spaces and the fractional powers of the underlying operators, including the lifting property, smoothness reiteration and interpolation spaces. Finally, in Section \ref{S:Comparison} we compare our new Besov spaces with some classical ones, as well as the known Besov spaces developed by the integral transform method, or the functional calculus approach or the Littlewood-Paley decomposition technique, and we show that our new Besov spaces recover many known Besov spaces in the literature.

\bigskip

\noindent \emph{Notation}. Throughout this paper, we shall use the following notation:
\begin{itemize}
\item[] $\mathbb{R}_+ := [ 0, \infty )$, 
\item[] $\mathbb{C}_+ := \{ \alpha \in \mathbb{C}: \RE \alpha > 0 \}$ and $\mathbb{C}_+^* := \mathbb{C}_+ \cup \{ 0 \}$, 
\item[] $\Sigma_\theta := \{ z \in \mathbb{C} \setminus \{ 0 \}: | \arg z | < \theta \} $ for $\theta \in ( 0, \pi )$.
\end{itemize}
We use the notation $ f ( t ) \equiv c$ to mean that $f$ is a constant function taking the value $c$. Also, we use the notation $a \lesssim b$ to mean that $a \leq C b$ for some positive constant $C$ independent of other relevant quantities and the notation $a \simeq b$ to mean that $a \lesssim b \lesssim a$ for the sake of simplicity. Moreover, by $q'$ we denote the conjugate of $q \in ( 1, \infty )$, i.e., $q' \in ( 1, \infty )$ satisfies that
\begin{align*}
\frac{ 1 }{ q } + \frac{ 1 }{ q' } = 1.
\end{align*}
For the consistency of notation, two standard classical inequalities are formulated as follows, which we shall need to get explicit quasi-norm estimates for abstract Besov spaces. One is, for $0 < q < \infty$, 
\begin{align}\label{e: q-inequality}
( a + b )^q \leq C_q ( a^q + b^q ), \quad a, b \geq 0,
\end{align}
with $C_q = \max\{ 1, 2^{ q - 1 } \}$ and the other is, for $0 < q \leq 1$,
\begin{align}\label{e: q-inequality-series<1}
\left( \sum_{ i = 1 }^\infty a_i \right)^q & \leq \sum_{ i = 1 }^\infty 
a_i^q, \quad a_i \geq 0, i \in \mathbb{N}.
\end{align} 
By $\Gamma$ we denote the well-known Gamma function. In this paper we also shall need the following two classical integrals associated with the Gamma function.
One is 
\begin{align}\label{e:Euler integral}
\frac{ \Gamma ( n ) }{ \Gamma ( \alpha ) \Gamma ( n - \alpha ) } \int_0^\infty \lambda^{ \alpha }  ( 1 + \lambda )^{ - n } \, \frac{ d \lambda }{ \lambda } = 
1, \quad 0 < \RE \alpha < n \in \mathbb{N}, 
\end{align}
and the other is 
\begin{align}\label{e:fractional powers-resolvent-scalar}
\frac{ 1 }{ \Gamma ( \alpha ) \Gamma ( 1 - \alpha ) } \int_0^\infty \frac{ \lambda \mu^\alpha }{ \lambda^2 + 2 \lambda \mu^\alpha \cos \pi \alpha + \mu^{ 2 \alpha } } \, \frac{ d \mu }{ \mu } = 1, \quad 0 < \RE \alpha < 1.
\end{align}
For a Banach space $X$, $\mathcal{L} ( X )$ denotes the Banach algebra of all bounded linear operators on
$X$. The domain, range, spectrum and resolvent set of an (unbounded) operator $A$ on $X$ are
denoted by $D(A)$, $R ( A )$, $\sigma (A)$ and $\rho ( A )$, respectively. Moreover, for an operator $A$ on $X$, by $A|_Y$ we denote the part of $A$ in $Y$, i.e., $A|_Y y = A y$ for $y \in D \left( A|_Y \right) := \left\{ y \in D ( A ) \cap Y: A y \in Y \right\}$. 


\smallskip


\section{Fractional powers of operators}\label{S:Fractional powers of operators}

In this section we present a number of estimates and representations for the fractional powers of non-negative operators which we shall need for what follows. Some of this material is quite standard and some is relatively new. 
$X$ is a complex Banach space and $A$ is a closed linear operator on $X$ in this section.

\smallskip

\subsection{Non-negative operators}

The class of non-negative operators in Banach spaces was introduced by A. V. Balakrishnan \cite{Balakrishnan1960} to discuss the fractional powers of operators and subsequently studied by H. Komatsu in a series of papers \cite{Komatsu1966}-\cite{Komatsu1972-73}. 
In particular, in \cite{Komatsu1969} H. Komatsu gave them this nomenclature.

Recall that $A$ is said to be non-negative on $X$ if $(- \infty, 0) \subset \rho (A)$ and 
\begin{align}\label{e:non-negative constant}
M_A := \sup_{ \lambda > 0 } \left\| \lambda (\lambda + A)^{-1} \right\| < \infty.
\end{align}
Thanks to a simple decomposition of the identity operator, i.e., 
\begin{align}\label{e:identity decomposition}
I = \lambda ( \lambda + A )^{-1} + A ( \lambda + A )^{-1}, \quad \lambda \in \rho ( A ), 
\end{align}
it can be seen that $A$ is non-negative on $X$ if and only
if $(- \infty, 0) \subset \rho (A)$ and
\begin{align}\label{e:non-negative constant 2}
L_A := \sup_{ \lambda > 0 } \left\| A (\lambda + A)^{-1} \right\| < \infty.
\end{align}
For convenience, we call $M_A$ and $L_A$ the non-negativity constants of $A$. Moreover, thanks to (\ref{e:non-negative constant 2}), if a non-negative operator $A$ is injective then the inverse $A^{ - 1 }$ of $A$ is also non-negative with the non-negativity constants $M_{ A^{ - 1 } } = L_A$ and $L_{ A^{ - 1 } } = M_A$.
Sometimes a non-negative operator $A$ with $0 \in \rho ( A )$ is also called a positive operator.

 
Also, recall that a family $ \{ A_t \}_{ t \in \Lambda } $ of non-negative operators $A_t$ is said to be uniformly non-negative if
	\begin{align}\label{e:uniformly non-negative constant}
    M = M_{ \{ A_t \}_{ t \in \Lambda } } := \sup_{ t  \in \Lambda } M_{ A_t } < \infty,
	\end{align}
	where $M_{ A_t }$ is the non-negativity constant of $A_t$ given by (\ref{e:non-negative constant}).
Clearly, a family $ \{ A_t \}_{ t \in \Lambda} $ of non-negative operators $A_t$  is uniformly non-negative if and only if 
\begin{align}\label{e:uniformly non-negative constant 2}
L = L_{ \{ A_t \}_{ t \in \Lambda } } := \sup_{ t  \in \Lambda } L_{ A_t } < \infty,
\end{align}
where $L_{ A_t }$ is the non-negativity constant of $A_t$ given by (\ref{e:non-negative constant 2}).
Moreover, we call $M$ and $L$ the uniform non-negativity constants of
$\{ A_t \}_{ t \in \Lambda}$ for convenience.

Some uniformly non-negative families are listed as follows, which are simple but important in the theory of fractional powers of operators and will be used frequently in the sequel.

\begin{example}\label{E:uniform non-negativity and boundedness}
Let $A$ be non-negative on $X$. The following statements hold.
\begin{itemize}
\item[(i)] $ \{ t (t + A )^{-1} \}_{ t > 0 } $ is
uniformly non-negative with the uniform non-negativity constants $M_{ \left\{ t ( t + A )^{-1} \right\}_{ t > 0 } } \leq M_A + L_A$ and $L_{ \left\{ t ( t + A )^{-1} \right\}_{ t > 0 } } \leq M_A$.
\item[(ii)] $ \{ A (t + A )^{-1} \}_{ t > 0 } $ is
uniformly non-negative with the uniform non-negativity constants $M_{ \{ A ( t + A )^{-1} \}_{ t > 0 } } \leq M_A + L_A$ and $L_{ \{ A ( t + A )^{-1} \}_{ t > 0 } } \leq L_A$.
\item[(iii)] $ \{ (s + A) ( t + A )^{-1} \}_{ s, t > 0 } $ is uniformly non-negative with the uniform non-negativity constants $M_{ \{ (s + A) (t + A )^{-1} \}_{ s, t > 0 } }, L_{ \{ (s + A) (t + A )^{-1} \}_{ s, t > 0 } } \leq M_A + L_A$. 
\end{itemize}
\end{example}

Finally, it is necessary to point out that a concept related closely to non-negative operators is the so-called sectorial operators due to T. Kato \cite{Kato1960}. It is known nowadays that a closed linear operator on a Banach space is non-negative if and only if it is sectorial (see \cite[Proposition 1.2.1]{MartinezM2001}). Thanks to sectoriality, it is possible to develop the theory of functional calculi for non-negative operators on Banach spaces. Sectorial operators and functional calculi of them will be discussed in Section \ref{Sub:Functional calculus} 
below. We also refer the reader to, for instance, \cite{
Kunstmann and Weis2004,Vitse2005,AuscherM1998,
Gomilko2015,Batty2017,Kunstmann2017,Batty2019,Ferguson2019}, \cite[Chapter 4]{MartinezM2001} and \cite[Chapter 2]{HaaseM2006} for more information on this topic.

\smallskip


\subsection{Fractional powers of operators}\label{sub:Fractional powers of operators}



Generally speaking, there are three different approaches to fractional powers of non-negative operators on Banach spaces. One is the integral operator approach, including the complex integrals, i.e., the so-called Dunford-Riesz integrals due to sectoriality of non-negative operators (see, for example, \cite[Section 2.6]{Pazy1983M}) and real integrals, i.e., the so-called Balakrishnan-Komatsu operators (see, for example, \cite[Chapter 9]{MuramatuM1985} and \cite[Chapters 3 and 5]{MartinezM2001}). 
Another is a little more abstract approach via the so-called function calculi, such as the Hirsch functional calculus, $H^\infty$-functional calculus, natural functional calculus and so on (see \cite{deLaubenfels1993,MartinezM2001,HaaseM2006,Gomilko2015} and reference therein). 
And the third approach involves some abstract (incomplete) Cauchy problem and the fractional powers of the underlying operator can be given by the (unique) solution of the abstract (incomplete) Cauchy problem in a suitable way (see  \cite{Caffarelli2007,StingaT2010,Stinga2010,Gale2013,Cabre2016,Arendt2018}, especially a quite recent work \cite[Theorem 6.2]{Meichsner2019}). 

Now we recall 
fractional powers of non-negative operators starting from real integrals. Let $A$ be non-negative on $X$ and let $\alpha \in \mathbb{C}_+$. 
Initially, suppose that $A \in \mathcal{L} ( X )$, one can define the fractional power $A^\alpha$ of $A$ by the Balakrishnan-Komatsu integral, i.e., 
\begin{align}\label{e:integral-Balakrishnan}
A^{ \alpha } := \frac{ \Gamma ( n ) }{ \Gamma ( \alpha ) \Gamma ( n - \alpha ) } \int_0^\infty \lambda^{ \alpha } A^n ( \lambda + A )^{ - n } \, \frac{ d \lambda }{ \lambda },
\end{align}
where $\RE \alpha < n \in \mathbb{N}$. 
Thanks to (\ref{e:non-negative constant}) and (\ref{e:non-negative constant 2}), it can be seen that the integral given in the right-hand side of (\ref{e:integral-Balakrishnan}) absolutely converges.
It can also be seen that the part given in the right-hand side of (\ref{e:integral-Balakrishnan}) is independent of integers $n > \RE \alpha$ which 
is a simple consequence of the use of integration by parts. 
This implies that $A^\alpha$ given in (\ref{e:integral-Balakrishnan}) is well-defined as a bounded linear operator on $X$. 
By using the well-known resolvent equation (see \cite[Appendix B, Proposition B.4]{ABHN2001}), 
one can verify the additivity of $A^\alpha$.
Furthermore, by use of additivity of fractional powers one can also verify the injectivity of $A^\alpha$ whenever $A$ is injective. 
 
The fractional powers of unbounded non-negative operators can be defined by using a standard approximation technique. More precisely, 
if $A$ is unbounded with $0 \in \rho ( A )$, it makes sense to consider the fractional power $( A^{ - 1 } )^\alpha$ of $A^{ - 1 }$ and, thanks to the injectivity of $( A^{ - 1 } )^\alpha$ as mentioned above, it is possible to define the fractional power $A^\alpha$ of $A$ as the inverse of $( A^{ - 1 } )^\alpha$, i.e., 
\begin{align*} 
A^\alpha := \left[ ( A^{ - 1 } )^\alpha \right]^{ - 1 }.
\end{align*}
And if $A$ is unbounded while $0 \in \sigma ( A )$, it makes sense to consider the fractional power $( A + \epsilon )^\alpha$ of $A + \epsilon$ with $\epsilon > 0$ due to the fact that $0 \in \rho ( A + \epsilon )$ and define the fractional power $A^\alpha$ of $A$ by the strong limit 
\begin{align*} 
A^\alpha := s-\lim_{ \epsilon \rightarrow 0 } ( A + \epsilon )^\alpha  
\end{align*}
with maximal domain (see \cite[Definitions 5.1.2 and 5.1.3]{MartinezM2001}).

The positive powers of non-negative operators not only look like the classical powers of numbers but also work like the classical powers of numbers. More precisely, they satisfy the following two laws of exponents. One is additivity, i.e., 
\begin{align}\label{e:additivity of fractional powers}
A^\alpha A^\beta = A^{ \alpha + \beta }, \quad \alpha, \beta \in \mathbb{C}_+.
\end{align}
The other is multiplicativity, i.e.,
\begin{align}\label{e:multiplicativity of fractional powers}
( A^\alpha )^\beta = A^{ \alpha \beta }, \quad 0 < \alpha < \pi / \omega, \beta \in \mathbb{C}_+, 
\end{align}
where $\omega \in ( 0, \pi )$ is the sectoriality angle of $A$ (see \cite[Theorems 5.1.11 and 5.4.3]{MartinezM2001}).

The negative or imaginary powers of non-negative operators 
can be defined analogously. If this is the case, injectivity of non-negative operators is needed to ensure the single-valuedness of the desired fractional powers. More precisely, assume that the non-negative operator $A$ is injective. Thanks to (\ref{e:additivity of fractional powers}), it is easy to verify the injectivity of $A^\alpha$ for each $\alpha \in \mathbb{C}_+$ (also see \cite[Corollary 5.2.4 (ii)]{MartinezM2001}), and therefore one can define the negative power $A^{ - \alpha }$ of $A$ as the inverse of $A^\alpha$, i.e., 
\begin{align*}
A^{ - \alpha } := ( A^{ \alpha } )^{ - 1 }, \quad \alpha \in \mathbb{C}_+.
\end{align*}
Compared with positive or negative powers, the imaginary powers of non-negative operators have a slightly more complicated configuration and are the most mysterious objects in the field of fractional powers of operators, even the domains of them have not been fully understood as yet (see Remark \ref{r:mystery of imaginary powers} below).
In view of the importance of the closedness for unbounded operators on abstract spaces, one can define the imaginary powers of non-negative operators by using a closed extension of suitable operators. For example, one can consider the composition $A^{ - 1 } A^{ 1 + i t }$ and define 
the imaginary power $A^{ i t }$ of $A$ by its closed extension in the following way: 
\begin{align}\label{e:imaginary powers}
A^{ i t } := ( 1 + A )^2 A^{ - 1 } A^{ 1 + i t } ( 1 + A )^{ - 2 }, \quad t \in \mathbb{R}.
\end{align}
See \cite[Definition 7.1.2]{MartinezM2001}. 
Moreover, one may define $A^0 := I$, the identity operator, for the sake of convenience. 

Fractional (positive, negative or imaginary) powers of unbounded non-negative operators admit a general version of laws of exponents. On the one hand, in contrast to (\ref{e:additivity of fractional powers}), merely the following inclusion holds:
\begin{align}\label{e:additivity of fractional powers-complex}
	A^\alpha A^\beta \subset A^{ \alpha + \beta }, \quad \alpha, \beta \in \mathbb{C}.
\end{align}
On the other hand, analogous to (\ref{e:multiplicativity of fractional powers}), it holds that 
\begin{align*}
	( A^\alpha )^\beta = A^{ \alpha \beta }, \quad - \pi / \omega < \alpha < \pi / \omega, \beta \in \mathbb{C},
\end{align*}
where $\omega \in ( 0, \pi )$ is the sectoriality angle of $A$.
See \cite[Theorems 7.1.1 and 7.1.3]{MartinezM2001} or \cite[Proposition 3.2.1]{HaaseM2006}). 

In addition to the laws of exponents, the spectral mapping theorem
also plays an important role in the theory of fractional powers of operators. The spectral mapping theorem for fractional powers of non-negative operators reads that 
\begin{align*}
\sigma ( A^\alpha ) = \{ \mu^\alpha: \mu \in \sigma ( A ) \}, \quad \alpha \in \mathbb{C}_+.
\end{align*}
See \cite[Proposition 3.1.1 (j)]{HaaseM2006} or \cite[Theorem 5.3.1]{MartinezM2001}. As for  general spectral mapping theorems for functional calculi, we refer the reader to \cite[Section 4.3]{MartinezM2001} (for the Hirsch functional calculus) and \cite[Section 2.7]{HaaseM2006} (for the natural functional calculus).

Next we turn to representations of fractional powers of non-negative operators.   
Let $A$ be non-negative on $X$ and let $\alpha \in \mathbb{C}_+$ with $\RE \alpha < n \in \mathbb{N}$. 
Thanks to additivity of fractional powers of bounded non-negative operators, by using a standard approximation argument one can conclude from (\ref{e:integral-Balakrishnan}) that 
\begin{align}\label{e:representation of fractional powers-n}
A^{ \alpha } x = \frac{ \Gamma ( n ) }{ \Gamma ( \alpha ) \Gamma ( n - \alpha ) } \int_0^\infty \lambda^{ \alpha } A^n ( \lambda + A )^{ - n } x \, \frac{ d \lambda }{ \lambda }, \quad x \in D ( A^m ),
\end{align}
where $m$ is the smallest integer greater than $\RE \alpha$. Also, see \cite[Page 59, (3.4)]{MartinezM2001}. 
In particular, if $A$ is injective, replacing $A$ by $A^{ - 1 }$ in (\ref{e:representation of fractional powers-n}) yields 
\begin{align}\label{e:representation of fractional powers--n} 
A^{ - \alpha } x = \frac{ \Gamma ( n ) }{ \Gamma ( \alpha ) \Gamma ( n - \alpha ) } \int_0^\infty \lambda^{ - \alpha } \lambda^n ( \lambda + A )^{ - n } x \, \frac{ d \lambda }{ \lambda }, \quad x \in R ( A^m ).
\end{align}

Thanks to (\ref{e:representation of fractional powers-n}), 
we can now give some estimates of fractional powers of non-negative operators.
These estimates, especially
(\ref{e:resolvent estimate-discrete}) and (\ref{e:resolvent estimate-discrete-Charles}) below, will be used frequently to get explicit quasi-norm estimates for abstract Besov spaces. See Sections \ref{S:Besov spaces}, \ref{S:Basic properties} and \ref{S:Fractional powers} below. 

\begin{lem}\label{l:uniform boundedness compositions}
Let $A$ be non-negative and let $ \{ A_t \}_{ t \in \Lambda } $ be a family of non-negative operators $A_t$ on $X$. Fix $\alpha \in \mathbb{C}_+$ with $\RE \alpha < n \in \mathbb{N}$. The following statements hold. 
\begin{itemize}
\item[(i)] If $ \{ A_t \}_{ t \in \Lambda } $ is uniformly non-negative and uniformly bounded, then $ \{ A_t^\alpha \}_{ t \in \Lambda } $ is uniformly bounded and 
\begin{align}\label{e:boundedness of fractional powers-uniformly bounded families} 
\sup_{ t \in \Lambda } \left\| A^{ \alpha }_t \right\| \leq C_{ \alpha, n } ( L + K M )^n,  
\end{align}
where $M$ and $L$ are the uniform non-negativity constants of $\{ A_t \}_{ t \in \Lambda }$ given in (\ref{e:uniformly non-negative constant}) and (\ref{e:uniformly non-negative constant 2}), respectively,  $K = \sup_{ t \in \Lambda } \| A_t \|$ and
\begin{align}\label{e:uniform boundedness compositions}
C_{ \alpha, n } = \frac{ \Gamma ( \RE \alpha ) \Gamma ( n - \RE \alpha ) }{ | \Gamma ( \alpha ) \Gamma ( n - \alpha ) | }.
\end{align}
In particular, applying $A_t \equiv A \in \mathcal{L} ( X )$ to (\ref{e:boundedness of fractional powers-uniformly bounded families}) yields 
\begin{align}\label{e:boundedness of fractional powers-bounded operators} 
\| A^\alpha \| \leq C_{ \alpha, n } \left( L_A + M_A \| A \| \right)^n, 
\end{align}
applying $A_t = t ( t + A )^{ - 1 }$ with $t > 0$ to (\ref{e:boundedness of fractional powers-uniformly bounded families}) yields 
\begin{align}\label{e:uniform boundedness compositions-M}
\sup_{ t > 0 } \big\| t^\alpha ( t + A )^{ - \alpha } \big\| \leq C_{ \alpha, n } M_A^n \left( L_A + M_A + 1 \right)^n,
\end{align}
applying $A_t = A ( t + A )^{ - 1 }$ with $ t > 0 $ to (\ref{e:boundedness of fractional powers-uniformly bounded families}) yields 
\begin{align}\label{e:uniform boundedness compositions-L}
\sup_{ t > 0 } \big\| A^\alpha ( t + A )^{ - \alpha } \big\| \leq C_{ \alpha, n } L_A^n \left( L_A + M_A + 1 \right)^n,
\end{align} 
and applying $A_{ s, t } = ( s + A ) ( t + A )^{ - 1 }$ with $0 \le s / t \le c$ for given $c > 0$ to (\ref{e:boundedness of fractional powers-uniformly bounded families}) yields 
\begin{align}\label{e:uniform boundedness compositions-C}
\sup_{ 0 \leq s / t \leq c } \big\| ( s + A )^\alpha ( t + A )^{ - \alpha } \big\| \leq C_{ \alpha, n } \left( L_A + M_A \right)^n \left( L_A + c M_A + 1 \right)^n, 
\end{align}
where $M_A$ and $L_A$ are the non-negativity constants of $A$ given in (\ref{e:non-negative constant}) and (\ref{e:non-negative constant 2}), respectively.
\item[(ii)] Let $\beta \in \mathbb{C}_+$ with $\beta = \alpha$ or $\RE \alpha < \RE \beta < m \in \mathbb{N}$. For given $c > 0$, we have 
\begin{align}\label{e:resolvent estimate-discrete}
\big\| A^\alpha ( t + A )^{ - \beta } \big\| \leq C \, \big\| A^\alpha ( s + A )^{ - \beta } \big\|, \quad 0 \leq s / t \leq c,
\end{align}
and we also have 
\begin{align}\label{e:resolvent estimate-discrete-Charles}
\big\| ( s + A )^\alpha ( t + A )^{ - \beta } \big\| \leq C \, \big\| ( t + A )^{ - ( \beta - \alpha ) } \big\|, \quad 0 \leq s / t \leq c,
\end{align}
where $C = C_{ \beta, m } ( L_A + M_A )^{ m } ( L_A + c M_A + 1 )^m $ with $C_{ \beta, m }$, $M_A$ and $L_A$ given by (\ref{e:uniform boundedness compositions}), (\ref{e:non-negative constant}) and (\ref{e:non-negative constant 2}), respectively.
\end{itemize}
\end{lem}

\begin{proof}
The statements can be verified by use of (\ref{e:representation of fractional powers-n}) and 
(\ref{e:Euler integral}).
Indeed, from (\ref{e:representation of fractional powers-n}) and (\ref{e:Euler integral}), it follows that   
	\begin{align*}
		\| A_t^\alpha \| & \leq C_{\alpha, n } \sup_{ \lambda > 0 } \left\| ( 1 + \lambda )^n A_t^n ( \lambda + A_t )^{ - n } \right\| \\
		& \leq C_{\alpha, n } \sup_{ \lambda > 0 } \left\| ( 1 + \lambda ) A_t ( \lambda + A_t )^{ - 1 } \right\|^n \leq C_{ \alpha, n } ( L + K M )^n,
	\end{align*}
from which (\ref{e:boundedness of fractional powers-uniformly bounded families}) follows immediately. 
Note that (\ref{e:boundedness of fractional powers-bounded operators}) is a direct consequence of 
(\ref{e:boundedness of fractional powers-uniformly bounded families}), 
while (\ref{e:uniform boundedness compositions-M}), (\ref{e:uniform boundedness compositions-L}) and (\ref{e:uniform boundedness compositions-C}) are all simple consequences of (\ref{e:boundedness of fractional powers-uniformly bounded families}) and Example \ref{E:uniform non-negativity and boundedness}, where the uniform non-negativity constants of the corresponding families of operators are used. Moreover, (\ref{e:resolvent estimate-discrete}) and (\ref{e:resolvent estimate-discrete-Charles}) are two direct consequences of (\ref{e:uniform boundedness compositions-C}). The proof is complete.
\end{proof}

\begin{rem}
It is necessary to point out that sometimes we may obtain more precise estimates for a variety of  families of operators by a routine calculation starting from (\ref{e:representation of fractional powers-n}) and (\ref{e:Euler integral})
rather than the direct use of (\ref{e:boundedness of fractional powers-uniformly bounded families}). Indeed, the estimates given in (\ref{e:boundedness of fractional powers-bounded operators}), (\ref{e:uniform boundedness compositions-M}), (\ref{e:uniform boundedness compositions-L}) and (\ref{e:uniform boundedness compositions-C}) can be improved in the following way. 
\noindent\begin{itemize}
\item[(i)] Analogous to (\ref{e:boundedness of fractional powers-uniformly bounded families}), by use of (\ref{e:representation of fractional powers-n}) and (\ref{e:Euler integral}) one can verify that  
\begin{itemize}
\item[(\ref{e:uniform boundedness compositions-M})$^* \quad \quad $] $\sup_{ t > 0 } \big\| t^\alpha ( t + A )^{ - \alpha } \big\| \leq C_{ \alpha, n } M_A^n,$
\item[(\ref{e:uniform boundedness compositions-L})$^* \quad \quad $] $\sup_{ t > 0 } \big\| A^\alpha ( t + A )^{ - \alpha } \big\| \leq C_{ \alpha, n } L_A^n,$
\item[(\ref{e:uniform boundedness compositions-C})$^* \quad \quad $] $\sup_{ 0 \leq s / t \leq c } \big\| ( s + A )^\alpha ( t + A )^{ - \alpha } \big\| \leq C_{ \alpha, n } \big( L_A + \max\{ c, 1 \} \cdot M_A \big)^n,$ 
\end{itemize}
where $0 < \RE \alpha < n \in \mathbb{N}$, $C_{ \alpha, n }$ is given by (\ref{e:uniform boundedness compositions}), and $M_A$ and $L_A$ are the non-negativity constants of $ A $ given by (\ref{e:non-negative constant}) and (\ref{e:non-negative constant 2}), respectively. 
Clearly, the estimates (\ref{e:uniform boundedness compositions-M})$^*$, (\ref{e:uniform boundedness compositions-L})$^*$ and (\ref{e:uniform boundedness compositions-C})$^*$ are more precise than (\ref{e:uniform boundedness compositions-M}), (\ref{e:uniform boundedness compositions-L}) and (\ref{e:uniform boundedness compositions-C}), respectively. 
Moreover, (\ref{e:uniform boundedness compositions-M})$^*$ and (\ref{e:uniform boundedness compositions-L})$^*$ are also given in \cite[Corollary 3.1.13]{HaaseM2006} in the case $0 < \alpha < n \in \mathbb{ N }$.
\item[(ii)] 
Compared with (\ref{e:boundedness of fractional powers-bounded operators}), an alternative estimate of $\left\| A^\alpha x \right\|$ can be given by the so-called moment inequality \cite[Lemma 3.1.7]{MartinezM2001}, i.e.,
\begin{align}\label{e:monent inequality-Martinez}
\| A^\alpha x \| \le C ( \alpha, n, M_A ) \left\| A^n x \right\|^{ \RE \alpha / n } \left\| x \right\|^{ ( n - \RE \alpha ) / n }, \quad x \in D ( A^\alpha ), 
\end{align}
where $0 < \RE \alpha < n \in \mathbb{N}$, $M_A$ is the non-negativity constant of $A$ given by (\ref{e:non-negative constant}) and 
\begin{align*}
C ( \alpha, n, M_A ) = \frac{ \Gamma ( n +1 ) } { \left| \Gamma ( \alpha ) \Gamma ( n - \alpha ) \right| } \frac{ M_A^{ \RE \alpha } ( M_A + 1 )^{ n - \RE \alpha } }{ \RE \alpha ( n - \RE \alpha ) }.
\end{align*}
In particular, if $A \in L ( X )$, from (\ref{e:monent inequality-Martinez}) it follows that 
\begin{itemize}
\item[(\ref{e:boundedness of fractional powers-bounded operators})$^* \quad \quad $] $\| A^\alpha \| \leq C ( \alpha, n, M_A ) \| A \|^{ \RE \alpha },$
\end{itemize}
an estimate different from (\ref{e:boundedness of fractional powers-bounded operators}). Also, see \cite[Remark 5.1.2
]{MartinezM2001}.
\end{itemize}
\end{rem}

From this we turn to ergodicity of non-negative operators. In view of (\ref{e:representation of fractional powers-n}), a link between fractional powers and integral powers, 
it is necessary to point out two simple facts for bounded linear operators. More precisely, for $S, T \in \mathcal{L} ( X )$ satisfying commutativity, 
it is routine to verify that    
\begin{align}\label{e:a+b}
( S + T )^n = \sum_{ k = 0 }^{ n } \binom{ n }{ k } S^k T^{ n - k }, \quad n \in \mathbb{N}, 
\end{align}
and that  
\begin{align}\label{e:1-a}
S^n - T^n = ( S - T ) \sum_{ k = 0 }^{ n - 1 } T^k S^{ n - 1 - k }, \quad n \in \mathbb{N}.   
\end{align}

The following lemma improves slightly \cite[Theorem 6.1.1]{MartinezM2001} and \cite[Lemma 2.3]{Matsumoto and Ogawa2010}, which can be verified by using a standard density argument and some operator identities, especially 
(\ref{e:a+b}) and (\ref{e:1-a}) above. 

\begin{lem}\label{l:ergodicity}
	Let $A$ be a non-negative operator on $X$ and let $\alpha \in \mathbb{C}_+$. The following statements hold.
	\begin{itemize}
		\item[(i)] $\overline{D ( A )} = \overline{D ( A^\alpha )}$. Moreover, for $x \in X$,
		\begin{align*}
			x \in \overline{ D ( A^\alpha ) } & \Leftrightarrow \lim\limits_{ t
				\rightarrow \infty }{ t^\alpha ( t + A )^{ - \alpha } x } = x 
			\Leftrightarrow \lim\limits_{ t \rightarrow \infty }{ t^\alpha ( t + A
				)^{ - \alpha } x } \mbox{ exists } \\
			& \Leftrightarrow \lim\limits_{ t \rightarrow \infty }{ A^\alpha ( t +
				A )^{ - \alpha } x } = 0 \Leftrightarrow \lim\limits_{ t \rightarrow
				\infty }{ A^\alpha ( t + A )^{ - \alpha } x } \mbox{ exists} \\
           & \Leftrightarrow A^\alpha ( \lambda + A )^{ - \alpha } x \in \overline{ D ( A^\alpha ) } \mbox{ for all (or, some) } \lambda > 0.
		\end{align*}
		\item[(ii)] $\overline{R ( A )} = \overline{R ( A^\alpha )}$. Moreover, for $x \in X$,
		\begin{align*}
			x \in \overline{R ( A^\alpha )} & \Leftrightarrow \lim\limits_{ t
				\rightarrow 0 }{ A^\alpha ( t + A )^{ - \alpha } x } = x \Leftrightarrow
			\lim\limits_{ t \rightarrow 0 }{ t^\alpha ( t + A )^{ - \alpha } x } = 0 \\
           & \Leftrightarrow \lambda^\alpha ( \lambda + A )^{ - \alpha } x \in \overline{ R ( A^\alpha ) } \mbox{ for all (or, some) } \lambda > 0.
		\end{align*}
		\item[(iii)] $Ker ( A ) = Ker ( A^\alpha ) = Ker ( A^\alpha ( t + A )^{ - \alpha } )$ for $ t > 0 $. Moreover, for $x \in X$,
		\begin{align*}
			x \in Ker ( A^\alpha ) & \Leftrightarrow A^\alpha ( t + A )^{ - \alpha } x \equiv 0 \Leftrightarrow
			\lim\limits_{ t \rightarrow 0 }{ A^\alpha ( t + A )^{ - \alpha } x } = 0 \\
			& \Leftrightarrow t^\alpha ( t + A )^{ - \alpha } x \equiv x \Leftrightarrow
			\lim\limits_{ t \rightarrow 0 }{ t^\alpha ( t + A )^{ - \alpha } x } = x.
		\end{align*}
		\item[(iv)] For $x \in X$, we have 
		\begin{align*}
			x \in \overline{R ( A^\alpha ) } \oplus Ker ( A^\alpha ) & \Leftrightarrow 
			\lim\limits_{ t \rightarrow 0 }{ A^\alpha ( t + A )^{ - \alpha } x } \mbox{ exists } \\
			& \Leftrightarrow \lim\limits_{ t \rightarrow 0 }{ t^\alpha ( t + A )^{ - \alpha } x } \mbox{ exists}.
		\end{align*}
		Moreover, $\lim\limits_{ t \rightarrow 0 }{ t^\alpha ( t + A )^{ - \alpha } x } = x_0$ and $\lim\limits_{ t \rightarrow 0 }{ A^\alpha ( t + A )^{ - \alpha } x } = x_1$ whenever $x = x_0 + x_1$ with $x_0 \in Ker ( A^\alpha )$ and $x_1 \in \overline{R ( A^\alpha ) }$.
	\end{itemize}
\end{lem}

\begin{proof}
(i) Thanks to \cite[Theorem 6.1.1 (ii)]{MartinezM2001}, it remains to verify that
\begin{align}\label{e:ergodicity-proof-Ltinfty} 
\lim_{ t \rightarrow \infty } A^\alpha ( t + A )^{ - \alpha } x \mbox{ exists } \Rightarrow \lim_{ t \rightarrow \infty } A^\alpha ( t + A )^{ - \alpha } x = 0,
\end{align}
\begin{align}\label{e:ergodicity-proof-Mtinfty} 
\lim_{ t \rightarrow \infty } t^\alpha ( t + A )^{ - \alpha } x \mbox{ exists } \Rightarrow \lim_{ t \rightarrow \infty } t^\alpha ( t + A )^{ - \alpha } x = x
\end{align}
and 
\begin{align}\label{e:ergodicity-proof-Lbeta} 
\begin{split}
x \in \overline{ D ( A^\alpha ) } \Leftrightarrow {} & A^\alpha ( \lambda + A )^{ - \alpha } x \in \overline{ D ( A^\alpha ) } \mbox{ for all (or, some) } \lambda > 0.
\end{split}
\end{align}

Indeed,
(\ref{e:ergodicity-proof-Ltinfty}) is a simple consequence of the uniform boundedness of $\{ t^\alpha ( t + A )^{ - \alpha } \}_{ t > 0 }$ (see Lemma \ref{l:uniform boundedness compositions} (i) above) and the closedness of $A^\alpha$. 

Now we verify (\ref{e:ergodicity-proof-Mtinfty}).
Let $x \in X$ such that the limit $\lim\limits_{ t \rightarrow \infty } t^\alpha ( t + A )^{ - \alpha } x$ exists. From (\ref{e:representation of fractional powers-n}) and (\ref{e:1-a}) it follows that 
\begin{align*} 
	x - t^\alpha ( t + A )^{ - \alpha } x = A x_t, \quad t > 0,
\end{align*}
where 
\begin{align*} 
x_t = \frac{ \Gamma ( n ) }{ \Gamma ( \alpha ) \Gamma ( n - \alpha ) } \sum_{ k = 0 }^{ n - 1 } \int_{ 0 }^\infty \frac{ \lambda^\alpha }{ ( 1 + \lambda )^n } \bigg[ \frac{ t ( 1 + \lambda ) }{ \lambda } \bigg]^k \bigg[ \frac{ t ( 1 + \lambda ) }{ \lambda } + A \bigg]^{ - ( k + 1 ) } x \, \frac{ d \lambda }{ \lambda }  
\end{align*}
with $\RE \alpha < n \in \mathbb{N}$. Thanks to the dominated convergence theorem, from the uniform boundedness of $\{ t ( t + A )^{ - 1 } \}_{ t > 0 }$ (see, Example \ref{E:uniform non-negativity and boundedness} (b) above) it follows that $x_t \rightarrow 0$ as $t \rightarrow \infty$, so that $A x_t \rightarrow 0$ as $t \rightarrow \infty$ due to the closedness of $A$, i.e., $t^\alpha ( t + A )^{ - \alpha } x \rightarrow x$ as $t \rightarrow \infty$. Thus, we have verified (\ref{e:ergodicity-proof-Mtinfty}).

As for (\ref{e:ergodicity-proof-Lbeta}), it suffices to verify that 
\begin{align}\label{e:ergodicity-proof-Lbeta-1}
x \in \overline{ D ( A ) } \Leftrightarrow {} & A^\alpha ( \lambda + A )^{ - \alpha } x \in \overline{ D ( A ) } \mbox{ for all (or, some) } \lambda > 0 
\end{align}
due to the fact that $\overline{ D ( A^\alpha ) } = \overline{ D ( A ) }$. To this end, let $x \in X$ and fix an integer $n > \RE \alpha$. From (\ref{e:representation of fractional powers-n}) and (\ref{e:a+b}) it follows that 
\begin{align*} 
A^\alpha ( \lambda + A )^{ - \alpha } x - x = \sum_{ k = 1 }^{ n } \int_0^\infty \frac{ \mu^\alpha }{ ( 1 + \mu )^n } ( - 1 )^k \bigg(  \frac{ \lambda \mu }{ 1 + \mu } \bigg)^k \bigg( \frac{ \lambda \mu }{ 1 + \mu } + A \bigg)^{ - k } x \, \frac{ d \mu }{ \mu },
\end{align*}
where $c = \frac{ \Gamma ( n ) }{ \Gamma ( \alpha ) \Gamma ( n - \alpha ) }$. By \cite[Proposition 1.1.7]{ABHN2001} we conclude that the integrals given in the right-hand side of the last equality are all in $\overline{ D ( A ) }$, so that $A^\alpha ( \lambda + A )^{ - \alpha } x - x \in \overline{ D ( A ) }$. Thus, we have verified (\ref{e:ergodicity-proof-Lbeta-1}).

(ii) Analogous to (\ref{e:ergodicity-proof-Lbeta}), we can verify that 
\begin{align*} 
x \in \overline{ R ( A^\alpha ) } \Leftrightarrow \lambda^\alpha ( \lambda + A )^{ - \alpha } x \in \overline{ R ( A^\alpha ) } \mbox{ for all (or, some) } \lambda > 0.
\end{align*}
The other statements have been verified in \cite[Theorem 6.1.1 (iii)]{MartinezM2001}.

(iii) Recall that
\begin{align}\label{e:kernel-n}
Ker ( A ) = Ker ( A^n ), \quad n \in \mathbb{ N }, 
\end{align}
as shown in \cite[Corollary 1.1.6]{MartinezM2001}. A direct proof of (\ref{e:kernel-n}) can be given as follows. Indeed, it is trivial that $Ker ( A ) \subset Ker ( A^n )$ for each $n \in \mathbb{ N }$. Conversely, it suffices to verify that $Ker ( A^2 ) \subset Ker ( A )$. Let $x \in Ker ( A^2 )$. By using $x = ( t + A ) ( t + A )^{ - 1 } x$ with $t > 0$ we have 
\begin{align*}
A x = t A ( t + A )^{ - 1 } x + ( t + A )^{ - 1 } A^2 x = t A ( t + A )^{ - 1 } x.
\end{align*}
Applying $t \rightarrow 0$ in the last equality yields $A x = 0$ since the last term converges to zero as $t \rightarrow 0$ due to the uniform boundedness of $\{ A ( t + A )^{ - 1 } \}_{ t > 0 }$. Thus, we have verified (\ref{e:kernel-n}).

Now we verify the first equality of (iii), i.e.,  
\begin{align}\label{e:kernel-alpha}
Ker ( A ) = Ker ( A^\alpha ).
\end{align}
To this end, fix an integer $n > \RE \alpha$. It is clear that $Ker ( A^\alpha ) \subset Ker ( A^n ) = Ker ( A )$ due to (\ref{e:additivity of fractional powers}) and (\ref{e:kernel-n}). Conversely, let $x \in Ker ( A )$. It is clear that $x \in Ker ( A^n )$ due to (\ref{e:kernel-n}), and hence, $A^\alpha x = 0$ due to the representation (\ref{e:representation of fractional powers-n}). This implies that $x \in Ker ( A^\alpha )$, and hence, $Ker ( A ) \subset Ker ( A^\alpha )$. Thus, we have verified (\ref{e:kernel-alpha}).

Next we verify the second equality of (iii), i.e., 
\begin{align}\label{e:kernel-alpha-compositions}
Ker ( A^\alpha ) = Ker \left( A^\alpha ( t + A )^{ - \alpha } \right), \quad t > 0.
\end{align}
Indeed, by using (\ref{e:identity decomposition}) and the commutativity between $A$ and $( t + A )^{ - 1 }$, we have $Ker ( A ( t + A )^{ - 1 } ) = Ker ( A )$ for each $t > 0$. Applying (\ref{e:kernel-alpha}) to the operator $A ( t + A )^{ - 1 }$ yields (\ref{e:kernel-alpha-compositions}), immediately.


And finally, we verify the desired equivalences of (iii). It is trivial that 
\begin{align*}
	x \in Ker ( A^\alpha ) \Rightarrow A^\alpha ( t + A )^{ - \alpha } x \equiv 0 \Rightarrow \lim\limits_{ t \rightarrow 0 }{ A^\alpha ( t + A )^{ - \alpha } x } = 0.  
\end{align*}
Now fix $x \in X$ such that $\lim\limits_{ t \rightarrow 0 }{ A^\alpha ( t + A )^{ - \alpha } x } = 0$. Note that  
\begin{align*}
	x & \leftarrow x - A^\alpha ( t + A )^{ - \alpha } x \\
	& = \frac{ \Gamma ( n ) }{ \Gamma ( \alpha ) \Gamma ( n - \alpha ) } \int_0^\infty \frac{ \lambda^\alpha }{ ( 1 + \lambda )^n } \frac{ \lambda t }{ 1 + \lambda } \sum_{ k = 0 }^{ n - 1 } A^k \bigg( \frac{ \lambda t }{ 1 + \lambda } + A \bigg)^{ - ( k + 1 ) } x \frac{ d \lambda }{ \lambda } := y_t,
\end{align*}
where the equality follows from (\ref{e:representation of fractional powers-n}) and (\ref{e:1-a}), while $A y_t \rightarrow 0$ as $t \rightarrow 0^+$ by the dominated convergence theorem. By use of the closedness of $A$ we conclude that $x \in D ( A )$ and $A x = 0$, i.e., $x \in Ker ( A )$. Applying (\ref{e:kernel-alpha}) yields $x \in Ker ( A ) = Ker ( A^\alpha )$. Thus, we have verified that 
\begin{align*}
	x \in Ker ( A^\alpha ) \Leftrightarrow A^\alpha ( t + A )^{ - \alpha } x \equiv 0 \Leftrightarrow \lim\limits_{ t \rightarrow 0 }{ A^\alpha ( t + A )^{ - \alpha } x } = 0.  
\end{align*}
By using analogous argunents, we can also verify the equivalence 
\begin{align*}
	x \in Ker ( A^\alpha ) \Leftrightarrow t^\alpha ( t + A )^{ - \alpha } x \equiv x \Leftrightarrow \lim\limits_{ t \rightarrow 0 }{ t^\alpha ( t + A )^{ - \alpha } x } = x.  
\end{align*}
Thus, we have verified the statement (iii).

(iv) We first verify the equivalence 
\begin{align}\label{e:independence of limits-L}
x \in \overline{ R ( A^\alpha ) } \oplus Ker ( A^\alpha ) \Leftrightarrow \lim_{ t \rightarrow 0^+ } A^\alpha ( t + A )^{ - \alpha } x \mbox{ exists}.
\end{align}
Fix $x \in \overline{R ( A^\alpha ) } \oplus Ker ( A^\alpha )$, and write $x = x_1 + x_2$ with $x_1 \in \overline{R ( A^\alpha ) }$ and $x_2 \in Ker ( A^\alpha )$. Thanks to the statement (ii), there exists the limit $\lim\limits_{ t \rightarrow 0 }{ A^\alpha ( t + A )^{ - \alpha } x }$. More precisely, 
\begin{align*}
\lim\limits_{ t \rightarrow 0 }{ A^\alpha ( t + A )^{ - \alpha } x } = \lim\limits_{ t \rightarrow 0 }{ A^\alpha ( t + A )^{ - \alpha } x_1 } = x_1.
\end{align*}
Conversely, fix $x \in X$ such that the limit $\lim\limits_{ t \rightarrow 0 }{ A^\alpha ( t + A )^{ - \alpha } x }$ exists. Write 
\begin{align*}
y := \lim\limits_{ t \rightarrow 0 }{ A^\alpha ( t + A )^{ - \alpha } x }.
\end{align*}
It suffices to verify $x - y \in Ker ( A^\alpha )$ due to the fact that $y \in \overline{R ( A^\alpha ) }$. To this end, write $x_t := A^\alpha ( t + A )^{ - \alpha } x$ with $t > 0$. From (\ref{e:representation of fractional powers-n}) and (\ref{e:a+b}) it follows that 
\begin{align*}
x - x_t 
= C_\alpha \sum_{ k = 1 }^{ n } ( - 1 )^{ k + 1 } \binom{ n }{ k } \int_0^\infty \frac{ \lambda^\alpha }{ ( 1 + \lambda )^n } \bigg( \frac{ \lambda t }{ 1 + \lambda } \bigg)^k \bigg(  \frac{ \lambda t }{ 1 + \lambda } + A \bigg)^{ - k } x \, \frac{ d \lambda }{ \lambda },  
\end{align*}
where $C_\alpha = \frac{ \Gamma ( n ) }{ \Gamma ( \alpha ) \Gamma ( n - \alpha ) }$. By using 
(\ref{e:non-negative constant}) and (\ref{e:non-negative constant 2}) we conclude that $x - x_t \in D ( A )$ and $A ( x -  x_t ) \rightarrow 0$ as $t \rightarrow 0$.
Applying the closedness of $A$ yields $x - y \in Ker ( A )$, and hence, $x - y \in Ker ( A^\alpha )$ due to (\ref{e:kernel-alpha}). Thus, we have verified (\ref{e:independence of limits-L}). 

The equivalence
\begin{align*}
x \in \overline{ R ( A^\alpha ) } \oplus Ker ( A^\alpha ) \Leftrightarrow \lim_{ t \rightarrow 0^+ } t^\alpha ( t + A )^{ - \alpha } x \mbox{ exists}.
\end{align*}
can be verified analogously.
Moreover, thanks to the statements (ii) and (iii), 
it can be seen that $\lim\limits_{ t \rightarrow 0 }{ t^\alpha ( t + A )^{ - \alpha } x } = x_0$ and $\lim\limits_{ t \rightarrow 0 }{ A^\alpha ( t + A )^{ - \alpha } x } = x_1$ whenever $x = x_0 + x_1$ with $x_0 \in Ker ( A^\alpha )$ and $x_1 \in \overline{ R ( A^\alpha ) }$. The proof is complete. 
\end{proof}

In the following lemma, we present some reproducing formulas for the fractional powers of non-negative operators, which will be used to construct Besov spaces associated with operators in Section \ref{S:Besov spaces} below. 

\begin{lem}\label{l:reproducing formulas}
Let $A$ be non-negative on $X$, and let $\alpha, \beta \in \mathbb{C}_+^*$ and $\lambda, \mu > 0$. The following statements hold.
\begin{itemize}
\item[(i)] If $\RE \alpha > 0$ and $\beta = m \in \mathbb{N}$ then, for $x \in \overline{D ( A )}$,
\begin{align}\label{e:inhomogeneous Calderon}
\begin{split}
x = {} & \frac{ \Gamma ( \alpha + m ) }{ \Gamma ( \alpha ) \Gamma ( m ) }
\int_\lambda^\infty t^\alpha A^m ( t + A )^{ - \alpha - m }
x \, \frac{ d t }{ t } \\
+ {} & \sum_{ k = 0 }^{ m - 1 } \frac{ \Gamma ( \alpha + k ) }{ \Gamma ( \alpha ) \Gamma ( k + 1 ) } \big[ A ( \lambda + A )^{ - 1 } \big]^k \lambda^\alpha
( \lambda + A )^{ - \alpha } x. 
\end{split}
\end{align}
\item[(ii)] For $x \in X$,
\begin{equation}\label{e:transform resolvent}
\begin{split}
\lambda^{ \alpha + \beta } A^\beta ( \lambda + A )^{ - \alpha -
\beta } x = {} & ( \alpha + \beta ) \int_\mu^\lambda t^{ \alpha + \beta
} A^{ \beta + 1 } ( t + A )^{ - \alpha - \beta - 1 } x \, \frac{ d t }{ t } \\
+ {} & \mu^{ \alpha + \beta } A^\beta ( \mu + A )^{ - \alpha - \beta }
x.
\end{split}
\end{equation}
\item[(iii)] If $\RE \beta > 0$, then
\begin{equation}\label{e:transform resolvent 2}
\begin{split}
\lambda^{ \alpha + \beta } A^\beta ( \lambda + A )^{ - \alpha -
\beta } x = ( \alpha + \beta ) \int_0^\lambda t^{ \alpha + \beta }
A^{ \beta + 1 } ( t + A )^{ - \alpha - \beta - 1 } x \, \frac{ d t
}{ t }
\end{split}
\end{equation}
for $x \in X$; and if $\beta = 0$, then
(\ref{e:transform resolvent 2}) holds for $x \in \overline{R (
A )}$.
\item[(iv)] For $x \in X$,
\begin{equation}\label{e:transform resolvent 3-0}
\begin{split}
A^\beta ( \lambda + A )^{ - \alpha - \beta } x = ( \alpha + \beta )
\int_\lambda^\mu A^\beta ( t + A )^{ - \alpha - \beta - 1 } x \, d t + A^\beta ( \mu + A )^{ - \alpha - \beta } x.
\end{split}
\end{equation}
\item[(v)] If $\RE \alpha > 0$, then
\begin{align}\label{e:transform resolvent 3}
A^\beta ( \lambda + A )^{ - \alpha - \beta } x = ( \alpha + \beta )
\int_\lambda^\infty A^\beta ( t + A )^{ - \alpha - \beta - 1 } x \, d t
\end{align}
for $x \in X$; and if $\alpha = 0 < \RE \beta$, then
(\ref{e:transform resolvent 3}) holds for $x \in \overline{D (
A )}$.
\end{itemize}
\end{lem}

\begin{proof}
By induction, it is routine to verify that  
\begin{align*}
& \frac{ d }{ d t } \Bigg\{ \sum_{ k = 0 }^{ m - 1 } \frac{ \Gamma ( \alpha + k ) }{ \Gamma ( \alpha ) \Gamma ( k + 1 ) } \big[ A ( t + A )^{ - 1 } \big]^k t^\alpha ( t + A )^{ - \alpha } x \Bigg\} \\
= {} & \frac{ \Gamma ( \alpha + m ) }{ \Gamma ( \alpha ) \Gamma ( m ) } t^{ \alpha - 1 } A^m ( t + A )^{ - \alpha - m } x.
\end{align*}
Since $t^\alpha A^m ( t + A )^{ - \alpha - m } x \rightarrow 0$ as $t \rightarrow \infty$ due to Lemma \ref{l:ergodicity} (i), integrating the last equality yields (\ref{e:inhomogeneous Calderon}), immediately.
The equalities (\ref{e:transform resolvent}) and (\ref{e:transform resolvent 3-0}) are direct consequences of the fundamental theorem of calculus. Applying $\mu \rightarrow 0$ to (\ref{e:transform resolvent}) yields (\ref{e:transform resolvent 2}), while applying $\mu \rightarrow \infty$ to (\ref{e:transform resolvent 3-0}) yields (\ref{e:transform resolvent 3}), immediately. The proof is complete.
\end{proof}

\begin{rem}\label{r:reproducing formulas}
The (inhomogeneous) reproducing formula (\ref{e:inhomogeneous Calderon}) above is essentially due to \cite[Lemma 2.1]{Kobayashi and Muramatu1992}, where (\ref{e:inhomogeneous Calderon}) is given for special values $\alpha = k \in \mathbb{N}$. Applying $\lambda \to 0$ to (\ref{e:inhomogeneous Calderon}) yields the following (homogeneous) reproducing formula:
\begin{align*}
x = \frac{ \Gamma ( \alpha + m ) }{ \Gamma ( \alpha ) \Gamma ( m ) } \int_0^\infty t^\alpha A^m ( t + A )^{ - \alpha - m } x \, \frac{ d t }{ t }, \quad x \in \overline{ D ( A ) } \cap \overline{ R ( A ) },
\end{align*}
where $\alpha \in \mathbb{C}_+$ and $m \in \mathbb{N}$. Also see (\ref{e:reproducing formula-ab}) below for an alternative version of the reproducing formula for fractional powers of operators.

\end{rem}


\smallskip


\subsection{Representations of fractional powers}


Given an operator $A$ on a Banach space $X$, a fundamental matter in the theory of fractional powers of operators is to provide explicit representations for the fracitonal power $A^\alpha$ and give exact description for the fractional domain $D ( A^\alpha )$. As mentioned above, the use of 
(\ref{e:representation of fractional powers-n}), i.e., a vector-valued version of the Euler integral (\ref{e:Euler integral}), for the fractional powers of operators goes back to A. V. Balakrishnan \cite{Balakrishnan1960}, while H. Komatsu \cite[Theorem 2.10]{Komatsu1967}  characterized the fractional domain $D ( A^\alpha )$ for the case in which $\overline{ D ( A ) } = X$ and $\alpha \in \mathbb{C}_+$ with $\RE \alpha \not\in \mathbb{N}$. 

Let $A$ be non-negative on $X$ and let $\alpha \in \mathbb{C}_+$ with $\RE \alpha < n \in \mathbb{N}$. Recall that (\ref{e:representation of fractional powers-n}) admits a more general version. More precisely, on the one hand, the regularity of $x$ in (\ref{e:representation of fractional powers-n}) can be weaken from $x \in D ( A^m )$ to $x \in X$ such that there exists the limit $\lim_{ N \rightarrow \infty } \int_0^{ N } \lambda^{ \alpha } A^n ( \lambda +
A )^{ - n } x \, \frac{ d \lambda }{ \lambda }$ (see \cite[Theorem 2.10]{Komatsu1967}).
On the other hand, (\ref{e:representation of fractional powers-n}) can be extended from integers $n > \RE \alpha$ to complex numbers $\beta$ satisfying $\RE \beta > \RE \alpha$, as shown in \cite[Corollary 5.1.13, (5.17)]{MartinezM2001}.
For the sake of the reader's convenience, we reformulate
\cite[Corollary 5.1.13]{MartinezM2001} as follows.


\begin{lem}\label{l:representation of fractional powers-beta-weak convergence}
Let $A$ be non-negative on $X$ and let $\alpha, \beta \in \mathbb{C}_+$ with $\RE \alpha < \RE \beta$. Fix $x \in X$ such that there exists the 
limit
\begin{align}\label{e:characterization of domains-KomatsuChen-weak limit}
\lim_{ N \rightarrow \infty } \frac{ \Gamma ( \beta ) }{ \Gamma ( \alpha )
\Gamma ( \beta - \alpha ) } \int_0^{ N } \lambda^{ \alpha } A^\beta ( \lambda +
A )^{ - \beta } x \, \frac{ d \lambda }{ \lambda } := y.
\end{align}
Then $x \in D ( A^\alpha )$ and $A^\alpha x = y$.
\end{lem}

In particular, the limit (\ref{e:characterization of domains-KomatsuChen-weak limit}) exists for $x \in X$ satisfying 
\begin{align}\label{e:representation of fractional powers-beta-absolute convergence} 
\int_0^\infty \big\| \lambda^\alpha A^\beta ( \lambda + A )^{ - \beta } x \big\| \frac{ d \lambda }{ \lambda } < \infty.
\end{align}
Thus, we have the following consequence of Lemma \ref{l:representation of fractional powers-beta-weak convergence}.


\begin{cor}\label{c:representation of fractional powers-beta-absolute convergence}
Let $A$ be non-negative on $X$ and let $\alpha, \beta \in \mathbb{C}_+$ with $\RE \alpha < \RE \beta$. Fix $x \in X$ such that (\ref{e:representation of fractional powers-beta-absolute convergence}) holds.
Then $x \in D ( A^\alpha )$ and
\begin{align}\label{e:representation of fractional powers-beta-domain of A-alpha+e} 
A^\alpha x = \frac{ \Gamma ( \beta ) }{ \Gamma ( \alpha ) \Gamma ( \beta - \alpha ) } \int_0^\infty  \lambda^\alpha A^\beta ( \lambda + A )^{ - \beta } x \, \frac{ d \lambda }{ \lambda }.
\end{align}
\end{cor}

By use of (\ref{e:non-negative constant}) and (\ref{e:non-negative constant 2}), it is easy to verify that (\ref{e:representation of fractional powers-beta-absolute convergence}) is satisfied whenever $x \in D ( A^{ \alpha + \epsilon} )$. Thus, we have the following simple consequence of Corollary \ref{c:representation of fractional powers-beta-absolute convergence}. Also see \cite[Lemma 3.1]{ChenLi2016}.

\begin{cor}\label{c:representation of fractional powers-beta-domain of A-alpha+e}
Let $A$ be non-negative on $X$, and let $\alpha, \beta \in \mathbb{C}_+$ with $\RE \alpha < \RE \beta$. Then (\ref{e:representation of fractional powers-beta-domain of A-alpha+e}) holds for $x \in D ( A^{ \alpha + \epsilon} )$ with $\epsilon > 0$. 
\end{cor}


Compared with Lemma \ref{l:representation of fractional powers-beta-weak convergence} above, the following lemma reads that the limit (\ref{e:characterization of domains-KomatsuChen-weak limit}) exists whenever $x \in D ( A^\alpha )$ with $\alpha \in \mathbb{C}_+$, possibly except for the case in which $\RE \alpha \in \mathbb{N}$ and $\IM \alpha \neq 0$, as shown in
\cite[Theorem 3.5, (i)$\Rightarrow$(iii)]{ChenLi2016}. 

\begin{lem}\label{l:characterization of domains-KomatsuChen}
Let $A$ be non-negative on $X$ and let $\alpha, \beta \in \mathbb{C}_+$ with $\RE \alpha < \RE \beta$. If $x \in D ( A^\alpha )$, then there exists the limit (\ref{e:characterization of domains-KomatsuChen-weak limit})
and $A^\alpha x = y$, possibly except for the case in which $\RE \alpha \in \mathbb{N}$ and $\IM \alpha \neq 0$. 
\end{lem}

\begin{rem}\label{r:mystery of imaginary powers}
Here we say a few more words on Lemma \ref{l:characterization of domains-KomatsuChen}. Recall that, as for the three statements (i), (ii) and (iii) given in \cite[Theorem 3.5]{ChenLi2016}, the equivalence (i)$\Leftrightarrow$(ii)
and the implication (iii)$\Rightarrow$(i) both hold,
while the implication (i)$\Rightarrow$(iii) was merely verified for the case in which $\RE \alpha \not\in \mathbb{N}$ or $\alpha = n \in \mathbb{N}$ (since, in Step II of the proof given in \cite[Theorem 3.5]{ChenLi2016}, the use of integration by parts is possible merely in the case $\RE \alpha = 1$ with $\IM \alpha = 0$, i.e., $\alpha = 1$). Thus, whether Lemma \ref{l:characterization of domains-KomatsuChen} above holds in the case $\alpha = n + i t$ with $n \in \mathbb{N}$ and $ 0 \neq t \in \mathbb{R}$ is still an {\bf open problem}. We also refer the reader to \cite[Theorem 2.10]{Komatsu1967} and \cite[Theorem 6.1.3 and Remark 6.1.1]{MartinezM2001} for more information on this topic.
\end{rem}

Analogous to (\ref{e:representation of fractional powers--n}), if $A$ is injective, replacing $A$ by $A^{ - 1 }$ in (\ref{e:representation of fractional powers-beta-domain of A-alpha+e}) yields
\begin{align}\label{e:representation of negative powers-beta}
A^{ - \alpha } x = \frac{ \Gamma ( \beta ) }{ \Gamma ( \alpha )
\Gamma ( \beta - \alpha ) } \int_0^\infty \lambda^{ - \alpha } \lambda^\beta ( \lambda +
A )^{ - \beta } x \, \frac{ d \lambda }{ \lambda }, \quad x \in R ( A^{ \alpha + \epsilon } ).
\end{align}
In particular, (\ref{e:representation of negative powers-beta}) holds for each $x \in X$ if $A$ is positive.

The next result
is new, which states that both (\ref{e:representation of fractional powers-beta-domain of A-alpha+e}) of positive powers $A^\alpha$ and (\ref{e:representation of negative powers-beta}) of negative powers $A^{ - \alpha }$ can be reformulated as a unified one.

\begin{prop}\label{p:representation of fractional powers-all}
Let $A$ be non-negative on $X$ and let $z \in \mathbb{C}$ and $\alpha, \beta \in \mathbb{C}_+^*$ such that $- \RE \alpha < \RE z < \RE \beta$. The following statements hold. 
\begin{itemize}
\item [(i)] If $\RE z > 0$, then (\ref{e:representation of fractional powers-all}) holds for $x \in D ( A^{ z + \epsilon } )$ with $\epsilon > 0$. 
\item [(ii)] If $\RE z < 0$, then (\ref{e:representation of fractional powers-all}) holds for $x \in R ( A^{ - z + \epsilon } )$ with $\epsilon > 0$ whenever $A$ is injective.
\item [(iii)] If $\RE z = 0$ then (\ref{e:representation of fractional powers-all}) holds for $x \in R ( A^{ - z + \epsilon } ) \cap D ( A^{ z + \epsilon } )$ with $\epsilon > 0$ whenever $A$ is injective.
\end{itemize} 
	
\end{prop}

\begin{proof}
(i) Let $\RE z > 0$ and fix $x \in D ( A^{ z + \epsilon } )$ with $0 < \epsilon < \RE \beta - \RE z$. 
First, the integral given in the right-hand side of (\ref{e:representation of fractional powers-all}) absolutely converges. This can be verified by using (\ref{e:uniform boundedness compositions-M}), (\ref{e:uniform boundedness compositions-L}) and the fact that $\RE \beta > \RE z + \epsilon$. More precisely,     
\begin{align*} 
\int_0^\infty \big\| \lambda^{ z + \alpha } A^{ \beta } ( \lambda + A )^{ - \alpha - \beta } x \big\| \, \frac{ d \lambda }{ \lambda } \lesssim \| x \| + \| A^{ z + \epsilon } x \|. 
\end{align*}

Next we verify (\ref{e:representation of fractional powers-all}). To this end, write $A_t := A + t$ with $t > 0$, and fix $\gamma \in \mathbb{C}_+$ such that $\RE \gamma > \RE \alpha + \RE \beta$. Applying (\ref{e:representation of negative powers-beta}) to the operator $\lambda + A_t$ yields 
\begin{align*} 
& \int_0^\infty \lambda^{ z + \alpha } A_t^{ \beta } ( \lambda + A_t )^{ - \alpha - \beta } x \, \frac{ d \lambda }{ \lambda } \\
= {} & C \int_0^\infty \lambda^{ z + \alpha } A_t^{ \beta } \int_{ 0 }^\infty \mu^{ \gamma - \alpha - \beta } \big ( \mu + \lambda + A_t \big )^{ - \gamma } x \, \frac{ d \mu }{ \mu } \, \frac{ d \lambda }{ \lambda },  
\end{align*}
where $C = \frac{ \Gamma ( \gamma ) }{ \Gamma ( \alpha + \beta ) \Gamma ( \gamma - \alpha - \beta ) }$. Applying the change of variable $\sigma = \lambda + \mu$ and exchanging the order of integration yields 
\begin{align*} 
\int_0^\infty \lambda^{ z + \alpha } A_t^{ \beta } ( \lambda + A_t )^{ - \alpha - \beta } x \, \frac{ d \lambda }{ \lambda } 
= C \int_0^\infty \sigma^{ z - \beta } \sigma^\gamma A_t^{ \beta } ( \sigma + A_t )^{ - \gamma } x \, \frac{ d \sigma }{ \sigma },
\end{align*}
where $C = \frac{ \Gamma ( \gamma ) \Gamma ( z + \alpha ) }{ \Gamma ( \alpha + \beta ) \Gamma ( z - \beta + \gamma ) }$.
Thanks to the closedness of $A_t^\beta$, 
by \cite[Proposition 1.1.7]{ABHN2001} we conclude that 
\begin{align*} 
\int_0^\infty \lambda^{ z + \alpha } A_t^{ \beta } ( \lambda + A_t )^{ - \alpha - \beta } x \, \frac{ d \lambda }{ \lambda } 
= C A_t^{ \beta } \int_0^\infty \sigma^{ z - \beta } \sigma^\gamma ( \sigma + A_t )^{ - \gamma } x \, \frac{ d \sigma }{ \sigma }, 
\end{align*}
where $C = \frac{ \Gamma ( \gamma ) \Gamma ( z + \alpha ) }{ \Gamma ( \alpha + \beta ) \Gamma ( z - \beta + \gamma ) }$. Since $A_t$ is positive, from (\ref{e:representation of negative powers-beta}) it follows that 
\begin{align*} 
\int_0^\infty \lambda^{ z + \alpha } A_t^{ \beta } ( \lambda + A_t )^{ - \alpha - \beta } x \, \frac{ d \lambda }{ \lambda } 
= C A_t^{ \beta } A_t^{ z - \beta } = C A_t^z,
\end{align*}
where $C = \frac{ \Gamma ( z + \alpha ) \Gamma ( \beta - z ) }{ \Gamma ( \alpha + \beta ) }$. 
Thus, we have verified that 
\begin{align}\label{e:representation of fractional powers-all-proof1} 
A_t^{ z } x = \frac{ \Gamma ( \alpha + \beta ) }{ \Gamma ( \alpha + z ) \Gamma ( \beta - z ) } \int_0^\infty \lambda^{ z + \alpha } A_t^{ \beta } ( \lambda + A_t )^{ - \alpha - \beta } x \, \frac{ d \lambda }{ \lambda }.
\end{align}
Note that $A^z x = \lim_{ t \rightarrow 0 } A_t^z x$ due to the fact that $x \in D ( A^{ z + \epsilon } ) \subset D ( A^z )$. Thanks to the dominated convergence theorem, applying $t \rightarrow 0$ to (\ref{e:representation of fractional powers-all-proof1}) yields (\ref{e:representation of fractional powers-all}), immediately. 

(ii) The statement is a simple consequence of (i). More precisely, let $\RE z < 0$ and fix $x \in R ( A^{ - z + \epsilon } )$ with $0 < \epsilon < \RE z + \RE \alpha$. Note that $- \RE \beta < \RE ( - z ) < \RE \alpha$ and $R ( A^{ - z + \epsilon } ) = D ( ( A^{ - 1 } )^{ - z + \epsilon } )$. 
Applying (i) to the $(-z)$-power of $A^{ - 1 }$ yields  
\begin{align*}  
A^{ z } x & = \frac{ \Gamma ( \alpha + \beta ) }{ \Gamma ( \beta - z ) \Gamma ( \alpha + z ) } \int_0^\infty \mu^{ - z + \beta } ( A^{ - 1 } )^{ \alpha } ( \mu + A^{ - 1 } )^{ - \beta - \alpha } x \, \frac{ d \mu }{ \mu } \\
& = \frac{ \Gamma ( \alpha + \beta ) }{ \Gamma ( \alpha + z ) \Gamma ( \beta - z ) } \int_0^\infty \lambda^{ z + \alpha } A^{ \beta } ( \lambda + A )^{ - \alpha - \beta } x \, \frac{ d \lambda }{ \lambda }.
\end{align*}
Thus, we have verified (\ref{e:representation of fractional powers-all}) for $\RE z < 0$.

(iii) Let $z = it$ with $t \in \mathbb{R}$ and fix $x \in D ( A^{ i t + \epsilon } ) \cap R ( A^{ i t + \epsilon } )$ with $0 < \epsilon < \min\{ 1, \RE \alpha, \RE \beta \}$. Thanks to (\ref{e:non-negative constant}) and (\ref{e:non-negative constant 2}), it is easy to verify the absolute convergence of the integral given in the right-hand side of (\ref{e:representation of fractional powers-all}) due to the fact that $0 < \epsilon < \min\{ \RE \alpha, \RE \beta \}$. By using additivity, we obtain from (\ref{e:imaginary powers}) that 
\begin{align*}  
A^{ i t } x = ( 1 + A )^2 A^{ - \epsilon } A^{ \epsilon + i t } ( 1 + A )^{ - 2 } x.
\end{align*}
Applying (i) to the positive power $A^{ \epsilon + i t }$ of $A$ with $- \RE ( \alpha - \epsilon ) < \RE ( \epsilon + i t ) < \RE ( \epsilon + \beta )$ yields 
\begin{align*}  
A^{ i t } x = C ( 1 + A )^2 A^{ - \epsilon } \int_0^\infty \lambda^{ \epsilon + i t } \lambda^{ \alpha - \epsilon } A^{ \beta + \epsilon } ( \lambda + A )^{ - \alpha - \beta } ( 1 + A )^{ - 2 } x \, \frac{ d \lambda }{ \lambda },
\end{align*}
where $C = \frac{ \Gamma ( \alpha + \beta ) }{ \Gamma ( \alpha + i t ) \Gamma ( \beta - i t ) }$.
Thanks to the fact that $0 < \epsilon < \min\{ 1, \RE \alpha, \RE \beta \}$ and the closedness of $( 1 + A )^2 A^{ - \epsilon }$,
applying \cite[Proposition 1.1.7]{ABHN2001} to the operator $( 1 + A )^2 A^{ - \epsilon }$ yields (\ref{e:representation of fractional powers-all}), immediately.
The proof is complete.
\end{proof}

\begin{rem}\label{r:homogeneous Calderon-regularity}
By Propostion \ref{p:representation of fractional powers-all} (iii), specifying $z = 0$ to (\ref{e:representation of fractional powers-all}) yields 
\begin{align}\label{e:reproducing formula-ab}  
x = \frac{ \Gamma ( \alpha + \beta ) }{ \Gamma ( \alpha ) \Gamma ( \beta ) } \int_0^\infty \lambda^{ \alpha } A^{ \beta } ( \lambda + A )^{ - \alpha - \beta } x \, \frac{ d \lambda }{ \lambda }, \quad x \in D ( A^\epsilon ) \cap R ( A^\epsilon ),
\end{align}
for $\epsilon > 0$. This (homogeneous) reproducing formula will be used to establish interpolations for abstract Besov spaces in Section \ref{S:Fractional powers} below.
\end{rem}


Furthermore, we present some representations for resolvents of fractional powers in the following lemma. In particular, the representation (\ref{e:fractional powers-resolvent-L}) below will be used to establish smoothness reiteration for abstract Besov spaces in Section \ref{S:Fractional powers}. 

\begin{lem}\label{l:representations-powers-semigroups} Let $A$ be non-negative on $X$ and let $0 < \alpha < 1$. 
Then $A^\alpha$ is non-negative with the non-negativity constant $M_{ A^\alpha } \leq M_A$. More precisely, for $\lambda > 0$,
\begin{align}\label{e:fractional powers-resolvent}
( \lambda + A^\alpha )^{ - 1 } = C \int_0^\infty \frac{ \mu^{ \alpha + 1 } }{ \lambda^2 + 2 \lambda \mu^\alpha \cos \pi \alpha + \mu^{ 2 \alpha } } ( \mu + A )^{ - 1 } \, \frac{ d \mu }{ \mu }
\end{align}
and 
\begin{align}\label{e:fractional powers-resolvent-L}
A^\alpha ( \lambda + A^\alpha )^{ - 1 } = C \int_0^\infty \frac{ \lambda \mu^{ \alpha } }{ \lambda^2 + 2 \lambda \mu^\alpha \cos \pi \alpha + \mu^{ 2 \alpha } } A ( \mu + A )^{ - 1 } \, \frac{ d \mu }{ \mu },
\end{align}
where $C = \frac{ 1 }{ \Gamma ( \alpha ) \Gamma ( 1 - \alpha ) }$ and $M_A$ is the non-negativity constant of $A$ given by (\ref{e:non-negative constant}).
\end{lem}

\begin{proof}
It is routine to verify (\ref{e:fractional powers-resolvent}) by using the vector-valued version of the Cauchy integral formula (see, for example, \cite[Proposition 10.2]{Komatsu1966}, \cite[Proposition 5.3.2(5.24)]{MartinezM2001} or \cite[Remark 3.1.16]{HaaseM2006}). Applying  (\ref{e:non-negative constant}) and (\ref{e:fractional powers-resolvent-scalar}) to (\ref{e:fractional powers-resolvent}) gives $M_{ A^\alpha } \le M_A$, immediately. Moreover, from (\ref{e:identity decomposition}) and (\ref{e:fractional powers-resolvent}) it follows that 
\begin{align*} 
A^\alpha ( \lambda + A^\alpha )^{ - 1 } = {} & I - C \int_0^\infty \frac{ \lambda \mu^{ \alpha } }{ \lambda^2 + 2 \lambda \mu^\alpha \cos \pi \alpha + \mu^{ 2 \alpha } } \mu ( \mu + A )^{ - 1 } \, \frac{ d \mu }{ \mu } \\
= {} & C \int_0^\infty \frac{ \lambda \mu^{ \alpha } }{ \lambda^2 + 2 \lambda \mu^\alpha \cos \pi \alpha + \mu^{ 2 \alpha } } A ( \mu + A )^{ - 1 } \, \frac{ d \mu }{ \mu }, \quad \lambda > 0,
\end{align*}
the desired (\ref{e:fractional powers-resolvent-L}). The proof is complete.
\end{proof}

Finally, in contrast to the fractional powers of non-negative operators on Banach spaces, a complete account of the theory of fractional powers of operators is beyond of the scope of our work, and we refer the reader to, for instance,  \cite{Fattorini1983,
Tarasov2007,Han2020} (for fractional powers associated with the fractional calculus), \cite{deLaubenfels1993,van Neerven1998,Periago2002,deLaubenfels2008,Kostic2011,Wang2012} (for fractional powers of operators with polynomially bounded resolvent), \cite{Martinez2005} (for fractional powers of almost non-negative operators), 
\cite[Section 5.7]{MartinezM2001} (for fractional powers of non-negative (multivalued) operators in Fr\'{e}chet spaces), and \cite{Auscher2008,Bui2017,Huang2018,Shen2018,Li2019} for some applications associated with the fractional Laplacians, fractional Schr\"{o}dinger operators and some other fractional operators.

\smallskip


\section{Besov spaces associated with operators}\label{S:Besov spaces}

This section is devoted to the construction of Besov spaces associated with non-negative operators in the framework of Banach spaces. In this section, $( X, \left\| \cdot \right\| )$ is a Banach space and $A$ is a non-negative operator (not necessarily defined densely or with dense range) on $X$.

\smallskip


\subsection{Subspaces of Besov type}


One of the advantages of the classical Besov spaces
is that they provide continuous scales for the smoothness of function spaces. From the operator-theoretic point of view, the classical Besov spaces $B^s_{ p, q } ( \mathbb{R}^n )$ admit a close relation with
the well known Laplacian on $L^p ( \mathbb{ R }^n )$ (see Examples \ref{E:Gaussian semigroup} and \ref{E:Poisson semigroup} below or \cite[Section 2.12.2]{TriebelM1983-2010}). Compared with (\ref{e:characterization of domains-KomatsuChen-weak limit}), (\ref{e:representation of fractional powers-beta-absolute convergence}) provides a quantitative estimate, i.e., $\lambda^\alpha A^\beta ( \lambda + A )^{ - \beta } x \in L^1 ( ( 0, \infty ); d \lambda / \lambda )$, for the fractional domain $D ( A^\alpha )$. The quantitative estimate (\ref{e:representation of fractional powers-beta-absolute convergence}) can also be refined via an additional size index $1 \le q \le \infty$, i.e., $\lambda^\alpha A^\beta ( \lambda + A )^{ - \beta } x \in L^q ( ( 0, \infty ); d \lambda / \lambda )$ (see (\ref{d:Komatsu spaces}) for $\beta = n \in \mathbb{N}$).
This is the main idea of H. Komatsu \cite{Komatsu1967} on abstract Besov spaces from the interpolation point of view. 

In order to give unified quantitative estimates of fractional domains for a full range $s \in \mathbb{R}$ and $0 < q \le \infty$, we start from the unified representation (\ref{e:representation of fractional powers-all}) and turn to the following dyadic series instead of the Lebesgue integrals.  

\begin{defn}\label{d:Besov type-inhomogeneous}
Let $0 < q \leq \infty$ and $s \in \mathbb{R}$. Fix $k \in \mathbb{Z}$ and $\alpha, \beta \in \mathbb{C}_+^*$ such that $- \RE \alpha < s < \RE \beta$. The space $R^{ s, A }_{ q, X } ( k, \alpha, \beta )$ is defined by
\begin{align*}
R^{ s, A }_{ q, X } ( k, \alpha, \beta ) := \left\{ x \in \overline{ D ( A ) }: \sum_{ i = k
}^\infty \big\| 2^{ i ( s + \alpha ) } A^{ \beta } ( 2^i + A )^{ - \alpha - \beta } x \big\|^q < \infty \right\},
\end{align*}
endowed with the semi-quasinorm
\begin{align*}
\left| x \right|_{ R^{ s, A }_{ q, X } ( k, \alpha, \beta ) } := \left\{ \sum_{ i = k
}^\infty \big\| 2^{ i ( s + \alpha ) } A^{ \beta } ( 2^i + A )^{ - \alpha
- \beta } x \big\|^q \right\}^{ 1 / q }
\end{align*}
(with the usual modification when $q = \infty$).
\end{defn}

\begin{rem}\label{r:quasi-norm of inhomogeneous Besov space-injective}
It is necessary to point out that the subspace $R^{ s, A }_{ q, X } ( k, \alpha, \beta )$ is not trivial and that the semi-quasinorm $\left| \cdot \right|_{ R^{ s, A }_{ q, X } ( k, \alpha, \beta ) }$ is indeed a quasi-norm in some special cases. More precisely, 
\begin{itemize}
\item[(i)] It is clear that $D ( A^\beta ) \subset R^{ s, A }_{ q, X } ( k, \alpha, \beta )$ if $s \geq 0$ and that $R^{ s, A }_{ q, X } ( k, \alpha, \beta ) = \overline{ D (A) }$ if $s < 0$ due to the uniform boundedness of $\{ 2^{ j \alpha } ( 2^j + A )^{ - \alpha } \}_{ j \geq k }$ and $\{ A^\beta ( 2^j + A )^{ - \beta } \}_{ j \geq k }$ as shown in Lemma \ref{l:uniform boundedness compositions} (ii), so that the space $R^{ s, A }_{ q, X } ( k, \alpha, \beta ) $ is non-zero for a non-trivial operator $A$. 
\item[(ii)] The semi-quasinorm $\left| \cdot \right|_{ R^{ s, A }_{ q, X } ( k, \alpha, \beta ) }$ is indeed a quasi-norm if $\beta = 0$ or $A$ is injective. Indeed, the case when $\beta = 0$ is trivial while the case when $A$ is injective can be seen by observing that 
\begin{align*}
	| x |_{ R^{ s, A }_{ q, X } ( k, \alpha, \beta ) } = 0 \Leftrightarrow A^{ \beta } ( 2^j + A )^{ - \beta } x = 0
\Leftrightarrow x \in Ker ( A ),
\end{align*}
where the last equivalence follows from Lemma \ref{l:ergodicity} (iii) above. 
\end{itemize}
\end{rem}


It is easy to verify that the space $R^{ s, A }_{ q, X } ( k, \alpha, \beta ) $ is independent of the choice of $k$ in the sense of equivalent semi-quasinorms.

\begin{lem}\label{l:independence of R-inhomogeneous-k}
Let $0 < q \leq \infty$ and $s \in \mathbb{R}$. Fix $k \in	\mathbb{Z}$ and $\alpha, \beta \in \mathbb{C}_+^*$ such that $- \RE \alpha < s < \RE \beta$. Then $R^{ s, A }_{ q, X } ( k, \alpha, \beta ) = R^{ s, A }_{ q, X } ( k', \alpha, \beta )$ from the set-theoretic point of view and the equivalence 
\begin{align}\label{e:independence of R-inhomogeneous k}
| x |_{ R^{ s, A }_{ q, X } ( k, \alpha, \beta ) } \simeq | x |_{ R^{ s, A }_{ q, X } ( k', \alpha, \beta ) }
\end{align}
holds for each $k' \in \mathbb{Z}$.
\end{lem}

\begin{proof}
It suffices to verify that 	
\begin{align*} 
| x |_{ R^{ s, A }_{ q, X } ( k, \alpha, \beta ) } \simeq | x |_{ R^{ s, A }_{ q, X } ( k + 1, \alpha, \beta ) }.  
\end{align*}
Indeed, it is trivial that   
\begin{align*} 
	| x |_{ R^{ s, A }_{ q, X } ( k + 1, \alpha, \beta ) } \leq | x |_{ R^{ s, A }_{ q, X } ( k, \alpha, \beta ) },  
\end{align*}
and hence, it remains to verify that
\begin{align}\label{e:independence of R-inhomogeneous-k-proof1}
| x |_{ R^{ s, A }_{ q, X } ( k, \alpha, \beta ) } \lesssim | x |_{ R^{ s, A }_{ q, X } ( k + 1, \alpha, \beta ) }.
\end{align}

If $0 < q < \infty$, applying (\ref{e: q-inequality}) yields 
\begin{align*} 
\begin{split}
| x |_{ R^{ s, A }_{ q, X } ( k, \alpha, \beta ) } \leq {} & C_{ 1 / q } \left[ | x |_{ R^{ s, A }_{ q, X } ( k + 1, \alpha, \beta ) } + \big\| 2^{ k ( s + \alpha ) } A^\beta ( 2^k + A )^{ - \alpha - \beta } x \big\| \right],
\end{split}
\end{align*}
while applying (\ref{e:uniform boundedness compositions-C}) with $c = 2$ to $( 2^{ k + 1 } + A )^{ \alpha + \beta } ( 2^k + A )^{ - ( \alpha + \beta ) }$ yields 
\begin{align}\label{e:independence of R-inhomogeneous-k-proof2} 
\begin{split}
\big\| 2^{ k ( s + \alpha ) } A^\beta ( 2^k + A )^{ - \alpha - \beta } x \big\| = {} & 2^{ - ( s + \RE \alpha ) } \big\| ( 2^{ k + 1 } + A )^{ \alpha + \beta } ( 2^k + A )^{ - ( \alpha + \beta ) }  \\
{} & \cdot 2^{ ( k + 1 ) ( s + \alpha ) } A^\beta ( 2^{ k + 1 } + A )^{ - ( \alpha + \beta ) }\big\| \\
\leq {} & C \, \big\| 2^{ ( k + 1 ) ( s + \alpha ) } A^\beta ( 2^{ k + 1 } + A )^{ - ( \alpha + \beta ) } \big\|, 
\end{split}
\end{align}
where $C = 2^{ - ( s + \RE \alpha ) } C_{ \alpha + \beta, N } ( L_A + 2 M_A )^N$ with $\RE ( \alpha + \beta ) < N \in \mathbb{ N }$ and $C_{ \alpha + \beta, N }$ given by (\ref{e:uniform boundedness compositions}).
This implies that 
\begin{align*} 
| x |_{ R^{ s, A }_{ q, X } ( k, \alpha, \beta ) } \leq {} & C_{ 1 / q } \Big[ | x |_{ R^{ s, A }_{ q, X } ( k + 1, \alpha, \beta ) } + C \big\| 2^{ ( k + 1 ) ( s + \alpha ) } A^\beta ( 2^{ k + 1 } + A )^{ - ( \alpha + \beta ) }\big\| \Big] \\
\leq {} & C_{ 1 / q } \max\{ 1, C \} \cdot | x |_{ R^{ s, A }_{ q, X } ( k + 1, \alpha, \beta ) },
\end{align*}
where $C_{ 1 / q }$ and $C$ are given in (\ref{e: q-inequality}) and (\ref{e:independence of R-inhomogeneous-k-proof2}), respectively. Thus, we have verified (\ref{e:independence of R-inhomogeneous-k-proof1}) for $0 < q < \infty$.

The case when $q = \infty$ can be verified analogously. More precisely, by using (\ref{e:independence of R-inhomogeneous-k-proof2}) again we have 
\begin{align*} 
| x |_{ R^{ s, A }_{ q, X } ( k, \alpha, \beta ) } & = \sup_{ j \geq k } \big\| 2^{ j ( s + \alpha ) } A^\beta ( 2^j + A )^{ - ( \alpha + \beta )} \big\| \\
& \leq \max\Big\{ \big\| 2^{ k ( s + \alpha ) } A^\beta ( 2^k + A )^{ - ( \alpha + \beta )} \big\|, | x |_{ R^{ s, A }_{ q, X } ( k + 1 , \alpha, \beta ) } \Big\} \\
& \leq C \, | x |_{ R^{ s, A }_{ q, X } ( k + 1 , \alpha, \beta ) },
\end{align*}
where $C$ is the constant given in (\ref{e:independence of R-inhomogeneous-k-proof2}).
Thus, we have verified (\ref{e:independence of R-inhomogeneous-k-proof1}).
The proof is complete.
\end{proof}

Thanks to Lemmas \ref{l:ergodicity} and \ref{l:reproducing formulas}, we can verify that the space $R^{ s, A }_{ q, X } ( k, \alpha, \beta ) $ is also independent of the choice of $\alpha$ in the sense of equivalent semi-quasinorms.

\begin{lem}\label{l:independence of lR-inhomogeneous}
Let $0 < q \leq \infty$ and $s \in \mathbb{R}$. Fix $k \in \mathbb{Z}$ and $\alpha, \beta \in \mathbb{C}_+^*$ such that $- \RE \alpha < s < \RE \beta$. Then
$R^{ s, A }_{ q, X } ( k, \alpha, \beta ) = R^{ s, A }_{ q, X } ( k, \alpha', \beta )$ from the set-theoretic point of view and 
\begin{align}\label{e:independence of R-inhomogeneous l}
| x |_{ R^{ s, A }_{ q, X } ( k, \alpha, \beta ) } \simeq | x |_{ R^{ s, A
	}_{ q, X } ( k, \alpha', \beta ) } 
\end{align}
for each $\alpha' \in \mathbb{C}_+^*$ satisfying $- \RE \alpha' < s$.
\end{lem}

\begin{proof}
It suffices to verify (\ref{e:independence of R-inhomogeneous l}) in the case $\RE \alpha \leq \RE \alpha'$, otherwise the conclusion follows from the symmetry of $\alpha$ and $\alpha'$. Indeed, it suffices to verify the following two inequalities. One is 
\begin{align}\label{e:independence of R-inhomogeneous l-proof-trivial} 
	| x |_{ R^{ s, A }_{ q, X } ( k, \alpha'', \beta ) } \lesssim | x |_{ R^{ s, A }_{ q, X } ( k, \alpha, \beta ) } 
\end{align}	
for each $\alpha'' \in \mathbb{C}_+^*$ satisfying $\RE \alpha'' > \RE \alpha$ and the other is 
\begin{align}\label{e:independence of R-inhomogeneous l-proof1}
| x |_{ R^{ s, A }_{ q, X } ( k, \alpha, \beta ) } \lesssim | x |_{ R^{ s, A }_{ q, X } ( k, \alpha + 1, \beta ) }.
\end{align}
If this is the case, from (\ref{e:independence of R-inhomogeneous l-proof1}) and (\ref{e:independence of R-inhomogeneous l-proof-trivial}) it follows that 
\begin{align*}
| x |_{ R^{ s, A }_{ q, X } ( k, \alpha, \beta ) } & \lesssim | x |_{ R^{ s, A }_{ q, X } ( k, \alpha + 1, \beta ) } \\
& \lesssim | x |_{ R^{ s, A }_{ q, X } ( k, \alpha', \beta ) } \\
& \lesssim | x |_{ R^{ s, A }_{ q, X } ( k, \alpha' + 1, \beta ) } \lesssim | x |_{ R^{ s, A }_{ q, X } ( k, \alpha, \beta ) } 
\end{align*}
whenever $\RE \alpha = \RE \alpha'$
and that 
\begin{align*}
| x |_{ R^{ s, A }_{ q, X } ( k, \alpha, \beta ) } \lesssim | x |_{ R^{ s, A }_{ q, X } ( k, \alpha + n, \beta ) } \lesssim | x |_{ R^{ s, A }_{ q, X } ( k, \alpha', \beta ) } 
 \lesssim | x |_{ R^{ s, A }_{ q, X } ( k, \alpha, \beta ) } 
\end{align*}
whenever $\RE \alpha < \RE \alpha'$, where $n \in \mathbb{N}$ with $\RE \alpha' < \RE \alpha + n$. Thus, we have verified (\ref{e:independence of R-inhomogeneous l}) by using (\ref{e:independence of R-inhomogeneous l-proof-trivial}) and (\ref{e:independence of R-inhomogeneous l-proof1}).

It remains to verify (\ref{e:independence of R-inhomogeneous l-proof-trivial}) and (\ref{e:independence of R-inhomogeneous l-proof1}). Indeed, (\ref{e:independence of R-inhomogeneous l-proof-trivial}) is a direct consequence of the uniform boundedness of the family $\{ 2^{ j ( \alpha'' - \alpha ) } ( 2^j + A )^{ - ( \alpha'' - \alpha ) } \}_{ j \geq k }$  due to (\ref{e:uniform boundedness compositions-M}) and it remains to verify (\ref{e:independence of R-inhomogeneous l-proof1}). We merely verify (\ref{e:independence of R-inhomogeneous l-proof1}) in the case $0 < q < \infty$ and the case when $q = \infty$ can be verified analogously.  
To this end, let $0 < q < \infty$. From (\ref{e:transform resolvent 3}) it follows that 
\begin{align*}
\begin{split}	 
	| x |_{ R^{ s, A }_{ q, X } ( k, \alpha, \beta ) } 
	& \leq \left| \alpha + \beta \right| \left\{ \sum_{ i = k }^\infty \left[ \int_{ 2^i }^{ \infty } \big\| 2^{ i ( s + \alpha ) } A^\beta ( t + A )^{ - \alpha - \beta - 1 } x \big\| \, d t \right]^q \right\}^{ 1 / q }.
\end{split}	
\end{align*}
Rewriting $\int_{ 2^i }^\infty = \sum_{ r = i }^\infty \int_{ 2^r }^{ 2^{ r + 1 } }$ and applying the estimate (\ref{e:resolvent estimate-discrete}) with $c = 1$ yields 
\begin{align}\label{e:independence of lR-inhomogeneous-proof-CI} 
| x |_{ R^{ s, A }_{ q, X } ( k, \alpha, \beta ) } 
\leq C J, 
\end{align}
where $C = \left| \alpha + \beta \right| C_{ \alpha + \beta + 1, N } ( M_A + L_A )^{ N }$ with $\RE \alpha + \RE \beta + 1 < N \in \mathbb{ N }$ and $C_{ \alpha + \beta + 1, N }$ given by (\ref{e:uniform boundedness compositions}) and 
\begin{align*} 
J = \left\{ \sum_{ i = k }^\infty \bigg[ \sum_{ r=i }^{ \infty } \big\| 2^{ i ( s + \alpha ) } 2^r A^\beta ( 2^r + A )^{ - \alpha - \beta - 1 } x \big\| \bigg]^q \right\}^{ 1 / q }.
\end{align*}
It can be verified that  
\begin{align}\label{e:independence of lR-inhomogeneous-proof-I<alpha+1}
	J \lesssim | x |_{ R^{ s, A }_{ q, X } ( k, \alpha + 1, \beta ) }.
\end{align}
Indeed, if $0 < q \leq 1$, applying the classical inequality (\ref{e: q-inequality-series<1}) and exchanging the order of summation yields 
\begin{align*} 
	J 
	& \leq \left\{ \sum_{ r = k }^\infty \sum_{ i=k }^{ r } \big\| 2^{ i ( s + \alpha ) }  2^{ r } A^\beta ( 2^r + A )^{ - \alpha - \beta - 1 } x \big\|^q \right\}^{ 1 / q } \\
	& \leq \left[ 1 - 2^{ - ( s + \RE \alpha ) q } \right]^{ - 1 / q } | x |_{ R^{ s, A }_{ q, X } ( k, \alpha + 1, \beta ) }.
\end{align*}
And if $1 < q < \infty$, rewriting $I$ as   
\begin{align*} 
J = \left\{ \sum_{ i = k }^\infty 2^{ r ( s + \RE \alpha ) q } \bigg[ \sum_{ r = i }^{ \infty } 2^{ r ( \delta - s - \RE \alpha ) } \big\| 2^{ r ( s + \alpha - \delta + 1 ) } A^\beta ( 2^r + A )^{ - \alpha - \beta - 1 } x \big\| \bigg]^q \right\}^{ 1 / q }  
\end{align*}
with an auxiliary constant $\delta \in (0, s + \RE \alpha )$ and applying the H\"{o}lder inequality to the inner bracket yields 
\begin{align*}
J \le {} & \frac{ 1 }{ [ 1 - 2^{ ( \delta - s - \RE \alpha ) q' } ]^{ 1 / q' } } \Bigg\{ \sum_{ i = k }^\infty \sum_{ r = i }^{ \infty } \big\| 2^{ i \delta } 2^{ r ( s + \alpha - \delta + 1 ) } A^\beta ( 2^r + A )^{ - \alpha - \beta - 1 } x \big\|^q \Bigg\}^{ 1 / q }.
\end{align*}
Exchanging the order of summation yields 
\begin{align*}
J \leq {} & \frac{ 1 }{ [ 1 - 2^{ ( \delta - s - \RE \alpha ) q' } ]^{ 1 / q' } } \left( \frac{ 2^{ \delta q } }{ 2^{ \delta q } - 1 } \right)^{ 1 / q } | x |_{ R^{ s, A }_{ q, X } ( k, \alpha + 1, \beta ) }.
\end{align*}
Thus, we have verified (\ref{e:independence of lR-inhomogeneous-proof-I<alpha+1}). Finally, from (\ref{e:independence of lR-inhomogeneous-proof-CI}) and (\ref{e:independence of lR-inhomogeneous-proof-I<alpha+1}) the desired inequality (\ref{e:independence of R-inhomogeneous l-proof1}) follows immediately. The proof is complete.
\end{proof}

For a relatively satisfactory independence of the space $R^{ s, A }_{ q, X } ( k, \alpha, \beta )$ with respect to $\beta$, however, an auxiliary term $\| ( 2^k + A )^{ - \alpha } x \|$ seems to be indispensable. The next result is crucial to the theory of inhomogeneous Besov spaces associated with non-negative operators.

\begin{lem}\label{l:independence of klm-inhomogeneous}
Let $0 < q \leq \infty$ and $s \in \mathbb{R}$, and let $k \in \mathbb{Z}$ and $\alpha, \beta \in \mathbb{C}_+^*$ such that $- \RE \alpha < s < \RE \beta$.
Then $R^{ s, A }_{ q, X } ( k, \alpha, \beta ) = R^{ s, A }_{ q, X } ( k', \alpha', \beta' )$ from the set-theoretic point of view and the equivalence
\begin{align}\label{e:independence of klm-inhomogeneous}
\left\| ( 2^k + A )^{ - \alpha } x \right\| + | x |_{ R^{ s, A }_{ q, X } ( k, \alpha, \beta ) } \simeq \big\| ( 2^{ k' } + A )^{ - \alpha' } x \big\| + | x |_{ R^{ s, A }_{ q, X } ( k', \alpha', \beta' ) }
\end{align}
holds for all $\alpha', \beta' \in \mathbb{C}_+^*$ satisfying $- \RE \alpha' < s < \RE \beta'$ and $k' \in \mathbb{Z}$.
\end{lem}

\begin{proof}
Fix $k' \in \mathbb{Z}$ and let $\alpha', \beta' \in \mathbb{C}_+^*$ such that $- \RE \alpha' < s < \RE \beta'$. It suffices to verify that 
\begin{align}\label{e:independence of klm-inhomogeneous k}
\big\| ( 2^k + A )^{ - \alpha } x \big\| + | x |_{ R^{ s, A }_{ q, X } ( k, \alpha, \beta ) } \simeq \big\| ( 2^{ k' } + A )^{ - \alpha } x \big\| + | x |_{ R^{ s, A }_{ q, X } ( k', \alpha, \beta ) },
\end{align}
\begin{align}\label{e:independence of klm-inhomogeneous m}
\big\| ( 2^k + A )^{ - \alpha } x \big\| + | x |_{ R^{ s, A }_{ q, X } ( k, \alpha, \beta) } \simeq \big\| ( 2^{ k } + A )^{ - \alpha } x \big\| + | x |_{ R^{ s, A }_{ q, X } ( k, \alpha, \beta' ) },
\end{align}
and
\begin{align}\label{e:independence of klm-inhomogeneous l}
\big\| ( 2^k + A )^{ - \alpha } x \big\| + | x |_{ R^{ s, A }_{ q, X } ( k, \alpha, \beta ) } \simeq \big\| ( 2^{ k } + A )^{ - \alpha' } x \big\| + | x |_{ R^{ s, A }_{ q, X } ( k, \alpha', \beta ) }.
\end{align}

First we verify (\ref{e:independence of klm-inhomogeneous k}). Thanks to (\ref{e:independence of R-inhomogeneous k}), it suffices to verify that 
\begin{align*}
\big\| ( 2^k + A )^{ - \alpha } x \big\| \simeq \big\| ( 2^{ k' } + A )^{ - \alpha } x \big\|.
\end{align*}
Indeed, note that $( 2^{ k' } + A ) ( 2^k + A )^{ - 1 }$ is bounded on $X$ due to (\ref{e:non-negative constant}) and (\ref{e:non-negative constant 2}), so is $( 2^{ k' } + A )^\alpha ( 2^k + A )^{ - \alpha }$ as the $\alpha$-power of $( 2^{ k' } + A ) ( 2^k + A )^{ - 1 }$ given by (\ref{e:integral-Balakrishnan}) above (also, see (\ref{e:uniform boundedness compositions-C}) above). This implies that
\begin{align*}
\big\| ( 2^k + A )^{ - \alpha } x \big\| 
\leq \big\| ( 2^{ k' } + A )^\alpha ( 2^k + A )^{ - \alpha } \big\| \cdot \big\| ( 2^{ k' } + A )^{ - \alpha } \big\| 
\lesssim \big\| ( 2^{ k' } + A )^{ - \alpha } x \big\|.  
\end{align*}
The inverse inequality can be verified analogously. 
Thus, we have verified (\ref{e:independence of klm-inhomogeneous k}). 

Next we verify (\ref{e:independence of klm-inhomogeneous m}). Analogous to (\ref{e:independence of R-inhomogeneous l-proof1}), it suffices to verify that
\begin{align*}
| x |_{ R^{ s, A }_{ q, X } ( k, \alpha, \beta' ) } \lesssim | x |_{ R^{ s, A }_{ q, X } ( k, \alpha, \beta ) } \lesssim \left\| ( 2^k + A )^{ - \alpha } x \right\| + | x |_{ R^{ s, A }_{ q, X } ( k, \alpha, \beta + 1 ) } 
\end{align*}
for $\RE \beta' > \RE \beta$. Indeed, the former is a direct consequence of (\ref{e:uniform boundedness compositions-L}). The latter can be verified by using Lemmas \ref{l:uniform boundedness compositions} and \ref{l:reproducing formulas}. We merely give a proof in the case $0 < q < \infty$ and the case when $q = \infty$ can be proved analogously. 
To this end, fix $x \in \overline{ D ( A ) }$ such that 
$| x |_{ R^{ s, A }_{ q, X, k, \alpha, \beta + 1 } } < \infty$. 
Applying (\ref{e:transform resolvent}) with $\lambda = 2^i$ and $\mu = 2^k$ yields
\begin{align*}
| x |_{ R^{ s, A }_{ q, X } ( k, \alpha, \beta ) } 
\leq {} & \left| \alpha + \beta \right| \Bigg\{ \sum_{ i = k }^\infty \bigg[ 2^{ i ( s - \RE \beta ) } \int_{ 2^k }^{ 2^i } \left\| t^{ \alpha + \beta } A^{ \beta + 1 } ( t + A )^{ - \alpha - \beta - 1 } x \right\| \frac{ d t }{ t } \\
+ {} & 2^{ i ( s - \RE \beta ) +  k ( \RE \alpha + \RE \beta ) } \big\| A^\beta ( 2^k + A )^{ - \alpha - \beta } x \big\| \bigg]^q \Bigg\}^{ 1 / q }.
\end{align*}	
Rewriting $\int_{ 2^k }^{ 2^i } = \sum_{ r = k }^{ i - 1 } \int_{ 2^r }^{ 2^{ r + 1 } }$ and applying (\ref{e: q-inequality}) and (\ref{e:resolvent estimate-discrete}) with $c = 1$ yields
\begin{align}\label{e:independence of klm-inhomogeneous-proof-J1+J2}
| x |_{ R^{ s, A }_{ q, X } ( k, \alpha, \beta ) } \lesssim J_1 + J_2,
\end{align}
%
where 
\begin{align*}
	J_1 = \left\{ \sum_{ i = k + 1 }^\infty \left[ \sum_{ r = k
	}^{ i - 1 } \big\| 2^{ i ( s - \beta ) } 2^{ r ( \alpha + \beta ) } A^{ \beta + 1 }
	( 2^r + A )^{ - \alpha - \beta - 1 } x \big\| \right]^q \right\}^{ 1 / q }
\end{align*}
and 
\begin{align*}
	J_2 = \left\{ \sum_{ i = k }^\infty 2^{ [ i ( s - \RE \beta ) + k ( \RE \alpha + \RE \beta ) ] q } \left\| A^\beta ( 2^k + A )^{ - \alpha - \beta } x \right\|^q \right\}^{ 1 / q }. 
\end{align*}
Analogous to (\ref{e:independence of lR-inhomogeneous-proof-I<alpha+1}), by using the classical inequality (\ref{e: q-inequality-series<1}) for $0 < q \leq 1$ and the H\"{o}lder inequality for $1 < q < \infty$ we have   
\begin{align*}
	J_1 \lesssim | x |_{ R^{ s, A }_{ q, X } ( k, \alpha, \beta + 1 ) }.
\end{align*}
By (\ref{e:uniform boundedness compositions-L}) we also have 
\begin{align*}
	J_2 \lesssim \left\| ( 2^k + A )^{ - \alpha } x \right\|  
\end{align*}
due to the fact that $s - \RE \beta < 0$. Then it follows from (\ref{e:independence of klm-inhomogeneous-proof-J1+J2}) that     
\begin{eqnarray*}
	| x |_{ R^{ s, A }_{ q, X } ( k, \alpha, \beta ) } \lesssim \| ( 2^k + A )^{ - \alpha } x \| + | x |_{ R^{ s, A }_{ q, X } ( k, \alpha, \beta + 1 ) }, 
\end{eqnarray*}
which is the desired inequality. Thus, we have verified (\ref{e:independence of klm-inhomogeneous m}).

Finally, we verify (\ref{e:independence of klm-inhomogeneous l}). To this end, we may assume that $\RE \alpha \leq \RE \alpha'$, otherwise the desired conclusion follows from the symmetry of $\alpha$ and $\alpha'$. Moreover, it suffices to verify the following two inequalities for some special indices. One is  
\begin{align}\label{e:independence of klm-inhomogeneous-proof-a1<a2}
\big\| ( 2^k + A )^{ - \alpha_1 } x \big\| + | x |_{ R^{ s, A }_{ q, X } ( k, \alpha_1, \beta ) }  \lesssim \big\| ( 2^k + A )^{ - \alpha_2 } x \big\| + | x |_{ R^{ s, A }_{ q, X } ( k, \alpha_2, \beta ) }
\end{align}
for all $\alpha_1, \alpha_2 \in \mathbb{C}_+^*$ with $\RE \alpha_1 > \RE \alpha_2$, and the other is  
\begin{align}\label{e:independence of klm-inhomogeneous-proof-a<a+1}
\big\| ( 2^k + A )^{ - \alpha } x \big\| + | x |_{ R^{ s, A }_{ q, X } ( k, \alpha, \beta ) }  \lesssim \big\| ( 2^k + A )^{ - ( \alpha +1 ) } x \big\| + | x |_{ R^{ s, A }_{ q, X } ( k, \alpha + 1, \beta ) }.
\end{align} 
If this is the case, from (\ref{e:independence of klm-inhomogeneous-proof-a<a+1}) and (\ref{e:independence of klm-inhomogeneous-proof-a1<a2}) it follows that 
\begin{align*}
\big\| ( 2^k + A )^{ - \alpha } x \big\| + | x |_{ R^{ s, A }_{ q, X } ( k, \alpha, \beta ) } & \lesssim \big\| ( 2^k + A )^{ - ( \alpha + 1 ) } x \big\| + | x |_{ R^{ s, A }_{ q, X } ( k, \alpha + 1, \beta ) } \\
& \lesssim \big\| ( 2^k + A )^{ - \alpha' } x \big\| + | x |_{ R^{ s, A }_{ q, X } ( k, \alpha', \beta ) } \\
& \lesssim \big\| ( 2^k + A )^{ - \alpha' + 1 } x \big\| + | x |_{ R^{ s, A }_{ q, X } ( k, \alpha' + 1, \beta ) } \\
& \lesssim \big\| ( 2^k + A )^{ - \alpha } x \big\| + | x |_{ R^{ s, A }_{ q, X } ( k, \alpha, \beta ) }
\end{align*}
whenever $\RE \alpha = \RE \alpha'$ and that 
\begin{align*}
\big\| ( 2^k + A )^{ - \alpha } x \big\| + | x |_{ R^{ s, A }_{ q, X } ( k, \alpha, \beta ) } & \lesssim \big\| ( 2^k + A )^{ - ( \alpha + n ) } x \big\| + | x |_{ R^{ s, A }_{ q, X } ( k, \alpha + n, \beta ) } \\
& \lesssim \big\| ( 2^k + A )^{ - \alpha' } x \big\| + | x |_{ R^{ s, A }_{ q, X } ( k, \alpha', \beta ) } \\
& \lesssim \big\| ( 2^k + A )^{ - \alpha } x \big\| + | x |_{ R^{ s, A }_{ q, X } ( k, \alpha, \beta ) }
\end{align*}
whenever $\RE \alpha < \RE \alpha'$, where $n \in \mathbb{N}$ satisfying $\RE \alpha + n > \RE \alpha'$. Thus, we have verified (\ref{e:independence of klm-inhomogeneous l}) by using (\ref{e:independence of klm-inhomogeneous-proof-a<a+1}) and (\ref{e:independence of klm-inhomogeneous-proof-a1<a2}).

It remains to verify (\ref{e:independence of klm-inhomogeneous-proof-a1<a2}) and (\ref{e:independence of klm-inhomogeneous-proof-a<a+1}). Indeed, (\ref{e:independence of klm-inhomogeneous-proof-a1<a2}) is a direct consequence of (\ref{e:uniform boundedness compositions-M}) and it is sufficient to verify (\ref{e:independence of klm-inhomogeneous-proof-a<a+1}). Thanks to (\ref{e:independence of klm-inhomogeneous m}), it suffices to verify (\ref{e:independence of klm-inhomogeneous-proof-a<a+1}) for the case in which $\beta$ with $\IM \beta = - \IM \alpha$ and $\RE \beta = n - \RE \alpha$, where $n$ is an integer large enough such that $n > s + \RE \alpha$. To this end, let $x \in \overline{ D ( A ) }$ such that $| x |_{ R^{ s, A }_{ q, X } ( k, \alpha + 1, \beta ) } < \infty$. Note that $\alpha + \beta = n \in \mathbb{N}$ due to the hypothesis, and hence, by using (\ref{e:inhomogeneous Calderon}) with $\alpha$ and $m$ replaced by $\alpha + 1$ and $\alpha + \beta$, respectively, we have 
\begin{align*}
\big\| ( 2^k + A )^{ - \alpha } x \big\| & \lesssim \big\| ( 2^k + A )^{ - ( \alpha + 1 ) } x \big\| + \int_{ 2^k }^\infty \big\| A^{ \beta } ( \mu + A )^{ - \alpha - \beta - 1 } x \big\| \, d \mu. 
\end{align*}
Decomposing the integral $\int_{ 2^k }^\infty = \sum_{ i = k }^{ \infty } \int_{ 2^i }^{ 2^{ i + 1 } }$ and applying (\ref{e:resolvent estimate-discrete}) yields 
\begin{align*}
& \big\| ( 2^k + A )^{ - \alpha } x \big\| \\
\lesssim {} & \big\| ( 2^k + A )^{ - ( \alpha + 1 ) } x \big\| + \sum_{ i = k }^\infty 2^{ - i ( s + \RE \alpha ) } \big\| 2^{ i s } 2^{ i ( \alpha + 1 ) } A^{ \beta } ( 2^i + A )^{ - \alpha - \beta - 1 } x \big\| \\
\lesssim {} & \big\| ( 2^k + A )^{ - ( \alpha + 1 ) } x \big\| + | x |_{ R^{ s, A }_{ q, X } ( k, \alpha +1, \beta ) },
\end{align*}
where the last inequality follows from the monotonicity of spaces $\ell_q$ if $0 < q \leq 1$, from the H\"{o}lder inequality if $1 < q < \infty$ and from the convergence of the series $\sum_{ i = k }^\infty 2^{ - i ( s + \RE \alpha ) }$ if $q = \infty$, respectively. Thus, we have verified that
\begin{align}\label{e:independence of klm-inhomogeneous-proof-final}
\big\| ( 2^k + A )^{ - \alpha } x \big\| \lesssim \big\| ( 2^k + A )^{ - ( \alpha + 1 ) } x \big\| + | x |_{ R^{ s, A }_{ q, X } ( k, \alpha + 1, \beta ) }.
\end{align}
Finally, by using (\ref{e:independence of klm-inhomogeneous-proof-final}) and (\ref{e:independence of R-inhomogeneous l}) (or, more precisely, (\ref{e:independence of R-inhomogeneous l-proof1})) we obtain the desired inequality (\ref{e:independence of klm-inhomogeneous-proof-a<a+1}), immediately. The proof is complete.
\end{proof}

\smallskip


\subsection{Inhomogeneous Besov spaces}


Recall that a vector space $\mathcal{X}$ on $\mathbb{K} = \mathbb{R}$ or $\mathbb{C}$ is said to be a quasi-normed space on $\mathbb{K}$ if there is a functional $\left\|
\cdot \right\|_\mathcal{X}$
on $\mathcal{X}$ which satisfies the following three conditions (a), (b) and (c):
\begin{itemize}
\item[(a)] $\| x \|_\mathcal{X} \geq 0$ for $x \in \mathcal{X}$, and $\| x \|_\mathcal{X} = 0$ if and only if $x = 0$,
\item[(b)] $\| \alpha x \|_\mathcal{X} = | \alpha | \, \| x \|_\mathcal{X}$ for $\alpha \in \mathbb{K}$ and $x \in \mathcal{X}$,
\item[(c)] there is a constant $K \geq 1$ such that
\begin{align}\label{e:quasi-triangle inequality}
\| x + y \|_\mathcal{X} \leq K ( \| x \|_\mathcal{X} + \| y \|_\mathcal{X} ), \quad x, y \in \mathcal{X}.
\end{align}
\end{itemize}
%
By the Aoki-Rolewicz theorem\cite{Aoki1942,Rolewicz1957}, 
for each quasi-normed space $( \mathcal{X}, \left\| \cdot \right\|_{ \mathcal{X} } )$, there is an equivalent quasi-norm $\interleave \cdot \interleave_\mathcal{X}$ on $\mathcal{X}$ which satisfies the $p$-subadditivity, i.e.,
\begin{align*}
\interleave x + y \interleave_\mathcal{X}^p \leq \interleave x \interleave_\mathcal{X}^p
+ \interleave y \interleave_\mathcal{X}^p, \quad x, y \in \mathcal{X},
\end{align*}
for some $0 < p \le 1$. More precisely, one can take $p = \ln 2 / ( \ln K + \ln 2 )$ with $K$, the modulus of concavity given in (\ref{e:quasi-triangle inequality}), and define an equivalent quasi-norm $\interleave \cdot \interleave_\mathcal{X}$ on $\mathcal{X}$ satisfying the $p$-subadditivity in the following way: 
\begin{align}\label{e:quasinorm-p-subadditivity}
\interleave x \interleave_\mathcal{X} := \inf \Bigg\{ \bigg( \sum_{ i = 1 }^n \| x_i \|^p_{ \mathcal{ X } } \bigg)^{ 1 / p }: x = \sum_{ i = 1 }^n x_i, x_i \in \mathcal{X}, i = 1, 2, \cdots, n \in \mathbb{N} \Bigg\} 
\end{align}
for $x \in \mathcal{X}$ (see \cite[Theorems 1.2 and 1.3]{KaltonM1984}).
%
%
The metric topology on $( \mathcal{X}, \left\| \cdot \right\|_{ \mathcal{X} } )$, induced by $\left\| \cdot \right\|_{ \mathcal{X} }$, can be defined by $d ( x, y ) := \interleave x - y \interleave_\mathcal{X}^p$ for $x, y \in \mathcal{X}$,
where $\interleave \cdot \interleave_\mathcal{X}$ is the equivalent quasi-norm on $\mathcal{X}$ given by (\ref{e:quasinorm-p-subadditivity}). Therefore, 
every quasi-normed space admits a completion in the sense of equivalent quasi-norms.
Moreover, a quasi-normed space $\mathcal{X}$ is said to be a quasi-Banach space if it is complete with respect to this metric $d$.  We also refer the reader to \cite{Kalton2003} for a brief review on quasi-Banach spaces.

Let $0 < q \leq \infty$ and $s \in \mathbb{R}$, and fix $k \in \mathbb{Z}$ and $\alpha, \beta \in \mathbb{C}_+^*$ satisfying $- \RE \alpha < s < \RE \beta$. 
We now turn to the vector space $R^{ s, A }_{ q, X } ( k, \alpha, \beta )$ given in Definition \ref{d:Besov type-inhomogeneous}. 
Observe that 
\begin{align}\label{e:quasi-norm of inhomogeneous Besov spaces}
\| x \|_{ B^{ s, A }_{ q, X } ( k, \alpha, \beta ) } := \left\| ( 2^k + A )^{ - \alpha } x \right\| + | x |_{ R^{ s, A }_{ q, X } ( k, \alpha, \beta ) },
\end{align}
is a quasi-norm on $R^{ s, A }_{ q, X } ( k, \alpha, \beta )$ due to the fact that $\left\| ( 2^k + A )^{ - \alpha } \cdot \right\|$ is a norm on the Banach space $X$. Therefore, there is a metric topology on $R^{ s, A }_{ q, X } ( k, \alpha, \beta )$
induced by the quasi-norm $ \left\| \cdot \right\|_{ B^{ s, A }_{ q, X } ( k, \alpha, \beta ) }$ as we discussed in the last paragraph. More precisely, a metric on $R^{ s, A }_{ q, X } ( k, \alpha, \beta )$ can be defined by 
\begin{align}\label{e:inhomogeneous Besov metric}
d ( x, y ) = \interleave x - y \interleave_{ B^{ s, A }_{ q, X } ( k, \alpha, \beta ) }^p, \quad x, y \in R^{ s, A }_{ q, X } ( k, \alpha, \beta ),
\end{align}
where $\interleave \cdot \interleave_{ B^{ s, A }_{ q, X } ( k, \alpha, \beta ) }$ is the equivalent quasi-norm given by (\ref{e:quasinorm-p-subadditivity}) with $( \mathcal{X}, \left\| \cdot \right\|_{ \mathcal{X} } )$
replaced by
$( R^{ s, A }_{ q, X } ( k, \alpha, \beta ), \left\| \cdot \right\|_{ B^{ s, A }_{ q, X } ( k, \alpha, \beta ) } )$.

Unfortunately, the space $R^{ s, A }_{ q, X } ( k, \alpha, \beta )$ is possibly not complete with respect to the metric $d$ given by (\ref{e:inhomogeneous Besov metric}), even in some classical examples (like $A = - \Delta_p$, the negative of the Laplacian on $L^p ( \mathbb{R}^n )$). 
We turn to the completion of $R^{ s, A }_{ q, X } ( k, \alpha, \beta )$ with respect to the quasi-norm $\left\| \cdot \right\|_{ B^{ s, A }_{ q, X } ( k, \alpha, \beta ) }$ 
and define (inhomogeneous) Besov spaces associated with $A$ as follows.


\begin{defn}\label{d:inhomogeneous Besov spaces}
Let $0 < q \leq \infty$ and $s \in \mathbb{R}$, and fix $k \in \mathbb{Z}$ and $\alpha, \beta \in \mathbb{C}_+^*$ satisfying $- \RE \alpha < s < \RE \beta$. The inhomogeneous
Besov space $B^{ s, A }_{ q, X } ( k, \alpha, \beta )$ associated with $A$ is defined as the completion of $ R^{ s, A }_{ q, X } ( k, \alpha, \beta ) $ with respect to the metric $d$ given in (\ref{e:inhomogeneous Besov metric}) (the quasi-norm on $B^{ s, A }_{ q, X } ( k, \alpha, \beta )$ is still denoted by $\left\| \cdot \right\|_{ B^{ s, A }_{ q, X } ( k, \alpha, \beta ) }$ for convenience).
\end{defn}


\begin{rem}\label{r:independence of klm-inhomogeneous}
As shown in Lemma \ref{l:independence of klm-inhomogeneous}, the quasi-norm $\left\| \cdot \right\|_{ B^{ s, A }_{ q, X } ( k, \alpha, \beta ) }$ given by (\ref{e:quasi-norm of inhomogeneous Besov spaces}) is independent of the choice of $k, \alpha$ and $\beta$ 
in the sense of equivalent quasi-norms.
We shall not distinguish between equivalent quasi-norms of a given quasi-Banach space, and therefore we 
write $B^{ s, A }_{ q, X }$ instead of $B^{ s, A }_{ q, X } ( k, \alpha, \beta )$ as well as $\left\| \cdot \right\|_{ B^{ s, A }_{ q, X } }$ instead of $\left\| \cdot \right\|_{ B^{ s, A }_{ q, X } ( k, \alpha, \beta ) }$ unless otherewise the indices $k, \alpha$ and $\beta$ need to be explicitly specified.
\end{rem}


The next result reveals that the inhomogeneous Besov space $B^{ s, A }_{ q, X }$ is indeed a subspace of the underlying space $X$ whenever $s > 0$. 

\begin{prop}\label{p:completeness-inhomogeneous}
Let $s > 0 $ and $0 < q \leq \infty$. Then $B^{ s, A }_{ q, X } (k, \alpha, \beta) = R^{ s, A }_{ q, X } (k, \alpha, \beta)$ for each $k \in \mathbb{Z}$, $\alpha \in \mathbb{C}_+^*$ and $\beta \in \mathbb{C}_+$ satisfying $- \RE \alpha < s < \RE \beta$. 
\end{prop}

\begin{proof}
Let $k \in \mathbb{Z}$ and fix $\alpha \in \mathbb{C}_+^*$ and $\beta \in \mathbb{C}_+$ such that 
$- \RE \alpha < s < \RE \beta$. It suffices to verify that $R^{ s, A }_{ q, X } (k, \alpha, \beta)$ is complete with respect to $\left\| \cdot \right\|_{ B^{ s, A }_{ q, X } ( k, \alpha, \beta ) }$. By Lemma \ref{l:independence of klm-inhomogeneous} above, it remains to verify that $R^{ s, A }_{ q, X } (0, 0, \beta)$ is complete with respect to $\left\| \cdot \right\|_{ B^{ s, A }_{ q, X } ( 0, 0, \beta) }$.
To this end, suppose that $\left\{ x_n \right\} \subset \overline{ D ( A ) }$ is a Cauchy sequence with respect to $\left\| \cdot \right\|_{ B^{ s, A }_{ q, X } ( 0, 0, \beta ) }$. Let $\epsilon > 0$ and fix $N \in \mathbb{N}$ such that 
\begin{align}\label{e:completeness-inhomogeneous-proof1}
\left\| x_m - x_n \right\|_{ B^{ s, A }_{ q, X } ( 0, 0, \beta ) } < \epsilon, \quad m, n \geq N.
\end{align}
By (\ref{e:quasi-norm of inhomogeneous Besov spaces}), 
it is clear that $\left\{ x_n \right\}$ is also Cauchy in $X$, and hence, there is an $x \in \overline{ D ( A ) }$ such that $x_n \rightarrow x$ as $n \rightarrow \infty$. Applying $n \rightarrow \infty$ to (\ref{e:completeness-inhomogeneous-proof1}) yields 
\begin{align*}
\| x_m - x \| + \left| x_m - x \right|_{ R^{ s, A }_{ q, X } ( 0, 0, \beta ) } \leq \epsilon, \quad m \geq N,
\end{align*}
and hence, 
\begin{align*}
\left| x \right|_{ R^{ s, A }_{ q, X } ( 0, 0, \beta ) } \lesssim {} & \left| x - x_N \right|_{ R^{ s, A }_{ q, X } ( 0, 0, \beta ) }
+ \left| x_N \right|_{ R^{ s, A }_{ q, X } ( 0, 0, \beta ) } < \infty.
\end{align*}
This implies that $x \in R^{ s, A }_{ q, X } (0, 0, \beta)$ and $x_n \to x$ with respect to the quasi-norm $\left\| \cdot \right\|_{ B^{ s, A }_{ q, X } ( 0, 0, \beta ) }$. The proof is complete.
\end{proof}

Now we give a connection between the fractional domains $D ( A^\alpha )$ and Besov spaces $B^{ s, A }_{ q, X }$, which will be used to establish interpolation spaces for abstract Besov spaces in Section \ref{S:Fractional powers} below. 

\begin{prop}\label{p:inclusions-powers-inhomogeneous}
Let $0 < q \le \infty$, and let $s, s' > 0$ and $\alpha \in \mathbb{C}_+$ with $0 < s < \RE \alpha < s' < \infty$. Then 
\begin{align}\label{e:inclusions-powers-inhomogeneous}
B^{ s', A }_{ q, X } \subset D ( A^\alpha ) \subset B^{ s, A }_{ q, X }.
\end{align}
\end{prop}

\begin{proof}
The inclusion 
\begin{align*}
D ( A^\alpha ) \subset B^{ s, A }_{ q, X } 
\end{align*}
follows from the non-negativity of $A$. Indeed, let $x \in D ( A^\alpha )$. Fix $\beta \in \mathbb{C}_+$ with $\RE \alpha < \RE \beta$. From (\ref{e:uniform boundedness compositions-M})$^*$ and (\ref{e:uniform boundedness compositions-L})$^*$ it follows that 
\begin{align*}
\| x \|_{ R^{ s, A }_{ q, X } ( 0, 0, \beta ) } = {} & \sup_{ j \ge 0 } \big\| 2^{ j s } A^\beta ( 2^j + A )^{ - \beta } x \big\| \le C \| A^\alpha x \| 
\end{align*}
for $q = \infty$ and 
\begin{align*}
\| x \|_{ R^{ s, A }_{ q, X } ( 0, 0, \beta ) } = {} & \sum_{ j = 0 }^{ \infty } \big\| 2^{ j s } A^\beta ( 2^j + A )^{ - \beta } x \big\|^q \le 
\frac{ C }{ 1 - 2^{ ( s - \RE \alpha ) q } } \| A^\alpha x \| 
\end{align*}
for $0 < q < \infty$, where $C = C_{ \alpha, n } C_{ \beta - \alpha, m } M_A^n L_A^m$ with $C_{ \alpha, n }, C_{ \beta - \alpha, m }$ given by (\ref{e:uniform boundedness compositions}) and $M_A$ and $L_A$ given by (\ref{e:non-negative constant}) and (\ref{e:non-negative constant 2}), respectively. This implies that $x \in B^{ s, A }_{ q, X }$.

It remains to verify that 
\begin{align}\label{e:inclusions-powers-inhomogeneous-proof-B subset D}
B^{ s', A }_{ q, X } \subset D ( A^\alpha ).  
\end{align}
Let $x \in B^{ s', A }_{ q, X }$. We now verify that $x \in D ( A^\alpha )$.  
Thanks to Corollary \ref{c:representation of fractional powers-beta-absolute convergence}, it suffices to verify that  
\begin{align}\label{e:inclusions-powers-inhomogeneous-proof-absolute convergence}
\int_0^\infty \big\| \lambda^\alpha A^\beta ( \lambda + A )^{ - \beta } x \big\| \, \frac{ d \lambda }{ \lambda } < \infty,
\end{align}
where $\beta \in \mathbb{C}_+$ with $\RE \alpha < \RE \beta$. To this end, fix $\beta \in \mathbb{C}_+$ such that $\RE \alpha < \RE \beta$. For the part $\int_0^1$, from (\ref{e:uniform boundedness compositions-L})$^*$ it follows that 
\begin{align*}
\int_0^1 \big\| \lambda^\alpha A^\beta ( \lambda + A )^{ - \beta } x \big\| \, \frac{ d \lambda }{ \lambda } \le \frac{ C_{ \beta, m } L_A^m }{ \RE \alpha } \| x \|,
\end{align*}
where $C_{ \beta, m }$ is given by (\ref{e:uniform boundedness compositions}) and $L_A$ is the non-negativity constant $A$ given by (\ref{e:non-negative constant 2}). As for the part $\int_1^\infty$, rewriting $\int_1^\infty = \sum_{ j = 0 }^{ \infty } \int_{ 2^j }^{ 2^{ j +1 } }$ and applying (\ref{e:resolvent estimate-discrete}) with $c = 1$ yields 
\begin{align*}
\int_1^\infty \big\| \lambda^\alpha A^\beta ( \lambda + A )^{ - \beta } x \big\| \, \frac{ d \lambda }{ \lambda } \lesssim {} & \sum_{ j = 0 }^{ \infty } \big\| 2^{ j \alpha } A^\beta ( 2^j + A )^{ - \beta } x \big\| \\
= {} & \sum_{ j = 0 }^{ \infty } \big\| 2^{ j ( \alpha - s' ) } 2^{ j s' }A^\beta ( 2^j + A )^{ - \beta } x \big\| \lesssim | x |_{ R^{ s', A }_{ q, X } ( 0, 0, \beta ) },
\end{align*}
where the last ineqaulity follows from (\ref{e: q-inequality}) in the case $0 < q \le 1$ and from the H\"{o}lder inequality in the case $1 < q < \infty$ (the case when $q = \infty$ is a direct consequence of the estimate (\ref{e:resolvent estimate-discrete}) with $c = 1$). Thus, we have verified (\ref{e:inclusions-powers-inhomogeneous-proof-absolute convergence}). 
The proof is complete.
\end{proof}

A density property of the inhomogeneous Besov spaces associated with non-negative operators is given as follows.

\begin{prop}\label{p:denseness-inhomogeneous}
Let $s \in \mathbb{R}$ and $0 < q < \infty$. Then $D ( A^\beta )$ is dense in $B^{ s, A }_{ q, X }$ for each $\beta \in \mathbb{C}_+$ satisfying $\RE \beta > |s|$.
\end{prop}

\begin{proof}
Since the Besov space $B^{ s, A }_{ q, X }$ is the completion of $R^{ s, A }_{ q, X }$ with respect to $\left\| \cdot \right\|_{ B^{ s, A }_{ q, X } }$, it suffices to show that $D ( A^\beta )$ is dense in $R^{ s, A }_{ q, X }$ with respect to $\left\| \cdot \right\|_{ B^{ s, A }_{ q, X } }$. To this end, let $x \in R^{ s, A }_{ q, X }$, fix $\beta \in \mathbb{C}_+$ such that $\RE \beta > |s|$ and write 
\begin{align*}
	x_n := n^\beta ( n + A )^{ - \beta } x \in D ( A^\beta ), \quad n \in \mathbb{N}. 
\end{align*}
It suffices to verify that $\{ x_n \}$ converges to $x$ in $B^{ s, A }_{ q, X }$. Indeed, fix $\epsilon > 0$. Thanks to the fact that $\| x \|_{ B^{ s, A }_{ q, X } } < \infty$, there exists a $K \in \mathbb{N}$ such that 
\begin{align}\label{e:denseness-inhomogeneous-proof1}
\left\{
\sum_{ i = K }^\infty \big\| 2^{ i ( s + \beta ) } A^{ \beta } ( 2^i + A )^{ - 2 \beta } x \big\|^q \right\}^{ 1 / q } < \frac{ \epsilon }{ 4 C_{ 1/q }^2 C_{ \beta, m } M_A^m },
\end{align}
where $\RE \beta < m \in \mathbb{N}$ and $C_{ 1 / q }$, $C_{ \beta, m }$ and $M_A$ are the three constants given in (\ref{e: q-inequality}), (\ref{e:uniform boundedness compositions}) and (\ref{e:non-negative constant}), respectively. Furthermore, since $x_n \rightarrow x$ in $X$ as $n \rightarrow \infty$ due to Lemma \ref{l:ergodicity} (i), there exists an $N \in \mathbb{N}$ such that 
\begin{align}\label{e:denseness-inhomogeneous-proof2}
	\| x_n - x \| < \frac{ \epsilon }{ 4 C_{ 1 / q } D_K }, \quad n \geq N,
\end{align}
where $D_K = \max \big\{ d_K, \| ( 1 + A )^{ - \beta } \| \big\}$ with 
\begin{align*}
	d_K = \left\{ \sum_{ i = 0 }^{ K - 1 } \big\| 2^{ i ( s + \beta ) } A^{ \beta } ( 2^i + A )^{ - 2 \beta } \big\|^q \right\}^{ 1 / q }, 
\end{align*}
and hence, 
\begin{align}\label{e:denseness-inhomogeneous-proof3}
\big\| ( 1 + A )^{ - \beta } ( x_n - x ) \big\| < \frac{ \epsilon }{ 4 }, \quad n \geq N.
\end{align}
From (\ref{e:denseness-inhomogeneous-proof1}), (\ref{e:denseness-inhomogeneous-proof2}) and (\ref{e:denseness-inhomogeneous-proof3}) it follows that 
\begin{align*}
\| x_n - x \|_{ B^{ s, A }_{ q, X } } 
 \leq {} & \big\| ( 1 + A )^{ - \beta } ( x_n - x ) \big\| \\
+ {} & C_{ 1 / q }^2 \left\{ \sum_{ i = K }^\infty \big\| 2^{ i ( s + \beta ) } A^{ \beta } ( 2^i + A )^{ - 2 \beta } x_n \big\|^q \right\}^{ 1 / q } \\
+ {} & C_{ 1 / q }^2 \left\{ \sum_{ i = K }^\infty \big\| 2^{ i ( s + \beta ) } A^{ \beta } ( 2^i + A )^{ - 2 \beta } x \big\|^q \right\}^{ 1 / q } \\
+ {} & C_{ 1 / q } \left\{ \sum_{ i = 0 }^{ K - 1 } \big\| 2^{ i ( s + \beta ) } A^{ \beta } ( 2^i + A )^{ - 2 \beta } ( x_n - x ) \big\|^q \right\}^{ 1 / q } \\
< {} & \frac{ \epsilon }{ 4 } + \frac{ \epsilon }{ 4 } + \frac{ \epsilon }{ 4 } + \frac{ \epsilon }{ 4 } = \epsilon, \quad n \geq N,
\end{align*}
and hence, $\{x_n\} \subset D ( A^\beta )$ converges to $x$ as $n \rightarrow \infty$ in $ B^{ s, A }_{ q, X } $. Thus, we have verified that $D ( A^\beta )$ is dense in $R^{ s, A }_{ q, X }$ with respect to $\left\| \cdot \right\|_{ B^{ s, A }_{ q, X } }$. The proof is complete.
\end{proof}

In addition to the (inhomogeneous) Besov spaces
associated with non-negative operators given in Definition \ref{d:inhomogeneous Besov spaces}, an alternative version of abstract (inhomogeneous) Besov spaces
is given as follows. 

\begin{defn}\label{d:inhomogeneous Besov spaces-breve}
Let $s \in \mathbb{R}$ and $0 < q \leq \infty$, and let $A$ be injective. The inhomogeneous Besov space $\breve{B}^{ s, A }_{ q, X } := \breve{B}^{ s, A }_{ q, X } ( k, \alpha, \beta )$ associated with $A$ is defined as the completion of the subspace 
\begin{align*}
\breve{R}^{ s, A }_{ q, X } ( k, \alpha, \beta ) := \left\{ x \in \overline{ R ( A ) }: \sum_{ i = - \infty }^{ k } \big\| 2^{ i ( s + \alpha ) } A^\beta ( 2^i + A )^{ - \alpha - \beta } x \big\|^q < \infty \right\}
\end{align*}
with respect to the 
quasi-norm
\begin{align*}
\| x \|_{ \breve{B}^{ s, A }_{ q, X } } := \left\| A^\beta ( 2^k + A )^{ - \beta } x \right\| + \left\{ \sum_{ i = - \infty }^{ k } \big\| 2^{ i ( s + \alpha ) } A^\beta ( 2^i + A )^{ - \alpha - \beta } x \big\|^q \right\}^{ 1 / q }
\end{align*}
(with the usual modification if $q = \infty$), where $k \in \mathbb{Z}$ and $\alpha \in \mathbb{C}_+^*$ and $\beta \in \mathbb{C}_+$ satisfying $- \RE \alpha < s < \RE \beta$.
\end{defn}

\begin{rem}
Analogous to $B^{ s, A }_{ q, X }$, it can be verified that $\breve{B}^{ s, A }_{ q, X }$ is also independent of the choice of indices $k, \alpha$ and $\beta$, and hence, $\breve{B}^{ s, A }_{ q, X }$ is well defined as a quasi-Banach space (Banach space) for each $0 < q \leq \infty$ ($1 \leq q \leq \infty$). In particular, the inhomogeneous Besov space $\breve{B}^{ s, A }_{ q, X }$ with  $s < 0$ and $1 \leq q \leq \infty$ coincides with the Komatsu space $R^{ - s } _q ( A )$ \cite[Definitions 2.1 and 2.4]{Komatsu1969}.
\end{rem}


The next result reveals a connection between Besov spaces $\breve{B}^{ s, A }_{ q, X }$ and $B^{ s, A }_{ q, X }$.

\begin{prop}\label{p:inverse-inhomogeneous}
Let $A$ be injective, and let $s \in \mathbb{R}$ and $0 < q \leq \infty$. Then $\breve{B}^{ - s, A }_{ q, X } = B^{ s, A^{ - 1 } }_{ q, X }$ in the sense of equivalent quasi-norms. In particular, $\breve{B}^{ 0, A }_{ q, X } = B^{ 0, A^{ - 1 } }_{ q, X }$ in the sense of equivalent quasi-norms. 
\end{prop}

\begin{proof}
Thanks to the density argument, it suffices to verify that $R^{ s, A }_{ q, X } = \breve{R}^{ s, A }_{ q, X }$ from the set-theoretic point of view and that  
\begin{align*}
\| x \|_{ \breve{B}^{ - s, A }_{ q, X } } \simeq \| x \|_{ B^{ s, A^{ - 1 } }_{ q, X } }, \quad x \in R^{ s, A }_{ q, X }. 
\end{align*}
Indeed, 
let $x \in R^{ s, A }_{ q, X }$ and fix $k \in \mathbb{Z}$ and $\alpha, \beta \in \mathbb{C}_+^*$ such that $- \RE \alpha < s < \RE \beta$. By Definitions \ref{d:inhomogeneous Besov spaces} and \ref{d:inhomogeneous Besov spaces-breve}, we have 
\begin{align*}
\begin{split}
 \| x \|_{ \breve{B}^{ - s, A }_{ q, X } } 
 \simeq {} & \big\| A^\alpha ( 2^k + A )^{ - \alpha } x \big\| \\
+ {} & \left\{ \sum_{ i = - \infty }^{ k } \big\| 2^{ i ( - s + \beta ) } A^\alpha ( 2^i + A )^{ - \beta - \alpha } x \big\|^q \right\}^{ 1 / q } \\ 
= {} &2^{ - k \RE \alpha } \big\| ( 2^{ - k } + A^{ - 1 } )^{ - \alpha } x \big\| \\
 + {} & \left\{ \sum_{ i = - k }^{ \infty } \big\| 2^{ i ( s + \alpha ) } A^{ - \beta } ( 2^{ i } + A^{ - 1 } )^{ - \alpha - \beta } x \big\|^q \right\}^{ 1 / q } \simeq \left\| x \right\|_{ B^{ s, A^{ - 1 } }_{ q, X } },
\end{split}	
\end{align*}
the desired conclusion. The proof is complete.
\end{proof}

In the next subsection we will introduce the homogeneous Besov space associated with
non-negative operators. Further properties of inhomogeneous Besov spaces will be given in Sections \ref{S:Basic properties} and \ref{S:Fractional powers} subsequently.

\smallskip


\subsection{Homogeneous Besov spaces}\label{Sub:Homogeneous Besov spaces}

Let $A$ be an injective non-negative operator on $X$. Analogous to the inhomogeneous Besov spaces associated with $A$, in order to define homogeneous Besov spaces associated
with $A$, we start with the following subspace
$\dot{R}^{ s, A }_{ q, X } ( \alpha, \beta )$.

\begin{defn}\label{d:homogeneous Besov spaces-R+Ker}
Let $0 < q \leq \infty$ and $s \in \mathbb{R}$, and let $\alpha \in \mathbb{C}_+^*$ and $\beta \in \mathbb{C}_+$ such that $- \RE \alpha < s < \RE \beta$. The space $\dot{R}^{ s, A }_{ q, X } ( \alpha, \beta )$ is defined by
\begin{align*}
\dot{R}^{ s, A }_{ q, X } ( \alpha, \beta ) := \left\{ x \in \overline{ D (A) } \cap \overline{ R (A) }: | x |_{ \dot{R}^{ s, A }_{ q, X } ( \alpha, \beta ) } < \infty \right\}, 
\end{align*}
endowed with the 
quasi-norm
\begin{align*}
| x |_{ \dot{R}^{ s, A }_{ q, X } ( \alpha, \beta ) } := \left\{ \sum_{ i = -
\infty }^\infty \big\| 2^{ i ( s + \alpha ) } A^{ \beta } ( 2^i + A
)^{ - \alpha - \beta } x \big\|^q \right\}^{ 1 / q }
\end{align*}
(with the usual modification if $q = \infty$).
\end{defn}



Analogous to Lemma \ref{l:independence of klm-inhomogeneous} above, we can verify that the
space $\dot{R}^{ s, A }_{ q, X } ( \alpha, \beta )$ is independent of the choice
of indices $\alpha$ and $\beta$ in the sense of equivalent quasi-norms.

\begin{lem}\label{l:independence of l and m-homogeneous}
Let $s \in \mathbb{R}$ and $0 < q \leq \infty$, and let $\alpha \in \mathbb{C}_+^*$ and $\beta \in \mathbb{C}_+$ such that $- \RE \alpha < s < \RE \beta$. 
Then $\dot{R}^{ s, A }_{ q, X } ( \alpha, \beta ) = \dot{R}^{ s, A }_{ q, X } ( \alpha', \beta' )$
in the sense of equivalent quasi-norms
for all $\alpha' \in \mathbb{C}_+^*$ and $\beta' \in \mathbb{C}_+$ satisfying $- \RE \alpha' < s < \RE \beta'$.
\end{lem}

\begin{proof}
By analogous arguments given in the first paragraph of the proof of Lemma \ref{l:independence of lR-inhomogeneous}, it suffices to verify that
\begin{align}\label{e:proof 1 of l:independence of l and m-homogeneous}
| x |_{ \dot{R}^{ s, A }_{ q, X } ( \alpha, \beta ) } \lesssim | x |_{
\dot{R}^{ s, A }_{ q, X } ( \alpha + 1, \beta ) }
\end{align}
and that 
\begin{align}\label{e:proof 2 of l:independence of l and m-homogeneous}
| x |_{ \dot{R}^{ s, A }_{ q, X } ( \alpha, \beta ) } \lesssim | x |_{
\dot{R}^{ s, A }_{ q, X } ( \alpha, \beta + 1 ) }.
\end{align}
Moreover, we merely verify the conclusion in the case $0 < q < \infty$ and the case when $q = \infty$ can be verified analogously. 

First we verify (\ref{e:proof 1 of l:independence of
l and m-homogeneous}). To this end, 
let $x \in \dot{R}^{ s, A }_{ q, X } ( \alpha + 1, \beta )$. Note that $x \in \overline{D (
A )}$ by Definition \ref{d:homogeneous Besov spaces-R+Ker} if $\alpha = 0
$ and that $( 2^i + A )^{ - \alpha } x \in D ( A^\alpha ) \subset \overline{ D ( A ) }$ if $\RE \alpha > 0$, and hence, by using (\ref{e:transform resolvent 3}) with $\lambda = 2^i$ we have 
\begin{align*}
| x |_{ R^{ s, A }_{ q, X } ( \alpha, \beta ) } 
& \leq \left| \alpha + \beta \right| \left\{ \sum_{ i = - \infty }^\infty \bigg[ 2^{ i ( s + \RE \alpha ) } \int_{ 2^i }^\infty \big\| A^{ \beta } ( t + A )^{ - \alpha - \beta - 1 } x \big\| \, d t \bigg]^q \right\}^{ 1 / q }.
\end{align*}
Rewriting $\int_{ 2^i }^\infty = \sum_{ r = i }^{ \infty } \int_{ 2^r }^{ 2^{ r + 1 } }$ and applying (\ref{e:resolvent estimate-discrete}) yields
\begin{align*}
| x |_{ R^{ s, A }_{ q, X } ( \alpha, \beta ) }
& \leq C \left\{ \sum_{ i = - \infty }^\infty \bigg[ \sum_{ r = i }^\infty 2^{ i ( s + \RE \alpha ) } 2^r \big\| A^{ \beta } ( 2^r + A )^{ - \alpha - \beta - 1 } x \big\| \bigg]^q \right\}^{ 1 / q } \\
& = C K,
\end{align*}
where $C = \left| \alpha + \beta \right| C_{ \alpha + \beta + 1, N } ( L_A + 2 M_A )^{ N } $ with $\RE ( \alpha + \beta + 1 ) < N \in \mathbb{N}$ and $C_{ \alpha + \beta + 1, N }$ given by (\ref{e:uniform boundedness compositions}) and
\begin{align*}
K = \left\{ \sum_{ i = - \infty }^\infty \bigg[ \sum_{ r = i }^\infty 2^{ i ( s + \RE \alpha ) } 2^r \big\| A^{ \beta } ( 2^r + A )^{ - \alpha - \beta - 1 } x \big\| \bigg]^q \right\}^{ 1 / q }.
\end{align*}
It can be verified that 
\begin{align}\label{e:independence of l and m-homogeneous-proof-K}
K \lesssim | x |_{ R^{ s, A }_{ q, X } ( \alpha + 1, \beta ) }.
\end{align}
Indeed, if $0 < q \leq 1$, applying (\ref{e: q-inequality-series<1}) to the inner summation of $K$ and exchanging the order of summation yields
\begin{align*}
K
\leq \frac{ 1 }{ [ 1 - 2^{ - ( s + \RE \alpha ) q } ]^{ 1 / q } } | x |_{
R^{ s, A }_{ q, X } ( \alpha + 1, \beta ) }.
\end{align*}
And if $1 < q < \infty$, rewriting $K$ as
\begin{align*}
K 
& = \left\{ \sum_{ i = - \infty }^\infty 2^{ i ( s + \RE \alpha ) q }
\bigg[ \sum_{ r = i }^\infty 2^{ - r ( s + \alpha - \epsilon ) } \big\|
2^{ r ( s + \alpha - \epsilon ) } 2^r A^{ \beta } ( 2^r + A )^{ - \alpha - \beta - 1 }
x \big\| \bigg]^q \right\}^{ 1 / q }
\end{align*}
with $0 < \epsilon < s + \RE \alpha$ and applying the H\"{o}lder inequality
to the inner summation
yields 
\begin{align*}
K 
\leq \frac{ 1 }{ [ 1 - 2^{ - ( s + \RE \alpha - \epsilon ) q' } ]^{ 1 / q' }
} \frac{ 1 }{ ( 1 - 2^{ - \epsilon q } )^{ 1 / q } } | x |_{ R^{ s,
A }_{ q, X } ( \alpha + 1, \beta ) }.
\end{align*}
Thus, we have verified (\ref{e:independence of l and m-homogeneous-proof-K}).

Next we verify (\ref{e:proof 2 of l:independence of l and m-homogeneous}). To this end, 
let $x \in R^{ s, A }_{ q, X } ( \alpha, \beta + 1 )$. 
By letting $\lambda = 2^i$ in (\ref{e:transform resolvent 2}) we have
\begin{align*}
| x |_{ R^{ s, A }_{ q, X } ( \alpha, \beta ) } 
\leq \left| \alpha + \beta \right| \left\{ \sum_{ i = - \infty }^\infty \bigg[ \int_{ 0 }^{ 2^i } \big\| 2^{ i ( s - \beta ) } t^{ \alpha + \beta } A^{ \beta + 1 }
( t + A )^{ - \alpha - \beta - 1 } x \big\| \, \frac{ d t }{ t } \bigg]^q \right\}^{ 1 / q }. 
\end{align*}
Rewriting $\int_{ 0 }^{ 2^i } = \sum_{r = - \infty }^{ i - 1 } \int_{ 2^r }^{ 2^{ r + 1 } }$ and applying (\ref{e:resolvent estimate-discrete}) to the integrands yields
\begin{align*}
| x |_{ R^{ s, A }_{ q, X } ( \alpha, \beta ) } \lesssim J,
\end{align*}
where 
\begin{align*}
J = \left\{ \sum_{ i = - \infty }^\infty \bigg[ \sum_{ r = - \infty }^{ i - 1 } \big\| 2^{ i ( s - \beta ) } 2^{ r ( \alpha + \beta ) } A^{ \beta + 1 } ( 2^r + A )^{ - \alpha - \beta - 1 } x \big\| \bigg]^q \right\}^{ 1 / q }.
\end{align*}
It remains to verify that
\begin{align}\label{e:independence of l and m-homogeneous-proof-I}
J \lesssim | x |_{ R^{ s, A }_{ q, X } ( \alpha, \beta + 1 ) }.
\end{align}
Indeed, if $0 < q \leq 1$, applying (\ref{e: q-inequality}) to the inner summation of $J$ and exchanging the order of summation yields  
\begin{align*}
J & \leq \left\{ \sum_{ i = - \infty }^\infty \sum_{ r = - \infty
}^{ i - 1 } \big\| 2^{ i ( s - \beta ) } 2^{ r ( \alpha + \beta ) } A^{ \beta + 1 } (
2^r + A )^{ - \alpha - \beta - 1 } x \big\|^q \right\}^{ 1 / q } \\
& = \frac{ 2^{ s - \RE \beta } }{ [ 1 - 2^{ ( s - \RE \beta ) q } ]^{ 1 / q } } |
x |_{ R^{ s, A }_{ q, X } ( \alpha, \beta +1 ) }.
\end{align*}
If $1 < q < \infty$, rewriting the part $J$ as
\begin{align*}
J & = \left\{ \sum_{ i = - \infty }^\infty 2^{ i ( s - \RE \beta ) q }
\bigg[ \sum_{ r = - \infty }^{ i - 1 } \big\| 2^{ r \epsilon } 2^{ r
( \alpha + \beta - \epsilon ) } A^{ \beta + 1 } ( 2^r + A )^{ - \alpha - \beta - 1 } x
\big\| \bigg]^q \right\}^{ 1 / q },
\end{align*}
with $0 < \epsilon < \RE \beta - s$, applying the H\"{o}lder inequality to the inner summation of $J$ and exchanging the order of summation yields
\begin{align*}
J \leq 
\frac{ 1 }{ ( 2^{ \epsilon q' } - 1 )^{ 1 / q' } } \frac{ 2^{
s - \RE \beta + \epsilon } }{ [ 1 - 2^{ ( s - \RE \beta + \epsilon ) q } ]^{ 1 / q }
} | x |_{ R^{ s, A }_{ q, X } ( \alpha, \beta +1 ) }.
\end{align*}
Thus, we have verified (\ref{e:independence of l and m-homogeneous-proof-I}). The proof is complete.
\end{proof}

Thanks to Lemma \ref{l:independence of l and m-homogeneous}, we can now
define the homogeneous Besov spaces associated with non-negative operators as
follows.

\begin{defn}\label{d:Besov spaces-homogeneous}
Let $0 < q \leq \infty$ and $s \in \mathbb{R}$. The homogeneous
Besov space $\dot{B}^{ s, A }_{ q, X } := \dot{B}^{ s, A }_{ q, X } ( \alpha, \beta )$ associated with $A$ is defined as the completion of $\dot{R}^{ s, A }_{ q, X } ( \alpha, \beta )$ with respect to the 
quasi-norm
\begin{align}\label{e:Besov quasi-norms-homogeneous}
\left\| \cdot \right\|_{ \dot{B}^{ s, A }_{ q, X } } := \left\| \cdot \right\|_{ \dot{B}^{ s, A }_{ q, X } ( \alpha, \beta ) } := \left| \cdot \right|_{ \dot{R}^{ s, A }_{ q, X } ( \alpha, \beta ) },
\end{align}
where $\alpha \in \mathbb{C}_+^*$ and $\beta \in \mathbb{C}_+$ satisfying $- \RE \alpha < s < \RE \beta$.
\end{defn}

\begin{rem}
By Lemma \ref{l:independence of l and m-homogeneous}, the Besov space $\dot{B}^{ s, A }_{ q, X }$ is well defined as a quasi-Banach space (Banach space) for each $0 < q \leq \infty$ ($1 \leq q \leq \infty$).
In particular, for $q \geq 1$ and $s \in \mathbb{R}$, the Besov space $\dot{B}^{ s,
A }_{ q, X }$ coincides with the space $\dot{B}^\phi_{ X, q }$, with
$\phi ( \lambda ) = \lambda^s$, defined by T. Matsumoto and T. Ogawa
\cite[Definition 2.8]{Matsumoto and Ogawa2010}. Also, see Remark \ref{r:inhomogeneous Besov spaces-equivalent quasi-norms-continuity} below for more information.
\end{rem}

%

Compared with inhomogeneous 
Besov spaces as in Proposition \ref{p:inverse-inhomogeneous}, the next result for homogeneous 
Besov spaces is of independent interest. 

\begin{prop}\label{p:inverse-homogeneous}
Let $0 < q \leq \infty$ and $s \in \mathbb{R}$. Then $\dot{B}^{ - s, A }_{ q, X } = \dot{B}^{ s, A^{ - 1 } }_{ q, X }$ in the sense of equivalent quasi-norms whenever $A$ is injective.
\end{prop}

\begin{proof}
Let $A$ be injective. It suffices to show that $\dot{R}^{ - s, A }_{ q, X } ( \alpha, \beta ) =
\dot{R}^{ s, A^{ - 1 } }_{ q, X } ( \alpha, \beta )$ in the sense of equivalent
quasi-norms for fixed $\alpha \in \mathbb{C}_+^*$ and $\beta \in \mathbb{C}_+$ satisfying $- \RE \alpha < s < \RE \beta$. To this end, let $\alpha \in \mathbb{C}_+^*$ and $\beta \in \mathbb{C}_+$ satisfying $- \RE \alpha < s < \RE \beta$. It is clear that $- \RE \beta < - s < \RE \alpha$, and hence,
\begin{align*}
| x |_{ \dot{R}^{ - s, A }_{ q, X } ( \beta, \alpha ) } & = \left\{ \sum_{ i = - \infty }^{ \infty } \big\| 2^{ i ( - s + \beta ) } A^\alpha ( 2^i + A )^{ - \alpha - \beta } x \big\|^q \right\}^{ 1 / q } \\
	& = \left\{ \sum_{ i = - \infty }^{ \infty } \big\| 2^{ i ( s - \beta ) } A^\alpha ( 2^{ - i } + A )^{ - \alpha - \beta } x \big\|^q \right\}^{ 1 / q } 
= | x |_{ \dot{R}^{ s, A^{-1} }_{ q, X } ( \alpha, \beta ) }, 
\end{align*}
the conclusion desired. The proof is complete.
\end{proof}

Finally, we present a relationship between the inhomogeneous Besov spaces and homogeneous Besov spaces to end this section.

\begin{prop}\label{p:inhomogeneous-homogeneous s>0}
Let $0 < q \leq \infty$ and $s > 0$. Then $B^{ s, A }_{ q, X } =
\dot{B}^{ s, A }_{ q, X } \cap X$ whenever $\overline{ D ( A ) } = \overline{ R ( A ) } = X$.
\end{prop}

\begin{proof}
Let $k = \alpha = 0$ and $\beta \in \mathbb{C}_+$ with $s < \RE \beta$. It is clear that $B^{ s, A }_{ q, X } \subset X$ by Proposition \ref{p:completeness-inhomogeneous}.
Moreover, it is also clear that  
\begin{align*}
| x |_{ R^{ s, A }_{ q, X } ( 0, 0, \beta ) } < \infty \Leftrightarrow | x |_{ \dot{R}^{ s, A }_{ q, X } ( 0, \beta ) } < \infty 
\end{align*}
due to the fact that $s > 0$. Thus, $B^{ s, A }_{ q, X } = \dot{B}^{ s, A }_{ q, X } \cap X$.
\end{proof}

\smallskip


\section{Regular properties of Besov spaces}\label{S:Basic properties}

This section is devoted to basic properties of Besov spaces associated with
non-negative operators, including quasi-norm equivalence, continuous embedding and translation invariance. As in the last section, $(X, \left\| \cdot \right\|)$ is a Banach space and $A$ is a non-negative operator on $X$ in this section.

\smallskip


\subsection{Quasi-norm equivalence}

The semi-quasinorm $\left| \cdot \right|_{ R^{ s, A }_{ q, X } ( k, \alpha, \beta ) }$ given in Definition \ref{d:Besov type-inhomogeneous} can also be characterized by use of the Lebesgue integrals. For convenience, we write
\begin{align}\label{e:Besov type-inhomogeneous-continuity norm}
| x |^c_{ R^{ s, A }_{ q, X } ( k, \alpha, \beta ) } := \bigg\{ \int_{ 2^k
}^\infty \big\| t^{ s + \alpha } A^{ \beta } ( t + A )^{ - \alpha - \beta }
x \big\|^q \, \frac{ d t }{ t } \bigg\}^{ 1 / q },
\end{align}
with the usual modification when $q = \infty$.

\begin{lem}\label{l:Besov type-inhomogeneous-continuity}
Let $0 < q \leq \infty$ and $s \in \mathbb{R}$. Fix $k \in
\mathbb{Z}$ and $\alpha, \beta \in \mathbb{C}_+^*$ such that $- \RE \alpha < s < \RE \beta$, and
let $x \in \overline{ D (A) }$. Then 
$x \in R^{ s, A }_{ q, X } ( k, \alpha, \beta )$ if and only if $| x |^c_{ R^{ s, A }_{ q, X } ( k, \alpha, \beta ) } < \infty$. If this is the case, 
\begin{align*}
| x |_{ R^{ s, A }_{ q, X } ( k, \alpha, \beta ) } \simeq | x |^c_{ R^{ s, A
}_{ q, X } ( k, \alpha, \beta ) }.
\end{align*}
\end{lem}

\begin{proof}
First we verify the statement in the case $0 < q < \infty$. Let $| x |_{ R^{ s, A }_{ q, X } ( k, \alpha, \beta ) } < \infty$. From (\ref{e:Besov type-inhomogeneous-continuity norm}) it follows that 
\begin{align*}
| x |^c_{ R^{ s, A }_{ q, X } ( k, \alpha, \beta ) } & = \left\{ \sum_{ i = k }^\infty \int_{ 2^i }^{ 2^{ i + 1 } } t^{ ( s + \RE \alpha ) q - 1 } \| A^\beta ( t + A )^{ - \alpha - \beta } x \|^q \, d t \right\}^{ 1 / q } \\
& \lesssim \left\{ \sum_{ i = k }^\infty \int_{ 2^i }^{ 2^{ i + 1 } } t^{ ( s + \RE \alpha ) q - 1 } \| A^\beta ( 2^i + A )^{ - \alpha - \beta } x \|^q \, d t \right\}^{ 1 / q } \\
& = \bigg[ \frac{ 2^{ ( s + \RE \alpha ) q } - 1 }{ ( s + \RE \alpha ) q } \bigg]^{ 1 / q } \cdot | x |_{ R^{ s, A }_{ q, X } ( k, \alpha, \beta ) },
\end{align*}
where the last inequality follows from (\ref{e:resolvent estimate-discrete}).

Conversely, let $| x |^c_{ R^{ s, A }_{ q, X } ( k, \alpha, \beta ) } < \infty$. Applying the equality
\begin{align*}
\frac{ s + \RE \alpha }{ 2^{ s + \RE \alpha } - 1 } \int_{ 2^i }^{ 2^{ i + 1 } } t^{ s + \RE \alpha - 1 } 2^{ - i ( s + \RE \alpha ) } \, d t = 1
\end{align*}
yields 
\begin{align*}
| x |_{ R^{ s, A }_{ q, X } ( k, \alpha, \beta ) } & \simeq \left\{ \sum_{ i = k + 1 }^\infty 2^{ i ( s + \RE \alpha ) q } \| A^\beta ( 2^i + A )^{ - \alpha - \beta } x \|^q \right\}^{ 1 / q } \\
& = \bigg[ \frac{ ( s + \RE \alpha ) q  }{ 2^{ ( s + \RE \alpha ) q } - 1 } \bigg]^{ 1 / q } \left\{ \sum_{ i = k + 1 }^\infty \int_{ 2^i }^{ 2^{ i + 1 } } \| t^{ s + \alpha } A^\beta ( 2^i + A )^{ - \alpha - \beta } x \|^q \, \frac{ d t }{ t } \right\}^{ 1 / q }. 
\end{align*}
Applying (\ref{e:resolvent estimate-discrete}) with $c = 2$ yields 
\begin{align*}
| x |_{ R^{ s, A }_{ q, X } ( k, \alpha, \beta ) } 
& \lesssim \left\{ \sum_{ i = k + 1 }^\infty \int_{ 2^i }^{ 2^{ i + 1 } } \| t^{ s + \alpha } A^\beta ( t + A )^{ - \alpha - \beta } x \|^q \, \frac{ d t }{ t } \right\}^{ 1 / q } \\
& = \left\{ \int_{ 2^{ k + 1 } }^{ \infty } \| t^{ s + \alpha } A^\beta ( t + A )^{ - \alpha - \beta } x \|^q \, \frac{ d t }{ t } \right\}^{ 1 / q } \lesssim | x |^c_{ R^{ s, A }_{ q, X } ( k, \alpha, \beta ) }.
\end{align*}
Thus, we have verified the statement for $0 < q < \infty$.

Next we verify the statement for $q = \infty$. Indeed, it is clear that 
\begin{align*}
\sup_{ i \geq k } 2^{ i s } \cdot \big\| 2^{ i \alpha } A^\beta ( 2^i + A )^{ - \alpha - \beta } x \big\| \leq \sup_{ \lambda \geq 2^k } \lambda^s \cdot \big\| \lambda^{ \alpha } A^\beta ( \lambda + A )^{ - \alpha - \beta } x \big\|
\end{align*}
and that 
\begin{align*}
& \sup_{ \lambda \geq 2^k } \lambda^s \cdot \big\| \lambda^{ \alpha } A^\beta ( \lambda + A )^{ - \alpha - \beta } x \big\| \\
= {} & \sup_{ i \geq k } \sup_{ \lambda \in [ 2^i, 2^{ i + 1 } ] } \big\| \lambda^{ s + \alpha } A^\beta ( \lambda + A )^{ - \alpha - \beta } x \big\| \\
\lesssim {} & \sup_{ i \geq k } \big\| 2^{ i ( s + \alpha ) } A^\beta ( 2^i + A )^{ - \alpha - \beta } x \big\|,
\end{align*}
from which we obtain
\begin{align*}
\| x \|^c_{ R^{ s, A }_{ \infty, X } ( k, \alpha, \beta ) } \simeq \| x \|_{ R^{ s, A }_{ \infty, X } ( k, \alpha, \beta ) },
\end{align*}
immediately. The proof is complete.
\end{proof}

Thanks to Lemma \ref{l:Besov type-inhomogeneous-continuity}, we can now give an equivalent characterization of the inhomogeneous Besov spaces as follows. 

\begin{thm}\label{t:inhomogeneous Besov spaces-equivalent quasi-norms-continuity}
Let $s \in \mathbb{R}$ and $0 < q \leq \infty$. Fix $k \in \mathbb{Z}$ and fix $\alpha, \beta \in \mathbb{C}_+^*$ such that $- \RE \alpha < s < \RE \beta$.
Then $B^{ s, A }_{ q, X }$ is the completion of $\overline{ D ( A ) }$ with respect to the quasi-norm $\left\| \cdot \right\|^c_{ B^{ s, A }_{ q, X } }$ 
given by 
\begin{align*}
\| x \|^c_{ B^{ s, A }_{ q, X } } := \| ( 2^k + A )^{ - \alpha } x \| + | x |^c_{ R^{ s, A }_{ q, X } ( k, \alpha, \beta ) },
\end{align*}
where $\left| \cdot \right|^c_{ R^{ s, A }_{ q, X } ( k, \alpha, \beta ) }$ is given in (\ref{e:Besov type-inhomogeneous-continuity norm}) above.
\end{thm}
%
%
%

\begin{rem}\label{r:inhomogeneous Besov spaces-equivalent quasi-norms-continuity}
Let $0 < q \le \infty$ and $s \in \mathbb{R}$. We have the following explicit characterizations of the Besov space $B^{ s, A }_{ q, X }$.
\begin{itemize}
\item[(i)] Let $s > 0$. Fix $k = \alpha = 0 < s
< \RE \beta$. It is clear that $B^{ s, A }_{ q, X } \subset
X$ and that 
\begin{align*}
\| x \|_{ B^{ s, A }_{ q, X } } & \simeq \| x \| + \left\{
\sum_{ i = 0 }^\infty \big\| 2^{ i s } A^{ \beta } ( 2^i + A )^{ - \beta } x
\big\|^q \right\}^{ 1 / q } \\
& \simeq \| x \| + \left\{ \sum_{ i \in \mathbb{Z} } \big\| 2^{
i s } A^{ \beta } ( 2^i + A )^{ - \beta } x \big\|^q \right\}^{ 1 / q },
\quad x \in B^{ s, A }_{ q, X },
\end{align*}
where the last equivalence follows from the estimate (\ref{e:uniform boundedness compositions-L}) and the fact that $s > 0$. Moreover, by Lemma \ref{l:Besov type-inhomogeneous-continuity}, we have
\begin{align*}
\| x \|_{ B^{ s, A }_{ q, X } } & \simeq \| x \| + \left\{
\int_{ 1 }^\infty \big\| t^s A^{ \beta } ( t + A )^{ - \beta } x
\big\|^q \, \frac{ d t }{ t } \right\}^{ 1 / q } \\
& \simeq \| x \| + \left\{ \int_{ 0 }^\infty 
\big\| t^{ s } A^{ \beta } ( t + A )^{ - \beta } x \big\|^q \, \frac{ d t }{
t } \right\}^{ 1 / q }, \quad x \in B^{ s, A }_{ q, X }.
\end{align*}
In particular, for $1 \leq q \leq \infty$, the last equivalence
reveals that our Besov space $B^{ s, A }_{ q, X }$ 
coincides with the space $D^s_q ( A )$ due to H. Komatsu
\cite{Komatsu1967}. 
\item[(ii)] Let $s > 0$. Fix $k \in \mathbb{Z}$ and $\alpha, \beta \in \mathbb{C}_+$ such that $s < \RE \beta$. By Lemma \ref{l:Besov type-inhomogeneous-continuity}, 
\begin{align*}
x \in B_{ q, X }^{ s, A } & \Leftrightarrow \left\{ \int_{ 2^k }^\infty \big\| t^{ s + \alpha } A^{ \beta } ( t + A )^{ - \alpha - \beta } x \big\|^q \, \frac{ d t }{ t } \right\}^{ 1 / q } < \infty \\
& \Leftrightarrow \left\{ \int_{ 0 }^\infty \big\| t^{ s + \alpha } A^{ \beta } ( t + A
)^{ - \alpha - \beta } x \big\|^q \, \frac{ d t }{ t } \right\}^{ 1 / q } < \infty.
\end{align*}
The last equivalence gives an alternative characterization of abstract Besov spaces. 
In particular, for $1 \le q \le \infty$, it can be seen that $x \in ( X, D ( A^\alpha ) )_{ \theta, q }$ if and only if  
\begin{align*}
\int_0^\infty \big\| t^{ \RE \beta - ( 1 - \theta ) \RE \alpha } A^\alpha ( t + A )^{ - \beta } x \big\|^q \, \frac{ d t }{ t } < \infty 
\end{align*}
by observing the fact that $D^{ \theta \RE \alpha }_q ( A ) = ( X, D ( A^\alpha ) )_{ \theta, q }$. Also, see \cite[Section 7.3]{Haase2005}.
\item[(iii)] Fix $k = 0$, $\alpha = l $ and $\beta = m $, where $l$ and $m$ are the least non-negative integers such that $- l < s < m$. By Lemmas \ref{l:independence of klm-inhomogeneous} and
\ref{l:Besov type-inhomogeneous-continuity}, the Besov space $B^{ s,
A }_{ q, X }$ coincides with the completion of $X$ with respect 
to the quasi-norm
\begin{align*}
\| x \|^c_{ B^{ s, A }_{ q, X } } = \| ( 1 + A )^{ - l } x \| 
+ \left\{ \int_{ 1 }^\infty \big\| t^{ s + l } A^{ m } ( t + A
)^{ - m - l } x \big\|^q \, \frac{ d t }{ t } \right\}^{ 1 / q }.
\end{align*}
In particular, for $1 \leq q \leq \infty$, the Besov space $B^{ s,
A }_{ q, X }$ coincides with the space $D^s_q ( A )$
introduced by T. Kobayashi and T. Muramatu \cite{Kobayashi and
Muramatu1992}.
\end{itemize}
\end{rem}

\smallskip


\subsection{Continuous embedding}

By $X_1 \hookrightarrow X_2$ we denote that the quasi-normed space $X_1$
is continuously embedded into the quasi-normed space $X_2$, i.e.,
$X_1 \subset X_2$ and $\| x \|_{ X_2 } \lesssim \| x \|_{ X_1 }$ for
each $x \in X_1$.

Motivated by the continuous embedding of the classical Besov spaces on $\mathbb{R}^n$ (see \cite[P.47, (5) and (7) and P.244, Section 5.2.5 Miscellaneous Properties]{TriebelM1983-2010}), we have the following continuous embedding for abstrac Besov spaces.

\begin{prop}\label{p:embedding theorem}
Let $0 < q \leq \infty$ and $s \in \mathbb{R}$. The following 
statements hold.
\begin{itemize}
\item[(i)] $B^{ s, A }_{ q, X } \hookrightarrow B^{ s, A }_{ q_1, X }$ for
all $0 < q \leq q_1 \leq \infty$.
\item[(ii)] $B^{ s_1, A }_{ p, X } \hookrightarrow B^{ s, A }_{ q, X }$
for all $- \infty < s < s_1 < \infty$ and $0 < p \leq \infty$.
\item[(iii)] $\dot{B}^{ s, A }_{ q, X } \hookrightarrow \dot{B}^{ s, A }_{
q_1, X }$ for all $0 < q \leq q_1 \leq \infty$.
\end{itemize}
\end{prop}

\begin{proof}
(i) Let $q \leq q_1 \leq \infty$. Fix $\tilde{x} \in B^{ s, A }_{ q, X }$ and let $\alpha, \beta \in \mathbb{C}_+^*$ such that $- \RE \alpha < s < \RE \beta$. Since $R^{ s, A }_{ q, X } ( 0, \alpha, \beta )$ is dense in $B^{ s, A }_{ q, X } = B^{ s, A }_{ q, X } ( 0, \alpha, \beta )$, there is a sequence $\{ x_n \} \subset R^{ s, A }_{ q, X } ( 0, \alpha, \beta )$ such that $x_n \to \tilde{x}$ with respect to $\interleave \cdot \interleave_{ B^{ s, A }_{ q, X } ( 0, \alpha, \beta ) }$, the equivalent quasi-norm on $B^{ s, A }_{ q, X } ( 0, \alpha, \beta )$ given by (\ref{e:quasinorm-p-subadditivity}). 
From (\ref{e:quasinorm-p-subadditivity}), it can be seen that 
\begin{align*}
\interleave x \interleave_{ B^{ s, A }_{ q_1, X } ( 0, \alpha, \beta ) } \le \interleave x  \interleave_{ B^{ s, A }_{ q, X } ( 0, \alpha, \beta ) }, \quad x \in R^{ s, A }_{ q, X } ( 0, \alpha, \beta ),
\end{align*}
by observing that $| x |_{ R^{ s, A }_{ q_1, X } ( 0, \alpha, \beta ) } \le | x |_{ R^{ s, A }_{ q, X } ( 0, \alpha, \beta ) }$
due to the monotonicity of the $\ell_q$ spaces. This implies that the sequence $\{
x_n \}$ is also Cauchy with respect to $\interleave \cdot \interleave_{ B^{ s, A }_{ q_1, X } ( 0, \alpha, \beta ) }$, and therefore $x_n \to \tilde{x}^*$ in $B^{ s, A }_{ q_1, X } ( 0, \alpha, \beta )$ for some (unique) $\tilde{x}^* \in B^{ s, A }_{ q_1, X } ( 0, \alpha, \beta )$.
Furthermore, it can be seen that the mapping $\tau: \tilde{x} \mapsto \tilde{x}^*$ is injective and bounded from $B^{ s, A }_{ q, X } ( 0, \alpha, \beta )$ to $B^{ s, A }_{ q_1, X } ( 0, \alpha, \beta )$ by observing that
\begin{align*}
\interleave \tilde{x}^* \interleave_{ B^{ s, A }_{ q_1, X } ( 0, \alpha, \beta ) } = {} & \lim_{ n \to \infty } \interleave x_n \interleave_{ B^{ s, A }_{ q_1, X } ( 0, \alpha, \beta ) } \\
\le {} & \lim_{ n \to \infty } \interleave x_n \interleave_{ B^{ s, A }_{ q, X } ( 0, \alpha, \beta ) } = \interleave \tilde{x} \interleave_{ B^{ s, A }_{ q, X } ( 0, \alpha, \beta ) }.
\end{align*}
Therefore, $B^{ s, A }_{ q, X } ( 0, \alpha, \beta )$ is continuously embedded into $B^{ s, A }_{ q_1, X } ( 0, \alpha, \beta )$.

(ii) Let $s < s_1 < \infty$. First we verify the conclusion in the case
$p = q$. To this end, fix $\alpha, \beta \in \mathbb{C}_+^*$ such
that $- \RE \alpha < s < s_1 < \RE \beta$ and let $\tilde{x} \in B^{ s_1, A }_{ q, X } ( 0, \alpha, \beta )$. Then there is a sequence $\{ x_n \} \subset R^{ s_1, A }_{ q, X } ( 0, \alpha, \beta )$ such that $x_n \to \tilde{x}$ in $B^{ s_1, A }_{ q, X } ( 0, \alpha, \beta )$. 
For each $x \in R^{ s_1, A }_{ q, X } ( 0, \alpha, \beta )$, it can be seen that
\begin{align*}
| x |_{ R^{ s, A }_{ q, X } ( 0, \alpha, \beta ) } 
& = \left\{ \sum_{ i = 0 }^\infty \big\| 2^{ i ( s - s_1 ) } 2^{ i ( s_1 + \alpha ) } A^\beta ( 2^{ i } + A )^{ - \alpha - \beta } x \big\|^q \right\}^{ 1 / q } \\
& \leq \left\{ \sum_{ i = 0 }^\infty \big\| 2^{ i ( s_1 + \alpha ) } A^\beta ( 2^{ i } + A )^{ - \alpha - \beta } x \big\|^q \right\}^{ 1 / q } = | x |_{ R^{ s_1, A }_{ q, X } ( 0, \alpha, \beta ) } 
\end{align*}
by observing that $s - s_1 < 0$.
Thus, by using (\ref{e:quasinorm-p-subadditivity}) we have  
\begin{align*}
\interleave x \interleave_{ B^{ s, A }_{ q, X } ( 0, \alpha, \beta ) } \le \interleave x  \interleave_{ B^{ s, A }_{ q, X } ( 0, \alpha, \beta ) }, \quad x \in R^{ s_1, A }_{ q, X } ( 0, \alpha, \beta ),
\end{align*}
This implies that $\{ x_n \}$ is also a Cauchy sequence with respect
to $\interleave \cdot \interleave_{ B^{ s, A }_{ q, X } ( 0, \alpha, \beta ) }$, and hence, there
is a unique $\tilde{x}^* \in B^{ s, A }_{ q, X } ( 0, \alpha, \beta )$ such that $x_n \to \tilde{x}^*$ in $B^{ s, A }_{ q, X } ( 0, \alpha, \beta )$.
It can be seen that the mapping $\tilde{x} \mapsto \tilde{x}^*$ is injective and bounded from $B^{ s_1, A }_{ q, X } ( 0, \alpha, \beta )$ to $B^{ s, A }_{ q, X } ( 0, \alpha, \beta )$ by observing that 
\begin{align*}
\interleave x \interleave_{ B^{ s, A }_{ q, X } ( 0, \alpha, \beta ) } = {} & \lim_{ n \rightarrow \infty } \interleave x_n \interleave_{ B^{ s, A }_{ q, X } ( 0, \alpha, \beta ) } \\
\leq {} & \lim_{ n \rightarrow \infty } \interleave x \interleave_{ B^{ s_1, A }_{ q, X } ( 0, \alpha, \beta ) } = \interleave x \interleave_{ B^{ s_1, A }_{ q, X } ( 0, \alpha, \beta ) }.
\end{align*}
Therefore, $B^{ s_1, A }_{ q, X } ( 0, \alpha, \beta )$ is continuously embedded into $B^{ s, A }_{ q, X } ( 0, \alpha, \beta )$.

Next we verify that $B^{ s_1, A }_{ q, X } \hookrightarrow B^{ s, A
}_{ p, X }$ for $q \leq p \leq \infty$. Indeed, by the discussion above and (i), it is clear that 
\begin{align*}
	B^{ s_1, A }_{ q, X } \hookrightarrow B^{ s, A }_{ q, X }
	\hookrightarrow B^{ s, A }_{ p, X }.
\end{align*}
It remains to verify the desired embedding for $0 < p < q$. Fix $\alpha, \beta \in \mathbb{C}_+^*$ such that $- \RE \alpha < s < s_1 < \RE \beta$. Analogous to the discussion above, it suffices to verify that 
\begin{align}\label{e:embedding theorem-proof}
| x |_{ R^{ s, A }_{ p, X } ( 0, \alpha, \beta ) } \lesssim | x |_{ R^{ s_1, A }_{ q, X } ( 0, \alpha, \beta ) }, \quad x \in R^{ s_1, A }_{ q, X } ( 0, \alpha, \beta ).
\end{align}
Indeed, it is clear that
\begin{align*}
| x |_{ R^{ s, A }_{ p, X } ( 0, \alpha, \beta ) } 
& = \left\{ \sum_{ i = 0 }^\infty 2^{ i ( s - s_1 ) p } \big\| 2^{ i s_1 } 2^{ i \alpha } A^\beta ( 2^{ i } + A )^{ - \alpha - \beta } x \big\|^p \right\}^{ 1 / p }.
\end{align*}
Applying the H\"{o}lder inequality with respect to the index $q / p > 1$ yields
\begin{align*}
| x |_{ R^{ s, A }_{ p, X } ( 0, \alpha, \beta ) } \leq {} & \Bigg\{ \bigg[ \sum_{ i = 0 }^\infty 2^{ i ( s - s_1 ) p ( q / p )' } \bigg]^{ 1 / ( q / p )' } \Bigg\}^{ 1 / p } \\
\cdot {} & \Bigg\{ \bigg[ \sum_{ i = 0 }^\infty \big\| 2^{ i s_1 } 2^{ i \alpha } A^\beta ( 2^{ i } + A )^{ - \alpha - \beta } x \big\|^{ p ( q / p ) } \bigg]^{ 1 / ( q / p ) } \Bigg\}^{ 1 / p } \\
= {} & C \, | x |_{ R^{ s_1, A }_{ q, X } ( 0, \alpha, \beta ) },
\end{align*}
where $C = [ \frac{ 1 }{ 1 - 2^{ ( s - s_1 ) q p / ( q - p ) } } ]^{ 1 - p / q }$.
Thus, we have verified (\ref{e:embedding theorem-proof}), and therefore $B^{ s_1, A }_{ q, X } ( 0, \alpha, \beta )$ is continuously embedded into $B^{ s, A }_{ q, X } ( 0, \alpha, \beta )$.


(iii) The statement is a direct consequence of the monotonicity of the $\ell_q$ spaces. Indeed, let $q < q_1 \leq \infty$, and fix $x \in \dot{B}^{ s, A }_{ q, X }$. Then there is a sequence $\{ x_n \} \subset R^{ s, A }_{ q, X }$ such that $\{ x_n \}$ is Cauchy with respect to $\interleave \cdot \interleave_{ \dot{B}^{ s, A }_{ q, X } }$, an equivalent quasi-norm on $\dot{B}^{ s, A }_{ q, X }$ given by (\ref{e:quasinorm-p-subadditivity}).
Note that
\begin{align*}
\interleave x \interleave_{ \dot{B}^{ s, A }_{ q_1, X } } = \lim_{ n \rightarrow \infty } \interleave x_n \interleave_{ \dot{R}^{ s, A }_{ q_1, X } } \leq \lim_{ n \rightarrow \infty } \interleave x_n \interleave_{ \dot{R}^{ s, A }_{ q, X } } = \interleave x \interleave_{ \dot{B}^{ s, A }_{ q, X } }
\end{align*}
due to the monotonicity of the sequence spaces $\ell_q$, from which the desired
conclusion follows immediately. The proof is complete.
\end{proof}

\smallskip


\subsection{Translation invariance}

Another property of interest for inhomogeneous Besov spaces
associated with non-negative operators is that they are translation invariant
with respect to the underlying operators. More precisely, we have the following proposition. 

\begin{prop}\label{p:translation invariance-inhomogeneous}
Let $0 < q \leq \infty$ and $s \in \mathbb{R}$. Then $B^{ s, A +
\epsilon }_{ q, X } = B^{ s, A }_{ q, X }$ in the sense of equivalent quasi-norms for each $\epsilon > 0$.
\end{prop}

\begin{proof}
Let $\epsilon > 0$, take $\alpha, \beta \in \mathbb{C}_+^*$ such that $- \RE \alpha < s < \RE \beta$, and fix $k \in \mathbb{Z}$ such that $\epsilon < 2^k$. It suffices to verify that 
\begin{align}\label{p:translation invariance-inhomogeneous-proof2}
	\| x \|_{ B^{ s, A }_{ q, X } ( k, \alpha, \beta ) } \lesssim \| x \|_{ B^{ s, A + \epsilon }_{ q, X } ( k, \alpha, \beta ) }
\end{align}
and that 
\begin{align}\label{p:translation invariance-inhomogeneous-proof4}
\| x \|_{ B^{ s, A + \epsilon }_{ q, X } ( k, \alpha, \beta ) } \lesssim \| x \|_{ B^{ s, A }_{ q, X } ( k, \alpha, \beta ) }.	
\end{align}

First we verify the equivalence  
\begin{align}\label{p:translation invariance-inhomogeneous-proof1}
\| ( 2^k + A )^{ - \alpha } x \| \simeq \| ( 2^k + A + \epsilon )^{ - \alpha } x \|.
\end{align}
Indeed, applying (\ref{e:uniform boundedness compositions-C})$^*$ with $c = 1$ yields   
\begin{align*}
\| ( 2^k + A + \epsilon )^{ - \alpha } x \| & \le \big\| ( 2^k + A )^\alpha ( 2^k + \epsilon + A )^{ - \alpha } \big\| \cdot \big\| ( 2^k + A )^{ - \alpha } x \big\| \\
& \le C_{ \alpha, n } ( L_A + M_A )^n \big\| ( 2^k + A )^{ - \alpha } x \big\|.
\end{align*}
Thanks to the fact that $\epsilon < 2^k$, applying (\ref{e:uniform boundedness compositions-C})$^*$ with $c = 2$ yields  
\begin{align*}
\| ( 2^k + A )^{ - \alpha } x \| & \le \big\| ( 2^k + \epsilon + A )^\alpha ( 2^k + A )^{ - \alpha } \big\| \cdot \big\| ( 2^k + A + \epsilon )^{ - \alpha } x \big\| \\
& \le C_{ \alpha, n } ( L_A + 2 M_A )^n \big\| ( 2^k + A + \epsilon )^{ - \alpha } x \big\|,
\end{align*}
where $C_{ \alpha, n }$ is given by (\ref{e:uniform boundedness compositions}) with $\RE \alpha < n \in \mathbb{N}$, and $M_A$ and $L_A$ are the constants given by (\ref{e:non-negative constant}) and (\ref{e:non-negative constant 2}), respectively. Thus, we have verified (\ref{p:translation invariance-inhomogeneous-proof1}).

Next we verify (\ref{p:translation invariance-inhomogeneous-proof2}).
Thanks to (\ref{p:translation invariance-inhomogeneous-proof1}), it suffices to verify that
\begin{eqnarray}\label{p:translation invariance-inhomogeneous-proof3}
| x |_{ R^{ s, A }_{ q, X } ( k, \alpha, \beta ) } \lesssim | x |_{ R^{ s, A + \epsilon }_{ q, X } ( k, \alpha, \beta ) }.
\end{eqnarray}
We merely verify (\ref{p:translation invariance-inhomogeneous-proof3}) in the case $0 < q < \infty$ while the case when $q = \infty$ can be verified analogously. To this end, applying (\ref{e:uniform boundedness compositions-L})$^*$ to $ A^\beta ( \epsilon + A )^{ - \beta }$ and applying (\ref{e:uniform boundedness compositions-C})$^*$ with $c = 2$ to $ ( 2^j + \epsilon + A )^{ \alpha + \beta } ( 2^j + A )^{ - ( \alpha + \beta ) }$ yields 
\begin{align*} 
| x |_{ R^{ s, A }_{ q, X } ( k, \alpha, \beta ) } = {} & \Bigg\{ \sum_{ j = k }^\infty
\big\| 2^{ j s } 2^{ j \alpha } A^\beta ( 2^j + A )^{ - \alpha - \beta } x \big\|^q \Bigg\}^{ 1 / q } \\
\le {} & \Bigg\{ \sum_{ j = k }^\infty \big\| A^\beta ( \epsilon + A )^{ - \beta } \big\| \cdot \big\| ( 2^j + \epsilon + A )^{ \alpha + \beta } ( 2^j + A )^{ - \alpha - \beta } \big\| \\
& \cdot \big\| 2^{ j ( s + \alpha ) } ( A + \epsilon )^\beta ( 2^j + A + \epsilon )^{ - \alpha - \beta } x \big\|^q \Bigg\}^{ 1 / q }  \\
 \le {} & C_{ \beta, m } C_{ \alpha + \beta, N } L_A^m ( L_A + 2 M_A )^{ N } | x |_{ R^{ s, A + \epsilon }_{ q, X } ( k, \alpha, \beta ) },	
\end{align*}
where $C_{ \beta, m }$ and $C_{ \alpha + \beta, N }$ are given by (\ref{e:uniform boundedness compositions}) with $\RE \alpha < m \in \mathbb{N}$ and $\RE ( \alpha + \beta ) < N \in \mathbb{N}$, respectively, and $M_A$ and $L_A$ are the non-negativity constants of $A$ given by (\ref{e:non-negative constant}) and (\ref{e:non-negative constant 2}), respectively. Thus, we have verified (\ref{p:translation invariance-inhomogeneous-proof3}).

It remains to verify (\ref{p:translation invariance-inhomogeneous-proof4}).
By Lemma \ref{l:independence of klm-inhomogeneous} (the independence of Besov quasi-norms with respect to the index $\beta$), it suffices to verify (\ref{p:translation invariance-inhomogeneous-proof4}) for the case in which $\beta = m \in \mathbb{Z}$. Moreover, we may assume that $0 < q < \infty$ and the case when $q = \infty$ can be verified analogously. Note that
\begin{align*}
	& | x |_{ R^{ s, A + \epsilon }_{ q, X } ( k, \alpha, m ) } \\
 = {} & \Bigg\{ \sum_{ j = k }^\infty
	\big\| 2^{ j s } 2^{ j \alpha } ( A + \epsilon )^m ( 2^j + A + \epsilon )^{ - \alpha - m } x \big\|^q \Bigg\}^{ 1 / q } \\
 \leq {} & \Bigg\{ \sum_{ j = k }^\infty \bigg[ \sum_{ i = 0 }^m { m
		\choose i } \epsilon^{ m - i } \big\| 2^{ j s } 2^{ j \alpha } A^i ( 2^j + A + \epsilon )^{ - \alpha - m } x \big\| \bigg]^q \Bigg\}^{ 1 / q } \\
\le {} & \Bigg\{ \sum_{ j = k }^\infty \bigg[ \sum_{ i = 0 }^m { m
		\choose i } \epsilon^{ m - i } \big\| 2^{ j ( m - i ) } (
	2^j + A )^{ i + \alpha } ( 2^j + \epsilon + A )^{ - \alpha - m } \big\| \\
{} & \cdot \big\| 2^{ j ( s + i - m ) } 2^{ j \alpha } A^i ( 2^j + A )^{ - i -
		\alpha } x \big\| \bigg]^q \Bigg\}^{ 1 / q }.
\end{align*}
Decomposing $\alpha + m = ( m - i ) + ( i + \alpha )$ yields 
\begin{align}\label{p:translation invariance-inhomogeneous-proof5}
\begin{split}
| x |_{ R^{ s, A + \epsilon }_{ q, X } ( k, \alpha, m ) } \le {} & \Bigg\{ \sum_{ j = k }^\infty \bigg[ \sum_{ i = 0 }^m { m \choose i } \epsilon^{ m - i } \| 2^{ j ( m - i ) } ( 2^j + \epsilon + A )^{ - ( m - i ) } \| \\
{} & \cdot \big\| ( 2^j + A )^{ i + \alpha } ( 2^j + \epsilon + A )^{ - i - \alpha } \big\| \\
{} & \cdot \big\| 2^{ j ( s + i - m ) } 2^{ j \alpha } A^i ( 2^j + A )^{ - i -
		\alpha } x \big\| \bigg]^q \Bigg\}^{ 1 / q }.
\end{split}
\end{align}
Applying of (\ref{e:non-negative constant}) yields 
\begin{align}\label{p:translation invariance-inhomogeneous-proof6}
\| 2^{ j ( m - i ) } ( 2^j + \epsilon + A )^{ - ( m - i ) } \| \le \| 2^{ j } ( 2^j + \epsilon + A )^{ - 1 } \|^{ m - i } \le M_A^{ m - i }.
\end{align}
Applying (\ref{e:uniform boundedness compositions-C})$^*$ with $c = 1$ yields
\begin{align}\label{p:translation invariance-inhomogeneous-proof7}
\begin{split}
\| ( 2^j + A )^{ i + \alpha } ( 2^j + \epsilon + A )^{ - i - \alpha } \| \le C_{ i + \alpha, N } ( L_A + M_A )^N,
\end{split}
\end{align}
where $C_{ i + \alpha, N }$ is given by (\ref{e:uniform boundedness compositions}) with $\RE ( i + \alpha ) < N \in \mathbb{N}$.
Applying (\ref{p:translation invariance-inhomogeneous-proof5}), (\ref{p:translation invariance-inhomogeneous-proof6}) and (\ref{p:translation invariance-inhomogeneous-proof7}) yields 
\begin{align}\label{p:translation invariance-inhomogeneous-proof8}
\begin{split}
 | x |_{ R^{ s, A + \epsilon }_{ q, X } ( k, \alpha, m ) } \le {} & \Bigg\{ \sum_{ j = k }^\infty \bigg[ \sum_{ i = 0 }^m { m \choose i } ( \epsilon M_A )^{ m - i } C_{ i + \alpha, N } ( L_A + M_A )^N \\ 
{} & \cdot \big\| 2^{ j ( s + i - m ) } 2^{ j \alpha } A^i ( 2^j + A )^{ - i - \alpha } x \big\| \bigg]^q \Bigg\}^{ 1 / q } \\
 \le {} & C \sum_{ i = 0 }^m { m \choose i } ( \epsilon M_A )^{ m - i } C_{ i + \alpha, N } ( L_A + M_A )^N | x |_{ R^{ s + i - m, A }_{ q, X, k, \alpha, i } },
\end{split}
\end{align}
where $C = 1$ due to the Minkowski inequality
for $1 < q < \infty$ and $C = 2^{ m ( 1 / q - 1 ) }$ due to (\ref{e: q-inequality}) for $0 < q \le 1$. Applying (\ref{p:translation invariance-inhomogeneous-proof1}) and (\ref{p:translation invariance-inhomogeneous-proof8}) yields 
\begin{align*}
\| x \|_{ B^{ s, A + \epsilon }_{ q, X } ( k, \alpha, m ) } 
\lesssim {} & \| ( 2^k + A )^{ - \alpha } x \| \\
+ {} & \sum_{ i = 0 }^m { m \choose i } ( \epsilon M )^{ m - i } C_{ i + \alpha, N } ( L_A + M_A )^N | x |_{ R^{ s + i - m, A }_{ q, X } ( k, \alpha, i ) } \\
\lesssim {} & \sum_{ i = 0 }^m { m \choose i } ( \epsilon M )^{ m - i } C_{ i + \alpha, N } ( L_A + M_A )^N \|  x \|_{ B^{ s + i - m, A }_{ q, X } ( k, \alpha, i ) }.
\end{align*}
Applying Lemma \ref{l:independence of klm-inhomogeneous} (the independence of Besov quasi-norms with respect to the index $\beta$) yields
\begin{align*}
\| x \|_{ B^{ s, A + \epsilon }_{ q, X } ( k, \alpha, m ) } & \lesssim \sum_{ i = 0 }^m { m \choose i } ( \epsilon M )^{ m - i } C_{ i + \alpha, N } ( L_A + M_A )^N \|  x \|_{ B^{ s + i - m, A }_{ q, X } ( k, \alpha, m ) }.
\end{align*}
Since $s + i - m \leq s$, by Proposition \ref{p:embedding theorem} (ii) we obtain that
\begin{align*}
\| x \|_{ B^{ s, A + \epsilon }_{ q, X } ( k, \alpha, m ) } & \lesssim \sum_{ i = 0 }^m { m \choose i } ( \epsilon M )^{ m - i } C_{ i + \alpha, N } ( L_A + M_A )^N \cdot \|  x \|_{ B^{ s, A }_{ q, X } ( k, \alpha, m ) }.
\end{align*}
Thus, we have verified (\ref{p:translation invariance-inhomogeneous-proof4}). The proof is complete.
\end{proof}

%
%
%

\smallskip


\section{Fractional powers on Besov spaces}\label{S:Fractional powers}

This section is devoted to the connections between Besov spaces associated with non-negative operators and fractional powers of the underlying operators, including the lifting property, smoothness reiteration and interpolation. As in the last two sections, $( X, \left\| \cdot \right\| )$ is a Banach space and $A$ is a non-negative operator on $X$.

\smallskip


\subsection{Lifting property}

Recall that $x \in B^{ s, A }_{ q, X }$ if and only if $x \in D ( A^\gamma )$ and $A^\gamma x \in B^{ s - \RE \gamma , A }_{ q, X }$ whenever $0 < \RE \gamma < s$ and $1 \leq q \leq \infty$ (see \cite[Theorem 2.6]{Komatsu1967}). 
We can show further that the fractional power $A^\gamma$ is continuous from $B^{ s, A }_{ q, X }$ to $B^{ s - \RE \gamma, A }_{ q, X }$, even if $0 < q < 1$. Furthermore, it can be verified that $\left\|A^\gamma \cdot \right\|_{ B^{ s - \RE \gamma, A }_{ q, X } }$ is indeed an equivalent quasi-norm on $B^{ s, A }_{ q, X }$ whenever $A$ is positive. 

\begin{lem}\label{l:lifting property-positive}
Let $0 < \RE \gamma < s$ and $0 < q \leq \infty$. The following statements hold.
\begin{itemize}
\item[(i)] $x \in B^{ s, A }_{ q, X }$ if and only if $x \in D ( A^\gamma )$ and $A^\gamma x \in B^{ s - \RE \gamma , A }_{ q, X }$.
\item[(ii)] $A^\gamma$ is continuous from $B^{ s, A }_{ q, X }$ to $B^{ s - \RE \gamma, A }_{ q, X }$.
\end{itemize}
\end{lem}

\begin{proof}
(i) Sufficiency. Suppose that $x \in D ( A^\gamma )$ and $A^\gamma x \in B^{ s - \RE \gamma, A }_{ q, X }$. Fix $\alpha, \beta \geq 0$ such that $- \alpha < s < \beta$. It is clear that $x \in B^{ s, A }_{ q, X }$ due to the fact that 
\begin{align*}
\left\{ \sum_{ i = 0 }^\infty \big\| 2^{ i ( s + \alpha ) } A^{ \beta } ( 2^i + A )^{ - ( \alpha + \beta ) } x \big\|^q \right\}^{ 1 / q } 
\leq | A^\gamma x |_{ R^{ s, A }_{ q, X } ( 0, \alpha + \gamma, \beta - \gamma ) } < \infty.
\end{align*}

Necessity. Let $x \in B^{ s, A }_{ q, X }$. First we verify that $x \in D ( A^\gamma )$. To this end, fix $\alpha, \beta, \gamma \geq 0$ such that 
$- \alpha < \RE \gamma < \gamma' < s < \beta$. In order to verify that $x \in D ( A^\gamma )$, it suffices to show that
\begin{align}\label{e:lifting property-positive-proof-absolute convergence}
\int_0^\infty \big\| t^{ \gamma' } t^\alpha A^\beta ( t + A )^{ - \alpha - \beta } x \big\| \, \frac{ d t }{ t } < \infty. 
\end{align}
More precisely, if this is the case, then $x \in D ( A^{ \gamma' } )$ by Corollary \ref{c:representation of fractional powers-beta-absolute convergence}, while $D ( A^{ \gamma' } ) \subset D ( A^\gamma )$ due to the fact that $\gamma' > \RE \gamma$, and hence, $x \in D ( A^\gamma )$. 
We now verify (\ref{e:lifting property-positive-proof-absolute convergence}). On the one hand, from (\ref{e:uniform boundedness compositions-M})$^*$ and (\ref{e:uniform boundedness compositions-L})$^*$ it follows that  
\begin{align}\label{e:lifting property-positive-proof-convergence-0-1}
\int_0^1 \big\| t^{ \gamma' } t^\alpha A^\beta ( t + A )^{ - \alpha - \beta } x \big\| \, \frac{ d t }{ t } \leq C_{ \alpha, n } C_{ \beta, m } M_A^n L_A^m \left\| x \right\| / \gamma' < \infty
\end{align}
where $\alpha < n \in \mathbb{N}$, $\beta < m \in \mathbb{N}$ and $C_{ \alpha, m }$ and $C_{ \beta, n }$ are given by (\ref{e:uniform boundedness compositions}). On the other hand, it can be verified that 
\begin{align}\label{e:lifting property-positive-proof-convergence-1-infty}
\int_1^\infty \big\| t^{ \gamma' } t^\alpha A^\beta ( t + A )^{ - \alpha - \beta } x \big\| \, \frac{ d t }{ t } \lesssim \left| x \right|_{ R^{ s, A }_{ q, X } ( 0, \alpha, \beta ) } < \infty,
\end{align}
and hence, the desired convergence (\ref{e:lifting property-positive-proof-absolute convergence}) follows from (\ref{e:lifting property-positive-proof-convergence-0-1}) and (\ref{e:lifting property-positive-proof-convergence-1-infty}), immediately. It remains to verify (\ref{e:lifting property-positive-proof-convergence-1-infty}). Indeed, rewriting $\int_1^\infty = \sum_{ i = 0 }^{ \infty } \int_{ 2^i }^{ 2^{ i +1 } }$ and applying (\ref{e:resolvent estimate-discrete}) yields  
\begin{align}\label{e:lifting property-positive-proof-convergence-1-infty1}
\int_1^\infty \big\| t^{ \gamma' } t^\alpha A^\beta ( t + A )^{ - \alpha - \beta } x \big\| \, \frac{ d t }{ t } \leq C \sum_{ i = 0 }^{ \infty } \big\| 2^{ i ( \alpha + \gamma' ) } A^\beta ( 2^i + A )^{ - \alpha - \beta } x \big\|,
\end{align}
where $C = C_{ \alpha + \beta, N } ( L_A + 2 M_A )^N ( 2^{ \gamma' + \alpha } - 1 ) / ( \gamma' +  \alpha )$ with $ \alpha + \beta \leq N \in \mathbb{N}$ and $C_{ \alpha + \beta, N }$ given by (\ref{e:uniform boundedness compositions}). Furthermore, we can show that 
\begin{align}\label{e:lifting property-positive-proof-convergence-1-infty2}
J := \sum_{ i = 0 }^{ \infty } \big\| 2^{ i ( \alpha + \gamma' ) } A^\beta ( 2^i + A )^{ - \alpha - \beta } x \big\| \lesssim \left| x \right|_{ R^{ s, A }_{ q, X } ( 0, \alpha, \beta ) }, 
\end{align}
and hence, the desired inequality (\ref{e:lifting property-positive-proof-convergence-1-infty}) follows from (\ref{e:lifting property-positive-proof-convergence-1-infty1}) and (\ref{e:lifting property-positive-proof-convergence-1-infty2}), immediately. 
Thus, it is sufficient to verify (\ref{e:lifting property-positive-proof-convergence-1-infty2}).
Indeed, if $0 < q \leq 1$, applying the fact that $ \gamma' < s $ yields  
\begin{align}\label{e:lifting property-positive-proof-convergence-1-infty2-1}
\begin{split}
J
& \leq \sum_{ i = 0 }^{ \infty } \big\| 2^{ i ( s + \alpha ) } A^\beta ( 2^i + A )^{ - \alpha - \beta } x \big\| \\
& \leq \left\{ \sum_{ i = 0 }^{ \infty } \big\| 2^{ i ( s + \alpha ) } A^\beta ( 2^i + A )^{ - \alpha - \beta } x \big\|^q \right\}^{ 1 / q } = | x |_{ R^{ s, A }_{ q, X } ( 0, \alpha, \beta ) },
\end{split}
\end{align}
where the last inequality follows from (\ref{e: q-inequality-series<1}).
If $1 < q < \infty$, applying the H\"{o}lder inequality yields 
\begin{align}\label{e:lifting property-positive-proof-convergence-1-infty2-2}
\begin{split}
J 
& \leq \left\{ \sum_{ i = 0 }^{ \infty } 2^{ i ( \gamma' - s ) q' } \right\}^{ 1 / q'} \left\{ \sum_{ i = 0 }^{ \infty } \big\| 2^{ i ( s + \alpha ) } A^\beta ( 2^i + A )^{ - \alpha - \beta } x \big\|^q \right\}^{ 1 / q } \\
& = \bigg[ \frac{ 1 }{ 1 - 2^{ ( \gamma' - s ) q' } } \bigg]^{ 1 / q' } | x |_{ R^{ s, A }_{ q, X } ( 0, \alpha, \beta ) }.
\end{split}
\end{align}
And if $q = \infty$, it is clear that 
\begin{align}\label{e:lifting property-positive-proof-convergence-1-infty2-3}
J 
& \leq \left[ \sup_{ j \geq 0 } 2^{ j s } \big\| 2^{ j \alpha } A^\beta ( 2^j + A )^{ - \alpha - \beta } x \big\| \right] \cdot \left[ \sum_{ i = 0 }^{ \infty } 2^{ i ( \gamma' - s ) } \right] = \frac{ | x |_{ R^{ s, A }_{ q, X } ( 0, \alpha, \beta ) } }{ 1 - 2^{ \gamma' - s } }.
\end{align}
By (\ref{e:lifting property-positive-proof-convergence-1-infty2-1}), (\ref{e:lifting property-positive-proof-convergence-1-infty2-2}) and (\ref{e:lifting property-positive-proof-convergence-1-infty2-3}), we obtain (\ref{e:lifting property-positive-proof-convergence-1-infty2}), immediately.
Thus, we have verified that $x \in D ( A^\gamma)$. 

Next we verify that $A^\gamma x \in B^{ s - \gamma, A }_{ q, X }$. To this end, fix $\alpha, \beta \geq 0$ such that $ \alpha > 2 \RE \gamma$ and $- ( \alpha - 2 \RE \gamma ) < s < \beta$. On the one hand, from (\ref{e:uniform boundedness compositions-L})$^*$ it follows that  
\begin{align}\label{e:lifting property-positive-proof-inhomogeneous term boundedness}
\big\| ( 2^k + A )^{ - ( \alpha - \gamma ) } A^\gamma x \big\| 
\leq C_{ \gamma, n } L_A^n \big\| ( 2^k + A )^{ - ( \alpha - 2 \gamma ) } x \big\|
\end{align}
where $0 < \RE \gamma < n \in \mathbb{N}$ and $C_{ \gamma, n }$ and $L_A$ are given by (\ref{e:uniform boundedness compositions}) and (\ref{e:non-negative constant 2}), respectively. On the other hand, it is clear that 
\begin{align*}
& \left\{ \sum_{ i = 0 }^\infty \big\| 2^{ i ( s - \gamma ) } \cdot 2^{ i ( \alpha - \gamma ) } A^{ \beta - \gamma } ( 2^i + A )^{ - ( \alpha + \beta - 2 \gamma ) } A^\gamma x \big\|^q \right\}^{ 1 / q } 
= | x |_{ R^{ s, A }_{ q, X } ( 0, \alpha - 2 \gamma, \beta ) }.
\end{align*}
Thus, we have $A^\gamma x \in B^{ s - \gamma, A }_{ q, X }$ by the definition of inhomogeneous Besov spaces (see Definition \ref{d:inhomogeneous Besov spaces} and Remark \ref{r:independence of klm-inhomogeneous} above).

(ii) Fix $\alpha, \beta \geq 0$ such that $ \alpha > 2 \RE \gamma$ and $- ( \alpha - 2 \RE \gamma ) < s < \beta$. From (\ref{e:lifting property-positive-proof-inhomogeneous term boundedness}) it follows that 
\begin{align}\label{e:lefting property-positive-inequality}
\begin{split}
\| A^\gamma x \|_{ B^{ s - \gamma, A }_{ q, X } } & = \big\| ( 2^k + A )^{ - ( \alpha - \gamma ) } A^\gamma x \big\| + | A^\gamma x |_{ R^{ s, A }_{ q, X } ( 0, \alpha - \gamma, \beta - \gamma ) } \\
& \leq C_{ \gamma, n } L_A^n \big\| ( 2^k + A )^{ - ( \alpha - 2 \gamma ) } x \big\| + | x |_{ R^{ s, A }_{ q, X } ( 0, \alpha - 2 \gamma, \beta ) } \lesssim \| x \|_{ B^{ s, A }_{ q, X } }.
\end{split}
\end{align}
Thus, $A^\gamma$ is continuous from $B^{ s, A }_{ q, X }$ to $B^{ s - \RE \gamma, A }_{ q, X }$. The proof is complete.
\end{proof}

Thanks to Lemma \ref{l:lifting property-positive} above, we can now give an alternative equivalent quasi-norm on $B^{ s, A }_{ q, X }$ by the use of $\left\| A^\gamma \cdot \right\|_{ B^{ s - \RE \gamma, A }_{ q, X } }$ whenever $A$ is positive on $X$.

\begin{thm}\label{t:lifting property-negative}
Let $A$ be positive on $X$, and let $s > 0$ and $0 < q \leq \infty$. The following statements hold. 
\begin{itemize}
\item[(i)] 
$A^{ - \gamma }$ is continuous from $B^{ s, A }_{ q, X }$ to $B^{ s + \RE \gamma, A }_{ q, X }$ for each $\gamma \in \mathbb{C}_+$. 
\item[(ii)] $\| A^{ - \gamma} \cdot \|_{ B^{ s + \RE \gamma, A }_{ q, X } }$ is an equivalent quasi-norm on $B^{ s, A }_{ q, X }$ for each $\gamma \in \mathbb{C}_+$.
\item[(iii)] $\| A^{ \gamma} \cdot \|_{ B^{ s - \RE \gamma, A }_{ q, X } }$ is an equivalent quasi-norm on $B^{ s, A }_{ q, X }$ for each $\gamma \in \mathbb{C}_+$ satisfying $0 < \RE \gamma < s$.
\end{itemize}
\end{thm}

\begin{proof}
(i) Let $x \in B^{ s, A }_{ q, X }$ and $\gamma \in \mathbb{C}_+$. Fix $k \in \mathbb{Z}$, $\alpha, \beta \geq 0$ such that $\beta > s$ and $ \alpha > \RE \gamma$ (so that $- \alpha < s < \beta$ and $- ( \alpha - \RE \gamma ) < s + \RE \gamma < \beta + \RE \gamma$). By the definition of the quasi-norm on Besov space $B^{ s, A }_{ q, X }$, we have 
\begin{eqnarray}\label{e:lefting property-negative-inequality}
\begin{split}
\| A^{ - \gamma } x \|_{ B^{ s + \RE \gamma, A }_{ q, X } ( k, \alpha - \gamma, \beta + \gamma ) } 
= {} & \big\| ( 1 + 2^k A^{ - \gamma } ) ( 2^k + A )^{ - \alpha } x \big\| \\
 + {} & \left\{ \sum_{ i = k }^\infty \big\| 2^{ i ( s + \alpha ) } A^{ \beta } ( 2^i + A )^{ - \alpha - \beta } x \big\| \Big]^q \right\}^{ 1 / q } \\
\leq {}& C \| x \|_{ B^{ s, A }_{ q, X } ( k, \alpha, \beta ) },
\end{split}
\end{eqnarray}
where $C = \big( 1 + 2^k \| A^{ - \gamma } \| \big)$. Thus, $A^{ - \gamma }$ is continuous from $B^{ s, A }_{ q, X }$ to $B^{ s + \RE \gamma, A }_{ q, X }$.

(ii) The statement is a simple consequence of (\ref{e:lefting property-positive-inequality}) and (\ref{e:lefting property-negative-inequality}). Indeed, 
\begin{align*}
\| A^{ - \gamma } x \|_{ B^{ s + \RE \gamma, A }_{ q, X } } \lesssim \| x \|_{ B^{ s, A }_{ q, X } } = \| A^\gamma A^{ - \gamma } x \|_{ B^{ s + \RE \gamma - \RE \gamma, A }_{ q, X } } \lesssim  \| A^{ - \gamma } x \|_{ B^{ s + \RE \gamma, A }_{ q, X } }.
\end{align*}

(iii) The statement is a simple consequence of (\ref{e:lefting property-negative-inequality}) and (\ref{e:lefting property-positive-inequality}). Indeed, 
\begin{align*}
\| A^{ \gamma } x \|_{ B^{ s - \RE \gamma, A }_{ q, X } } \lesssim \| x \|_{ B^{ s, A }_{ q, X } } = \| A^{ - \gamma } A^{ \gamma } x \|_{ B^{ s - \RE \gamma + \RE \gamma, A }_{ q, X } } \lesssim  \| A^{ \gamma } x \|_{ B^{ s - \RE \gamma, A }_{ q, X } }.
\end{align*}
The proof is complete.
\end{proof}


%

\smallskip


\subsection{Smoothness reiteration}

The smoothness of abstract Besov spaces has a closely relation to the fractional powers of the underlying non-negative operators. By using the so-called smoothness reiteration, we can give more explicit estimates of the quasi-norms on abstract Besov spaces.

We will give the main result of this subsection in Theorem \ref{t:smoothness reiteration} below based on the following two lemmas.  

\begin{lem}\label{l:estimate of cos}
Let $0 < \alpha < 1$ and let 
\begin{align*}
f ( t ) := 1 + 2 t \cos \pi \alpha + t^2, \quad t > 0.
\end{align*}
Then, for given $t > 0$, we have
\begin{align}\label{e:f(t) less f(u)}
f ( u ) \leq K_\alpha f ( t ), \quad t / 2 \le u \le t,
\end{align}
where 
\begin{align*}
K_\alpha = \begin{cases}
            1, \quad & 0 < \alpha \le 1 / 2, \\
            \frac{ 1 }{ 4 } ( 1 + \frac{ 3 }{ \sin^2 \pi \alpha } ), \quad & 1/ 2 < \alpha < 1.
           \end{cases}
\end{align*}
\end{lem} 

\begin{proof}
The statement is obvious when $0 < \alpha \le 1 / 2$ since $f$ is increasing on $( 0, \infty )$. Thus, it remains to verify (\ref{e:f(t) less f(u)}) in the case $1 / 2 < \alpha < 1$. Fix $t > 0$ and write 
\begin{align*}
g_t ( u ) := \frac{ f ( u ) }{ f ( t ) }, \quad u > 0.
\end{align*}
It suffices to verify that 
\begin{align}\label{e: estimate of g}
\sup_{ t / 2 \le u \le t } g_t ( u ) \le \frac{ 1 }{ 4 } \Big( 1 + \frac{ 3 }{ \sin^2 \pi \alpha } \Big).
\end{align}
To this end, applying $g_t' ( u ) = 0$ yields $u = - \cos \pi \alpha$, so that $g_t$ is decreasing on $( 0, - \cos \pi \alpha )$ while increasing on $( - \cos \pi \alpha, \infty )$. If $t / 2 \geq - \cos \pi \alpha$, we have
\begin{align*}
\sup_{ t / 2 \le u \le t } g_t ( u ) \le g_t ( t ) = 1 < \frac{ 1 }{ 4 } \Big( 1 + \frac{ 3 }{ \sin^2 \pi \alpha } \Big) 
\end{align*}
since $g_t$ is increasing on $( t / 2, t )$.
If $t / 2 < - \cos \pi \alpha \le t$, we have 
\begin{align}\label{l:estimate of cos-proof-1}
\sup_{ t / 2 \le u \le t } g_t ( u ) \le \max \{ g_t ( t ), g_t ( t / 2 ) \} = \max \{ 1, g_t ( t / 2 ) \} 
\end{align}
since $g_t$ is decreasing on $( t / 2, - \cos \pi \alpha )$ and increasing on $( - \cos \pi \alpha, t )$. Observe that 
\begin{align}\label{l:estimate of cos-proof-2}
\begin{split}
g_t ( t / 2 ) & 
= \frac{ 1 }{ 4 } \bigg( 1 + \frac{ 3 + 2 t \cos \pi \alpha }{ 1 + 2 t \cos \pi \alpha + t^2 } \bigg) \\
& \le \frac{ 1 }{ 4 } \bigg( 1 + \frac{ 3 }{ 1 + 2 t \cos \pi \alpha + t^2 } \bigg)
 \le \frac{ 1 }{ 4 } \bigg( 1 + \frac{ 3 }{ \sin^2 \pi \alpha } \bigg). 
\end{split}
\end{align}
due to the fact that $\cos \pi \alpha < 0$. By using (\ref{l:estimate of cos-proof-1}) and (\ref{l:estimate of cos-proof-2}) we obtain (\ref{e: estimate of g}), immediately. And if $t < - \cos \pi \alpha$, it can be seen from (\ref{l:estimate of cos-proof-2}) that
\begin{align*}
\sup_{ t / 2 \le u \le t } g_t ( u ) \le g_t ( t / 2 ) \leq \frac{ 1 }{ 4 } \Big( 1 + \frac{ 3 }{ \sin^2 \pi \alpha } \Big)
\end{align*}
since $g_t$ is decreasing on $( t / 2, t )$. Thus, we have verified (\ref{e: estimate of g}). The proof is complete.
\end{proof}

\begin{lem}\label{l:lp-boundedness-T}
Let $0 < q \le \infty$ and let $0 < s, \alpha < 1$. Define $T : \ell_q \mapsto \ell_q$ by 
\begin{align*}
T { \bf a } := { \bf b }, \quad { \bf a } = \{ a_j \} \in \ell_q,
\end{align*}
where ${ \bf b } = \{ b_j \}$ with 
\begin{align*}
b_j = \sum_{ i = - \infty }^{ \infty } \frac{ ( 2^{ - j } 2^{ i \alpha } )^{ 1 - s } }{ 1 + 2 ( 2^{ - j } 2^{ i \alpha } ) \cos \pi \alpha + ( 2^{ - j } 2^{ i \alpha } )^{ 2 } } \cdot a_i, \quad j \in \mathbb{Z}.
\end{align*}
Then $T$ is bounded on $\ell_q$.
\end{lem}

\begin{proof}
First we verify the statement in the case $0 < q \le 1$. Fix ${\bf a } = \{ a_j \} \in \ell_q$. It follows from (\ref{e: q-inequality-series<1}) that  
\begin{align*}
\| T { \bf a } \|_{ l_q } & \leq \Bigg\{ \sum_{ j = - \infty }^\infty \sum_{ i = - \infty }^{ \infty } \bigg| \frac{ ( 2^{ - j } 2^{ i \alpha } )^{ 1 - s } a_i }{ 1 + 2 ( 2^{ - j } 2^{ i \alpha } ) \cos \pi \alpha + ( 2^{ - j } 2^{ i \alpha } )^{ 2 } } \bigg|^q \Bigg\}^{ 1 / q } \\
& = \Bigg\{ \sum_{ i = - \infty }^{ \infty } | a_i |^q \sum_{ j = - \infty }^\infty \bigg[ \frac{ ( 2^{ - j } 2^{ i \alpha } )^{ 1 - s } }{ 1 + 2 ( 2^{ - j } 2^{ i \alpha } ) \cos \pi \alpha + ( 2^{ - j } 2^{ i \alpha } )^{ 2 } } \bigg]^q \Bigg\}^{ \frac{ 1 }{ q } }.
\end{align*}
In order to verify the $\ell_q$-boundedness of $T$, it is sufficient to verify that    
\begin{align}\label{e:smoothness reiteration-proof-CT-q<1}
\begin{split}
& \sup_{ i \in \mathbb{ Z } } \sum_{ j = - \infty }^\infty \Big[ \frac{ ( 2^{ - j } 2^{ i \alpha } )^{ 1 - s } }{ 1 + 2 ( 2^{ - j } 2^{ i \alpha } ) \cos \pi \alpha + ( 2^{ - j } 2^{ i \alpha } )^2 } \Big]^q < \infty.
\end{split}
\end{align}
To this end, by using the identity $\int_{ 2^j }^{ 2^{ j + 1 } } 2^{ - j } \, d \lambda = 1$ for $j \in \mathbb{Z}$
we have  
\begin{align*} 
& \sup_{ i \in \mathbb{ Z } } \sum_{ j = - \infty }^\infty \Big[ \frac{ ( 2^{ - j } 2^{ i \alpha } )^{ 1 - s } }{ 1 + 2 ( 2^{ - j } 2^{ i \alpha } ) \cos \pi \alpha + ( 2^{ - j } 2^{ i \alpha } )^2 } \Big]^q \\
= {} & \sup_{ i \in \mathbb{ Z } } \sum_{ j = - \infty }^\infty \int_{ 2^j }^{ 2^{ j + 1 } } \Big[ \frac{ ( 2^{ - j } 2^{ i \alpha } )^{ 1 - s } }{ 1 + 2 ( 2^{ - j } 2^{ i \alpha } ) \cos \pi \alpha + ( 2^{ - j } 2^{ i \alpha } )^2 } \Big]^q \cdot 2^{ - j } \, d \lambda \\
\leq {} & 
2^{ ( 1 - s ) q + 1 } \sup_{ i \in \mathbb{ Z } } \sum_{ j = - \infty }^\infty \int_{ 2^j }^{ 2^{ j + 1 } } \Big[ \frac{ ( \lambda^{ - 1 } 2^{ i \alpha } )^{ 1 - s } }{ 1 + 2 ( 2^{ - j } 2^{ i \alpha } ) \cos \pi \alpha + ( 2^{ - j } 2^{ i \alpha } )^2 } \Big]^q \, \frac{ d \lambda }{ \lambda }, 
\end{align*}
where the last inequality follows from the fact that $\lambda \leq 2^{ j + 1 }$.
Since $2^{ - j } 2^{ i \alpha } / 2 \le \lambda^{ - 1 } 2^{ i \alpha } < 2^{ - j } 2^{ i \alpha }$, applying (\ref{e:f(t) less f(u)}) with $t = 2^{ - j } 2^{ i \alpha }$ to the integrand of the last integrel yields 
\begin{align*} 
\begin{split}
& \sup_{ i \in \mathbb{ Z } } \sum_{ j = - \infty }^\infty \Big[ \frac{ ( 2^{ - j } 2^{ i \alpha } )^{ 1 - s } }{ 1 + 2 ( 2^{ - j } 2^{ i \alpha } ) \cos \pi \alpha + ( 2^{ - j } 2^{ i \alpha } )^2 } \Big]^q \\
\le {} & K_\alpha 2^{ ( 1 - s ) q + 1 } \sup_{ i \in \mathbb{ Z } } \int_{ 0 }^{ \infty } \Big[ \frac{ ( \lambda^{ - 1 } 2^{ i \alpha } )^{ 1 - s } }{ 1 + 2 ( \lambda^{ - 1 } 2^{ i \alpha } ) \cos \pi \alpha + ( \lambda^{ - 1 } 2^{ i \alpha } )^2 } \Big]^q \, \frac{ d \lambda }{ \lambda } \\
= {} & K_\alpha 2^{ ( 1 - s ) q + 1 } \int_0^\infty \bigg( \frac{ \mu^{ 1 - s } }{ 1 + 2 \mu \cos \pi \alpha + \mu^2 } \bigg)^q \, \frac{ d \mu }{ \mu },
\end{split}
\end{align*}
from which the desired (\ref{e:smoothness reiteration-proof-CT-q<1}) follows immediately since the last integral converges due to the fact that $0 < s < 1$. More precisely, 
\begin{align*} 
\begin{split}
\| T { \bf a } \|_{ \ell_q } \le C \| { \bf a } \|_{ \ell_q },
\end{split}
\end{align*}
where $C = ( J K_\alpha )^{ 1 / q } 2^{ 1 - s + 1 / q }$ with
\begin{align*} 
\begin{split}
J = \int_0^\infty \bigg( \frac{ \mu^{ 1 - s } }{ 1 + 2 \mu \cos \pi \alpha + \mu^2 } \bigg)^q \, \frac{ d \mu }{ \mu }.
\end{split}
\end{align*}
Thus, we have verified the $\ell_q$-boundedness of $T$ for $0 < q \le 1$.

Next we verify the $\ell_q$-boundedness of $T$ for $q = \infty$. Indeed, 
analogous to (\ref{e:smoothness reiteration-proof-CT-q<1}), by using $\int_{ 2^{ i - 1 } }^{ 2^{ i } } 2^{ 1 - i } \, d \lambda = 1$
and the estimate (\ref{e:f(t) less f(u)}) we can conclude that  
\begin{align}\label{e:smoothness reiteration-proof-CT-q=infty}
\begin{split}
& \sup_{ j \in \mathbb{ Z } } \sum_{ i = - \infty }^{ \infty } \frac{ ( 2^{ - j } 2^{ i \alpha } )^{ 1 - s } }{ 1 + 2 ( 2^{ - j } 2^{ i \alpha } ) \cos \pi \alpha + ( 2^{ - j } 2^{ i \alpha } )^2 } \\
\le {} & \alpha^{ - 1 } K_\alpha 2^{ \alpha ( 1 - s ) + 1 } 
\int_0^\infty \frac{ \mu^{ 1 - s } }{ 1 + 2 \mu \cos \pi \alpha + \mu^2 } \, \frac{ d \mu }{ \mu }  := D_T < \infty,
\end{split}
\end{align}
where $K_\alpha$ is the constant given in (\ref{e:f(t) less f(u)}). This implies that 
\begin{align*}
\| T { \bf a } \|_{ l_\infty } = \sup_{ j \in \mathbb{ Z } } | b_j |
\leq D_T \cdot \sup_{ i \in \mathbb{ Z } } | a_i | = D_T \| { \bf a } \|_{ l_\infty },
\end{align*}
the desired $\ell_\infty$-boundedness of $T$.

Finally, the $\ell_q$-boundedness of $T$ for $1 < q < \infty$ is a direct consequence of the well-known Marcinkiewicz interpolation theorem due to the $\ell_1$-boundedness and $\ell_\infty$-boundedness of the (sub)linear operator $T$. 
The proof is complete.
\end{proof}

We can now establish the smoothness reiteration for Besov spaces $B^{ s, A }_{ q, X }$, which was first discussed by H. Komatsu \cite[Theorem 3.2]{Komatsu1967} via integral transforms in the case $1 \le q \le \infty$ and subsequently described by M. Haase \cite[Corollary 7.3]{Haase2005} by using functional calculi in the case $1 \le q \le \infty$ as well. Here we give a unified approach to the smoothness reiteration for Besov spaces with a full range of $0 < q \le \infty$. Moreover, recall that a closed linear operator on a Banach spaces is non-negative if and only if it is sectorial. 

\begin{thm}\label{t:smoothness reiteration}
Let $s > 0$, $0 < q \leq \infty$ and $0 \le \omega < \pi$, and let $A$ be sectorial of angle $\omega$. Then 
\begin{align*}
B^{ s, A^\alpha }_{ q, X } = B^{ s \alpha, A }_{ q, X }, \quad 0 < \alpha < \pi / \omega,
\end{align*}
in the sense of equivalent quasi-norms.
\end{thm}

\begin{proof}
It suffices to verify the statement in the case $0 < \alpha < 1$, because otherwise we have $A = ( A^\alpha )^{ 1 / \alpha }$ with $0 < 1 / \alpha < 1$. Moreover, thanks to Lemma \ref{l:lifting property-positive} (i) we may assume that $s$ is sufficiently small, say $0 < s < 1$.
Clearly, it needs merely to verify that 
\begin{align}\label{e:smoothness reiteration-proof-Apower<A}
\| x \|_{ B^{ s, A^\alpha }_{ q, X } } \lesssim \| x \|_{ B^{ s \alpha, A }_{ q, X } }  
\end{align}
and that  
  \begin{align}\label{e:smoothness reiteration-proof-A<Apower}
\| x \|_{ B^{ s \alpha, A }_{ q, X } } \lesssim \| x \|_{ B^{ s, A^\alpha }_{ q, X } }. 
\end{align}

First we verify (\ref{e:smoothness reiteration-proof-Apower<A}). To this end, let $x \in B^{ s \alpha, A }_{ q, X }$. 
From (\ref{e:fractional powers-resolvent-L}) it follows that
\begin{align*}
| x |_{ R^{ s, A^\alpha }_{ q, X } ( 0, 0, \alpha ) } = {} & \Bigg\{ \sum_{ j = 0 }^\infty \big\| 2^{ j s } A^\alpha ( 2^j + A^\alpha )^{ - 1 } x \big\|^q \Bigg\}^{ 1 / q } \\
 \leq {} & \frac{ 1 }{ \Gamma ( \alpha ) \Gamma ( 1 - \alpha ) } \Bigg\{ \sum_{ j = 0 }^\infty \bigg[ \int_0^\infty \frac{ 2^{ j ( s + 1 ) } \tau^\alpha \big\| A ( \tau + A )^{ - 1 } x \big\| }{ 2^{ 2 j } + 2 \cdot 2^j \tau^\alpha \cos \pi \alpha + \tau^{ 2 \alpha } } \, \frac{ d \tau }{ \tau } \bigg]^q \Bigg\}^{ 1 / q }. 
\end{align*}
Rewriting $\int_0^\infty = \sum_{ i = - \infty }^\infty \int^{ 2^{ i + 1 } }_{ 2^i }$
and applying the estimate (\ref{e:resolvent estimate-discrete}) to $\left\| A ( \tau + A )^{ - 1 } x \right\|$
yields 
\begin{align*}
| x |_{ R^{ s, A^\alpha }_{ q, X } ( 0, 0, \alpha ) } \leq {} & C \Bigg\{ \sum_{ j = 0 }^\infty \bigg[ \sum_{ i = - \infty }^\infty \int^{ 2^{ i + 1 } }_{ 2^i } \frac{ 2^{ j ( s + 1 ) } \big\| 2^{ i \alpha } A ( 2^i + A )^{ - 1 } x \big\| }{ 2^{ 2 j } + 2 \cdot 2^j \tau^\alpha \cos \pi \alpha + \tau^{ 2 \alpha } } \, \frac{ d \tau }{ \tau } \bigg]^q \Bigg\}^{ 1 / q } \\
= {} & C
\Bigg\{ \sum_{ j = 0 }^\infty \bigg[ \sum_{ i = - \infty }^\infty \int^{ 2^{ i + 1 } }_{ 2^i } \frac{ ( 2^{ - j } 2^{ i \alpha } )^{ 1 - s } \big\| 2^{ i s \alpha } A ( 2^i + A )^{ - 1 } x \big\| }{ 1 + 2 ( 2^{ - j } \tau^\alpha ) \cos \pi \alpha + ( 2^{ - j } \tau^{ \alpha } )^2 } \, \frac{ d \tau }{ \tau } \bigg]^q \Bigg\}^{ 1 / q }, 
\end{align*}
where $C = \frac{ 2^\alpha (L_A + M_A )^2 }{ \Gamma ( \alpha ) \Gamma ( 1 - \alpha ) }$ with $M_A$ and $L_A$ given by (\ref{e:non-negative constant}) and (\ref{e:non-negative constant 2}), respectively. 
Since $2^{ - j } \tau^\alpha / 2 < 2^{ - j } 2^{ i \alpha } \le 2^{ - j } \tau^\alpha$, applying (\ref{e:f(t) less f(u)}) with $t = 2^{ - j } \tau^{ \alpha }$ yields 
\begin{align}\label{e:smoothness reiteration-proof1}
| x |_{ R^{ s, A^\alpha }_{ q, X } ( 0, 0, \alpha ) }
\le \frac{ 2^\alpha K_\alpha (L_A + M_A )^2 }{ \Gamma ( \alpha ) \Gamma ( 1 - \alpha ) } \cdot J, 
\end{align}
where $K_\alpha$ is the constant given in (\ref{e:f(t) less f(u)}) and 
\begin{align*}
J = \Bigg\{ \sum_{ j = 0 }^\infty \bigg[ \sum_{ i = - \infty }^{ \infty } \frac{ ( 2^{ - j } 2^{ i \alpha } )^{ 1 - s } \big\| 2^{ i \alpha s } A ( 2^i + A )^{ - 1 } x \big\| }{ 1 + 2 ( 2^{ - j } 2^{ i \alpha } ) \cos \pi \alpha + ( 2^{ - j } 2^{ i \alpha } )^{ 2 } } \bigg]^q \Bigg\}^{ 1 / q }.  
\end{align*}
It can be verified that  
\begin{align}\label{e:smoothness reiteration-proof2}
J \lesssim \| x \|_{ B^{ s \alpha, A }_{ q, X } }.
\end{align}
Indeed, the statement is a simple consequence of the $\ell_q$-boundedness of $T$ as shown in Lemma \ref{l:lp-boundedness-T} above. More precisely, applying ${\bf a} = \{ a_i \}$ with  
\begin{align*}
a_i = \big\| 2^{ i s \alpha } A ( 2^i + A )^{ - 1 } x \big\|, \quad i \in \mathbb{Z},
\end{align*}
in Lemma \ref{l:lp-boundedness-T} yields 
\begin{align*}
J = \| { \bf b } \|_{ \ell_q } \lesssim \| { \bf a } \|_{ \ell_q } \lesssim \| x \|_{ B^{ s \alpha, A }_{ q, X } },
\end{align*}
which is the desired inequality (\ref{e:smoothness reiteration-proof2}). Moreover, from (\ref{e:smoothness reiteration-proof1}) and (\ref{e:smoothness reiteration-proof2}) it follows that 
\begin{align*}
\| x \|_{ B^{ s, A^\alpha }_{ q, X } ( 0, 0, \alpha ) } = \| x \| + | x |_{ R^{ s, A^\alpha }_{ q, X } ( 0, 0, \alpha ) } \lesssim \| x \|_{ B^{ s \alpha, A }_{ q, X } }.
\end{align*}
Thus, we have verified (\ref{e:smoothness reiteration-proof-Apower<A}).

Next we verify (\ref{e:smoothness reiteration-proof-A<Apower}).    
It suffices to verify (\ref{e:smoothness reiteration-proof-A<Apower}) in the case $\alpha = 1 / m$ with an odd integer $m$, i.e.,
\begin{align}\label{e:smoothness reiteration-proof5}
\| x \|_{ B^{ s / m, A }_{ q, X } } \lesssim \| x \|_{ B^{ s, A^{ 1 / m } }_{ q, X } },  
\end{align}
since in a general case $0 < \alpha < 1$ we can fix an odd integer $m$ large enough such that $1 / m < \alpha$ and, by using (\ref{e:smoothness reiteration-proof5}) and (\ref{e:smoothness reiteration-proof-Apower<A}), obtain that 
\begin{align*}
\| x \|_{ B^{ s \alpha, A }_{ q, X } } & = \| x \|_{ B^{ s \alpha m / m, A }_{ q, X } } \lesssim 
\| x \|_{ B^{ s \alpha m, A^{ 1 / m } }_{ q, X } } \\
& = \| x \|_{ B^{ s \alpha m, ( A^\alpha )^{ 1 / ( \alpha m ) } }_{ q, X } } 
\lesssim \| x \|_{ B^{ s, A^\alpha }_{ q, X } }
\lesssim \| x \|_{ B^{ s \alpha, A }_{ q, X } },
\end{align*}
which is the desired inequality (\ref{e:smoothness reiteration-proof-A<Apower}). 

It remains to verify (\ref{e:smoothness reiteration-proof5}). 
Let $m$ be an odd integer and let $x \in B^{ s, A^{ 1 / m } }_{ q, X }$. If $0 < q < \infty$, by Lemma \ref{l:Besov type-inhomogeneous-continuity} we have 
 \begin{align*}
| x |_{ R^{ s / m, A }_{ q, X } ( k, 0, 1 ) } & \simeq | x |^c_{ R^{ s / m, A }_{ q, X } ( k, 0, 1 ) } = \left\{ \int_{ 2^k }^\infty \big\| t^{ s / m } A ( t + A )^{ - 1 } x \big\|^q \, \frac{ d t }{ t } \right\}^{ 1 / q } \\
& \leq \left\{ \int_{ 0 }^\infty \big\| t^{ s / m } A ( t + A )^{ - 1 } x \big\|^q \, \frac{ d t }{ t } \right\}^{ 1 / q } \\
& = \bigg\{ \int_{ 0 }^\infty \bigg\| t^{ s / m } \prod_{ i = 1 }^{ m } A^{ 1 / m } ( z_i t^{ 1 / m } + A^{ 1 / m } )^{ - 1 } x \bigg\|^q \, \frac{ d t }{ t } \bigg\}^{ 1 / q },
\end{align*}
where $z_i's$ are all roots of $( - z )^m = - 1$ with $z_1 = 1$. Note that the family $\{ A^{ 1 / m } ( z_i t + A^{ 1 / m } )^{ - 1 } \}_{ t > 0 } $ is uniformly bounded for each $i = 2, 3, \cdots, m$. By Lemma \ref{l:Besov type-inhomogeneous-continuity} again we have 
\begin{align*}
| x |_{ R^{ s / m, A }_{ q, X } } & \lesssim \bigg\{ \int_{ 0 }^\infty \big\| t^{ s / m } A^{ 1 / m } ( t^{ 1 / m } + A^{ 1 / m } )^{ - 1 } x \big\|^q \, \frac{ d t }{ t } \bigg\}^{ 1 / q } \\
& = m^{ 1 / q } | x |^c_{ R^{ s, A^{ 1 / m } }_{ q, X } } \simeq | x |_{ R^{ s, A^{ 1 / m } }_{ q, X } },
\end{align*}
from which the desired inequality (\ref{e:smoothness reiteration-proof5}) follows immediately. 
And if $q = \infty$, by Lemma \ref{l:Besov type-inhomogeneous-continuity} we also have
\begin{align*}
\| x \|_{ B^{ s / m, A }_{ \infty, X } } & = \| x \| + \sup_{ j \geq 0 } \big\| 2^{ j s / m } A ( 2^j + A )^{ - 1 } x \big\| \\
& \lesssim \| x \| + \sup_{ \lambda > 0 } \big\| \lambda^{ s / m } A ( \lambda + A )^{ - 1 } x \big\| \\
& \lesssim \| x \| + \sup_{ \lambda > 0 } \big\| \lambda^{ s / m } A^{ 1 / m } ( \lambda^{ 1 / m } + A^{ 1 / m } )^{ - 1 } x \big\| \\
& = \| x \| + \sup_{ \mu > 0 } \big\| \mu^{ s } A^{ 1 / m } ( \mu + A^{ 1 / m } )^{ - 1 } x \big\| = \| x \|_{ B^{ s, A^{ 1 / m } }_{ \infty, X } },
\end{align*}
which is the desired inequality (\ref{e:smoothness reiteration-proof5}). 
The proof is complete. 
\end{proof}

The following continuous embedding of Besov spaces associated with fractional powers operators is a direct consequence of Theorem \ref{t:smoothness reiteration}.

\begin{cor}
Let $A$ be non-negative on $X$ with sectorial angle $\theta$, and let $s > 0$ and $0 < q \leq \infty$. Then, for $0 < \alpha \leq \beta < \pi / \theta$,
\begin{align*}
B^{ s, A^\beta }_{ q, X } \hookrightarrow B^{ s, A^\alpha }_{ q, X }.
\end{align*}
\end{cor}

\smallskip


\subsection{Interpolation spaces}

Let $( X_0, \left\| \cdot \right\|_{ X_0 } )$ and $( X_1, \left\| \cdot \right\|_{ X_1 } )$ be a couple of quasi-normed spaces continuously embedded into a topological vector space. Recall that, for each $x \in X_0 + X_1$, the $K$-functional $K ( t, x )$ is defined by 
\begin{align*}
K ( t, x ) := \inf_{ x_0 + x_1 = x } \big( \| x_0 \|_{ X_0 } + t \| x_1 \|_{ X_1 } \big), \quad x_i \in X_i, i = 0, 1,
\end{align*} 
with $0 < t < \infty$. It is easy to see that $K ( t, \cdot )$ is a quasi-norm on $X_0 + X_1$ for given $t > 0$ and that $K ( \cdot, x )$ is a non-negative, increasing and concave function for each $x$ fixed.  

Let $0 < \theta < 1$ and $0 < q \leq \infty$. Recall that the interpolation space $( X_0, X_1 )_{ \theta, q }$ between $X_0$ and $X_1$ is given by 
\begin{align*}
( X_0, X_1 )_{ \theta, q } := \left\{ x \in X_0 + X_1: \int_0^\infty \Big( t^{ - \theta } K ( t, x ) \Big)^q \frac{ d t }{ t } < \infty \right\}
\end{align*} 
and that $( X_0, X_1 )_{ \theta, q }$ is a quasi-Banach space endowed with the quasi-norm  
\begin{align}\label{e:quasi-norms-interpolation spaces}
\| x \|_{ ( X_0, X_1 )_{ \theta, q } } := \left\{ \int_0^\infty \Big( t^{ - \theta } K ( t, x ) \Big)^q \frac{ d t }{ t } \right\}^{ 1 / q }. 
\end{align} 
Also, recall the following equivalent quasi-norm on interpolation spaces (see \cite[Theorem 5.1]{Holmstedt1970}): 
\begin{align}\label{e:equivalent quasi-norm of interpolation spaces}
\| x \|_{ ( X_0, X_1 )_{ \theta, q } } \simeq \| x \|^*_{ ( X_0, X_1 )_{ \theta, q } } := \inf_{ x_0 ( t ) + x_1 ( t ) \equiv x } ( B_0 + B_1 ),
\end{align} 
where 
\begin{align*}
B_0 & = \left\{ \int_0^\infty \Big( t^{ - \theta } \| x_0 ( t ) \|_{ X_0 } \Big)^q \frac{ d t }{ t } \right\}^{ 1 / q }, \\
B_1 & = \left\{ \int_0^\infty \Big( t^{ 1 - \theta } \| x_1 ( t ) \|_{ X_1 } \Big)^q \frac{ d t }{ t } \right\}^{ 1 / q }.
\end{align*} 

Now we give the main result of this subsection, which states that the abstract Besov spaces can be characterized by the interpolation spaces even in the case $0 < q < 1$.

\begin{thm}\label{p:interpolation}
Let $0 < q \leq \infty$, $0 < \theta < 1$ and $\alpha > 0$. Then $( X, D ( A^\alpha ) )_{ \theta, q } = B^{ \theta \alpha, A }_{ q, X }$ 
in the sense of equivalent quasi-norms.
\end{thm}

\begin{proof}
Thanks to Proposition \ref{p:translation invariance-inhomogeneous}, we may suppose that $0 \in \rho ( A )$ without loss of generality. It suffices to verify that 
\begin{align}\label{e:interpolation-proof-interpolation<Besov}
\| x \|_{ ( X, D ( A^\alpha ) )_{ \theta, q } } \lesssim \| x \|_{ B^{ \theta \alpha, A }_{ q, X } } 
\end{align} 
and that 
\begin{align}\label{e:interpolation-proof-Besov<interpolation}
\| x \|_{ B^{ \theta \alpha, A }_{ q, X } } \lesssim \| x \|_{ ( X, D ( A^\alpha ) )_{ \theta, q } }. 
\end{align}

First we verify (\ref{e:interpolation-proof-interpolation<Besov}). We merely give a proof of (\ref{e:interpolation-proof-interpolation<Besov}) in the case $0 < q < \infty$ and the case when $q = \infty$ can be verified analogously. To this end, write $X_0 := X$ and $X_1 = D ( A^\alpha )$ endowed with the graph norm, i.e., $\left\| x \right\|_{ X_1 } = \left\| A^\alpha x \right\|$ for $x \in D ( A^\alpha )$. Let $x \in B^{ \theta \alpha, A }_{ q, X }$ and write  
\begin{align*}
u ( \lambda ) := \frac{ \Gamma ( 2 \alpha ) }{ \Gamma ( \alpha ) \Gamma ( \alpha ) } \lambda^\alpha A^\alpha ( \lambda + A )^{ - 2 \alpha } x.
\end{align*}
Moreover, let $x_0$ and $x_1$ be two simple functions defined by 
\begin{align*}
x_0 ( t ) := \int_{ 2^j }^\infty u ( \lambda ) \frac{ d \lambda }{ \lambda }, \quad x_1 ( t ) := \int_0^{ 2^j } u ( \lambda ) \frac{ d \lambda }{ \lambda }, \quad 2^{ - ( j + 1 ) \alpha } < t \le 2^{ - j \alpha }, j \in \mathbb{Z}.
\end{align*} 
On the one hand, by Proposition \ref{p:inclusions-powers-inhomogeneous} it can be seen that $x \in D ( A^\epsilon )$ for $0 < \epsilon < \theta \alpha$. Since $0 \in \rho ( A )$ due to the hypothesis, from (\ref{r:homogeneous Calderon-regularity}) it follows that 
\begin{align*}
x = \frac{ \Gamma ( 2 \alpha ) }{ \Gamma ( \alpha ) \Gamma ( \alpha ) } \int_0^\infty \lambda^\alpha A^\alpha ( \lambda + A )^{ - 2 \alpha } x \, \frac{ d \lambda }{ \lambda } = x_0 ( t ) + x_1 ( t ), \quad t > 0. 
\end{align*}
On the other hand, by using the change of variable $t = s^{ - \alpha }$ and decomposition $\int_0^\infty = \sum_{ j \in \mathbb{Z} } \int_{ 2^j }^{ 2^{ j + 1 } }$ in (\ref{e:quasi-norms-interpolation spaces}), we can rewrite $\| x \|_{ ( X_0, X_1 )_{ \theta, q } }$ as 
\begin{align*}
\| x \|_{ ( X_0, X_1 )_{ \theta, q } } 
= \alpha^{ 1 / q } \Bigg\{ \sum_{ j = - \infty }^\infty \int_{ 2^j }^{ 2^{ j + 1 } } \Big( s^{ \theta \alpha } K ( s^{ - \alpha }, x ) \Big)^q \frac{ d s }{ s } \Bigg\}^{ 1 / q }. 
\end{align*}
Since $K ( \cdot, x )$ is increasing, we have 
\begin{align*}
\| x \|_{ ( X_0, X_1 )_{ \theta, q } } 
\leq \frac{ \alpha^{ 1 / q } ( 2^{ q \alpha } - 1 ) }{ q \theta \alpha } \Bigg\{ \sum_{ j = - \infty }^\infty \Big( 2^{ j \theta \alpha } K ( 2^{ - j \alpha }, x ) \Big)^q \Bigg\}^{ 1 / q }.
\end{align*} 
By the definition of $K ( t, x )$ we have 
\begin{align*}
\| x \|_{ ( X_0, X_1 )_{ \theta, q } }  
\leq C \Bigg\{ \sum_{ j = - \infty }^\infty \Big[ 2^{ j \theta \alpha } \Big( \big\| x_0 ( 2^{ - j \alpha } ) \big\| + 2^{ - j \alpha } \big\| A^\alpha x_1 ( 2^{ - j \alpha } ) \big\| \Big) \Big]^q \Bigg\}^{ 1 / q }, 
\end{align*} 
where $C = \frac{ \alpha^{ 1 / q } ( 2^{ q \alpha } - 1 ) }{ q \theta \alpha }$. By the classical inequality (\ref{e: q-inequality}) and the definitions of $x_0 ( t )$ and $x_1 ( t )$ we have  
\begin{align*}
\| x \|_{ ( X_0, X_1 )_{ \theta, q } } 
 \lesssim {} & \Bigg\{ \sum_{ j = - \infty }^\infty \bigg[ 2^{ j \theta \alpha } \sum_{ i = j }^\infty \int_{ 2^i }^{ 2^{ i + 1 } } \big\| \lambda^\alpha A^\alpha ( \lambda + A )^{ - 2 \alpha } x \big\| \frac{ d \lambda }{ \lambda } \bigg]^q \Bigg\}^{ 1 / q } \\
 + {} & \Bigg\{ \sum_{ j = - \infty }^\infty \bigg[ 2^{ j ( \theta - 1 ) \alpha } \sum_{ i = - \infty }^{ j - 1 } \int_{ 2^i }^{ 2^{ i + 1 } } \big\| \lambda^\alpha A^{ 2 \alpha } ( \lambda + A )^{ - 2 \alpha } x \big\| \frac{ d \lambda }{ \lambda } \bigg]^q \Bigg\}^{ 1 / q }.
\end{align*}
And by using the estimate (\ref{e:resolvent estimate-discrete}) with $c = 1$ we have further that 
\begin{align*}
\| x \|_{ ( X_0, X_1 )_{ \theta, q } } \lesssim J_1 + J_2, 
\end{align*} 
where 
\begin{align*}
J_1 & = \Bigg\{ \sum_{ j = - \infty }^\infty \bigg[ 2^{ j \theta \alpha } \sum_{ i = j }^\infty \big\| 2^{ i \alpha } A^\alpha ( 2^i + A )^{ - 2 \alpha } x \big\| \bigg]^q \Bigg\}^{ 1 / q }, \\
J_2 & = \Bigg\{ \sum_{ j = - \infty }^\infty \bigg[ 2^{ j ( \theta - 1 ) \alpha } \sum_{ i = - \infty }^{ j - 1 } \big\| 2^{ i \alpha } A^{ 2 \alpha } ( 2^i + A )^{ - 2 \alpha } x \big\| \bigg]^q \Bigg\}^{ 1 / q }. 
\end{align*} 
It can be verified that both $J_1$ and $J_2$ are dominated by $\| x \|_{ B^{ \theta \alpha, A }_{ q, X } }$.
Indeed, if $0 < q \leq 1$, by using (\ref{e: q-inequality}) we have 
\begin{align*}
J_1 & \leq \Bigg\{ \sum_{ j = - \infty }^\infty \sum_{ i = j }^\infty \big\| 2^{ j \theta \alpha } 2^{ i \alpha } A^\alpha ( 2^i + A )^{ - 2 \alpha } x \big\|^q \Bigg\}^{ 1 / q } \\
& = \Bigg\{ \sum_{ i = - \infty }^\infty \sum_{ j = - \infty }^i \big\| 2^{ j \theta \alpha } 2^{ i \alpha } A^\alpha ( 2^i + A )^{ - 2 \alpha } x \big\|^q \Bigg\}^{ 1 / q } \\
& = \bigg( \frac{ 1 }{ 1 - 2^{ - q \theta \alpha } } \bigg)^{ 1 / q } \Bigg\{ \sum_{ i = - \infty }^\infty \big\| 2^{ i \theta \alpha } 2^{ i \alpha } A^\alpha ( 2^i + A )^{ - 2 \alpha } x \big\|^q \Bigg\}^{ 1 / q } \\
& \simeq | x |_{ \dot{R}^{ \theta \alpha, A }_{ q, X, k, \alpha, \alpha } } \leq \| x \|_{ B^{ \theta \alpha, A }_{ q, X } },
\end{align*} 
where the last inequality follows from Remark \ref{r:inhomogeneous Besov spaces-equivalent quasi-norms-continuity} (i) above. Analogously, by using (\ref{e: q-inequality}) and Remark \ref{r:inhomogeneous Besov spaces-equivalent quasi-norms-continuity} (i) we also have 

\begin{align*}
J_2 & \leq \Bigg\{ \sum_{ j = - \infty }^\infty \sum_{ i = - \infty }^{ j - 1 } \big\| 2^{ j ( \theta - 1 ) \alpha } 2^{ i \alpha } A^{ 2 \alpha } ( 2^i + A )^{ - 2 \alpha } x \big\|^q \Bigg\}^{ 1 / q } \\
& = \Bigg\{ \sum_{ i = - \infty }^\infty \sum_{ j = i + 1 }^\infty \big\| 2^{ j ( \theta - 1 ) \alpha } 2^{ i \alpha } A^{ 2 \alpha } ( 2^i + A )^{ - 2 \alpha } x \big\|^q \Bigg\}^{ 1 / q } \\
& = \bigg[ \frac{ 1 }{ 2^{ ( 1 - \theta ) \alpha q } - 1 } \bigg]^{ 1 / q } \Bigg\{ \sum_{ i = - \infty }^\infty 2^{ i ( \theta - 1 ) \alpha q } \big\| 2^{ i \alpha } A^{ 2 \alpha } ( 2^i + A )^{ - 2 \alpha } x \big\|^q \Bigg\}^{ 1 / q } \\
& \simeq | x |_{ \dot{R}^{ \theta \alpha, A }_{ q, X, k, 0, 2 \alpha } } \leq \| x \|_{ B^{ \theta \alpha, A }_{ q, X } }.
\end{align*} 
If $1 < q < \infty$, choosing $\epsilon \in ( 0, \theta \alpha )$ and applying the H\"{o}lder inequality yields 
\begin{align*}
J_1 & = \Bigg\{ \sum_{ j = - \infty }^\infty 2^{ j \theta \alpha q } \bigg[ \sum_{ i = j }^\infty 2^{ - i \epsilon } \big\| 2^{ i \epsilon } 2^{ i \alpha } A^\alpha ( 2^i + A )^{ - 2 \alpha } x \big\| \bigg]^q \Bigg\}^{ 1 / q } \\
& \leq \Bigg\{ \sum_{ j = - \infty }^\infty 2^{ j \theta \alpha q } \bigg( \sum_{ i = j }^\infty 2^{ - i \epsilon q' } \bigg)^{ q / q' } \sum_{ i = j }^\infty \big\| 2^{ i \epsilon } 2^{ i \alpha } A^\alpha ( 2^i + A )^{ - 2 \alpha } x \big\|^q \Bigg\}^{ 1 / q } \\
& = \bigg( \frac{ 1 }{ 1 - 2^{ - \epsilon q' } } \bigg)^{ 1 / q' } \bigg[ \frac{ 1 }{ 1 - 2^{ ( \epsilon - \theta \alpha ) q } } \bigg]^{ 1 / q } \Bigg\{ \sum_{ i = - \infty }^\infty \big\| 2^{ i \theta \alpha } 2^{ i \alpha } A^\alpha ( 2^i + A )^{ - 2 \alpha } x \big\|^q \Bigg\}^{ 1 / q } \\
& \simeq | x |_{ \dot{R}^{ \theta \alpha, A }_{ q, X, k, \alpha, \alpha } } \leq \| x \|_{ B^{ \theta \alpha, A }_{ q, X } },
\end{align*} 
while choosing $0 < c < ( 1 - \theta ) \alpha$ and applying the H\"{o}lder inequality yields
\begin{align*}
J_2 = \Bigg\{ \sum_{ j = - \infty }^\infty 2^{ j ( \theta - 1 ) \alpha q } \sum_{ i = - \infty }^{ j - 1 } 2^{ i \epsilon } \big\| 2^{ - i \epsilon } 2^{ i \alpha } A^{ 2 \alpha } ( 2^i + A )^{ - 2 \alpha } x \big\| \bigg)^q \Bigg\}^{ 1 / q } \lesssim \| x \|_{ B^{ \theta \alpha, A }_{ q, X } }.
\end{align*} 
Thus, we have verified (\ref{e:interpolation-proof-interpolation<Besov}). 

Next we verify (\ref{e:interpolation-proof-Besov<interpolation}). We merely give a proof of (\ref{e:interpolation-proof-Besov<interpolation}) in the case $0 < q < \infty$ and the case when $q = \infty$ can be verified analogously. To this end, let $x \in ( X, D ( A^\alpha ) )_{ \theta, q }$. The proof of (\ref{e:interpolation-proof-Besov<interpolation}) is divided into the following four steps. 

Step I. It can be verified that 
\begin{align}\label{e:interpolation-proof-Besov<interpolation-R-start-2}
\bigg\{ \int_0^\infty \big\| s^{ \theta \alpha } A^\alpha ( s + A )^{ - \alpha } x \big\|^q \frac{ d s }{ s } \bigg\}^{ 1 / q } \lesssim \| x \|_{ ( X, D ( A^\alpha ) )_{ \theta, q } }.
\end{align} 
Indeed, applying the change of variable $s = t^{ - 1 / \alpha }$ yields 
\begin{align*} 
\bigg\{ \int_0^\infty \big\| s^{ \theta \alpha } A^\alpha ( s + A )^{ - \alpha } x \big\|^q \frac{ d s }{ s } \bigg\}^{ 1 / q } \simeq
\bigg\{ \int_0^\infty \big\| t^{ - \theta } A^\alpha ( t^{ - 1 / \alpha } + A )^{ - \alpha } x \big\|^q \frac{ d t }{ t } \bigg\}^{ 1 / q }.
\end{align*} 
Let $x_0: ( 0, \infty ) \mapsto X$ and $x_1: ( 0, \infty ) \mapsto D ( A^\alpha )$ such that $x_0 ( t ) + x_1 ( t ) \equiv x$. Applying the classical inequality (\ref{e: q-inequality}) yields 
\begin{align*}
& \bigg\{ \int_0^\infty \big\| s^{ \theta \alpha } A^\alpha ( s + A )^{ - \alpha } x \big\|^q \frac{ d s }{ s } \bigg\}^{ 1 / q } \\
\lesssim {} & \bigg\{ \int_0^\infty \big\| t^{ - \theta } A^\alpha ( t^{ - 1 / \alpha } + A )^{ - \alpha } x_0 ( t ) \big\|^q \frac{ d t }{ t } \bigg\}^{ 1 / q } \\
+ {} & \bigg\{ \int_0^\infty \big\| t^{ - \theta } A^\alpha ( t^{ - 1 / \alpha } + A )^{ - \alpha } x_1 ( t ) \big\|^q \frac{ d t }{ t } \bigg\}^{ 1 / q }.
\end{align*} 
Applying the uniform boundedness of $\{ A^\alpha ( t + A )^{ - \alpha } \}_{ t > 0 }$ and $\{ t^{ \alpha } ( t + A )^{ - \alpha } \}_{ t > 0}$ (see Lemma \ref{l:uniform boundedness compositions} (i) above)
\begin{align*} 
\begin{split}
& \bigg\{ \int_0^\infty \big\| s^{ \theta \alpha } A^\alpha ( s + A )^{ - \alpha } x \big\|^q \frac{ d s }{ s } \bigg\}^{ 1 / q } \\
\lesssim {} & \bigg\{ \int_0^\infty \big( t^{ - \theta } \| x_0 ( t ) \| \big)^q \frac{ d t }{ t } \bigg\}^{ 1 / q } + \bigg\{ \int_0^\infty \big( t^{ 1 - \theta } \| A^\alpha x_1 ( t ) \| \big)^q \frac{ d t }{ t } \bigg\}^{ 1 / q }.
\end{split}
\end{align*} 
And applying (\ref{e:equivalent quasi-norm of interpolation spaces}) yields (\ref{e:interpolation-proof-Besov<interpolation-R-start-2}), immediately. More precisely,  
\begin{align*} 
\bigg\{ \int_0^\infty \big\| s^{ \theta \alpha } A^\alpha ( s + A )^{ - \alpha } x \big\|^q \frac{ d s }{ s } \bigg\}^{ 1 / q } \lesssim  
\| x \|^*_{ ( X, D ( A^\alpha ) )_{ \theta, q } } \lesssim \| x \|_{ ( X, D ( A^\alpha ) )_{ \theta, q } }.
\end{align*} 
Thus, we have verified (\ref{e:interpolation-proof-Besov<interpolation-R-start-2}).

Step II. It can be verified that 
\begin{align}\label{e:interpolation-proof-Besov<interpolation-R}
\left| x \right|_{ R^{ \theta \alpha, A }_{ q, X } ( 0, 0, \alpha ) } \lesssim \| x \|_{ ( X, D ( A^\alpha ) )_{ \theta, q } }.
\end{align} 
Indeed, by Lemma \ref{l:Besov type-inhomogeneous-continuity} (also, see Remark \ref{r:inhomogeneous Besov spaces-equivalent quasi-norms-continuity} (i)), we have 
\begin{align*} 
\left| x \right|_{ R^{ \theta \alpha, A }_{ q, X } ( 0, 0, \alpha ) } \lesssim {} 
& \bigg\{ \int_0^\infty \big\| t^{ \theta \alpha } A^\alpha ( t + A )^{ - \alpha } x \big\|^q \frac{ d t }{ t } \bigg\}^{ 1 / q }.
\end{align*} 
Applying (\ref{e:interpolation-proof-Besov<interpolation-R-start-2}) yields (\ref{e:interpolation-proof-Besov<interpolation-R}), immediately.

Step III. 
It can also be verified that 
\begin{align}\label{e:interpolation-proof-Besov<interpolation-x}
\| x \| \lesssim \| x \|_{ ( X, D ( A^\alpha ) )_{ \theta, q } }.
\end{align} 
To this end, observe that $x \in B^{ s, A }_{ q, X }$ due to (\ref{e:interpolation-proof-Besov<interpolation-R}), and therefore $x \in D ( A^\epsilon )$ for $0 < \epsilon < \theta \alpha$
by Proposition \ref{p:inclusions-powers-inhomogeneous} above. Thanks to $0 \in \rho ( A )$, by using (\ref{r:homogeneous Calderon-regularity}) we have 
\begin{align}\label{e:interpolation-proof-Besov<interpolation-x-integral}
\| x \| \le \frac{ \Gamma ( 2 \alpha ) }{ \Gamma ( \alpha ) \Gamma (\alpha ) } \int_0^\infty \big\| \lambda^\alpha A^\alpha ( \lambda + A )^{ - 2 \alpha } x \big\| \, \frac{ d \lambda }{ \lambda }.
\end{align} 

If $0 < q \le 1$, by using the decomposition $\int_0^\infty = \sum_{ j \in \mathbb{Z} } \int_{ 2^j }^{ 2^{ j + 1 } }$ and dyadic estimate (\ref{e:resolvent estimate-discrete}) we obtain from (\ref{e:interpolation-proof-Besov<interpolation-x-integral}) that 
\begin{align*} 
\begin{split}
\| x \| \lesssim \sum_{ j \in \mathbb{Z} } \big\| 2^{ j \alpha } A^\alpha ( 2^j + A )^{ - 2 \alpha } \big\| \leq \Bigg\{ \sum_{ j \in \mathbb{Z} } \big\| 2^{ j \alpha } A^\alpha ( 2^j + A )^{ - 2 \alpha } \big\|^q \Bigg\}^{ 1 / q },
\end{split}
\end{align*} 
where the last inequality follows from (\ref{e: q-inequality-series<1}). 
From (\ref{e: q-inequality}), it follows that  
\begin{align*} 
\begin{split}
\| x \| \lesssim {} & \Bigg\{ \sum_{ j = 0 }^{ \infty } \big\| 2^{ j \alpha } A^\alpha ( 2^j + A )^{ - 2 \alpha } \big\|^q \Bigg\}^{ 1 / q } + \Bigg\{ \sum_{ j = - \infty }^{ - 1 } \big\| 2^{ j \alpha } A^\alpha ( 2^j + A )^{ - 2 \alpha } \big\|^q \Bigg\}^{ 1 / q }.
\end{split}
\end{align*} 
Applying (\ref{e:uniform boundedness compositions-M}) to $\| 2^{ j \alpha } ( 2^j + A )^{ - \alpha } \|$ yields  
\begin{align*} 
\begin{split}
\Bigg\{ \sum_{ j = 0 }^{ \infty } \big\| 2^{ j \alpha } A^\alpha ( 2^j + A )^{ - 2 \alpha } \big\|^q \Bigg\}^{ 1 / q } \lesssim {} & \Bigg\{ \sum_{ j = 0 }^{ \infty } \big\| A^\alpha ( 2^j + A )^{ - \alpha } \big\|^q \Bigg\}^{ 1 / q } \\
\le {} & \Bigg\{ \sum_{ j = 0 }^{ \infty } \big\| 2^{ j \theta \alpha } A^\alpha ( 2^j + A )^{ - \alpha } \big\|^q \Bigg\}^{ 1 / q },
\end{split}
\end{align*} 
while applying (\ref{e:uniform boundedness compositions-M}) and (\ref{e:uniform boundedness compositions-L}) to $\| 2^{ j \alpha ( 1 - \theta ) } A^{ \theta \alpha } ( 2^j + A )^{ - \alpha } \|$ yields 
\begin{align*} 
\begin{split}
\Bigg\{ \sum_{ j = - \infty }^{ - 1 } \big\| 2^{ j \alpha } A^\alpha ( 2^j + A )^{ - 2 \alpha } \big\|^q \Bigg\}^{ 1 / q } \lesssim \| A^{ - \theta \alpha } \| \Bigg\{ \sum_{ j = 0 }^{ \infty } \big\| 2^{ j \theta \alpha } A^\alpha ( 2^j + A )^{ - \alpha } \big\|^q \Bigg\}^{ 1 / q }
\end{split}
\end{align*} 
due to the fact that $0 \in \rho ( A )$. This implies that  
\begin{align*} 
\begin{split}
\| x \| \lesssim {} & \Bigg\{ \sum_{ j = 0 }^{ \infty } \big\| 2^{ j \theta \alpha } A^\alpha ( 2^j + A )^{ - \alpha } \big\|^q \Bigg\}^{ 1 / q }.
\end{split}
\end{align*} 
Applying the identity $\int_{ 2^{ j } }^{ 2^{ j + 1 } } 2^{ - j } \, d \lambda = 1$ and the dyadic  estimate (\ref{e:resolvent estimate-discrete}) yields 
\begin{align*} 
\begin{split}
\| x \| \lesssim {} & \Bigg\{ \int_1^\infty \big\| \lambda^{ \theta \alpha } A^\alpha ( \lambda + A )^{ - \alpha } \big\|^q \, \frac{ d \lambda }{ \lambda } \Bigg\}^{ 1 / q } \le \Bigg\{ \int_0^\infty \big\| \lambda^{ \theta \alpha } A^\alpha ( \lambda + A )^{ - \alpha } \big\|^q \, \frac{ d \lambda }{ \lambda } \Bigg\}^{ 1 / q }.
\end{split}
\end{align*} 
Finally, applying (\ref{e:interpolation-proof-Besov<interpolation-R-start-2}) yields the desired (\ref{e:interpolation-proof-Besov<interpolation-x}), immediately. 

And if $1 < q < \infty$, from (\ref{e:interpolation-proof-Besov<interpolation-x-integral}) it follows that 
\begin{align*} 
\begin{split}
\| x \| \lesssim \int_0^\infty \big\| \lambda^{ \alpha ( 1 - \theta ) } ( \lambda + A )^{ - \alpha } \big\| \cdot \big\| \lambda^{ \theta \alpha } A^\alpha ( \lambda + A )^{ - \alpha } x \big\| \, \frac{ d \lambda }{ \lambda }.
\end{split}
\end{align*} 
Thanks to the H\"{o}lder inequality, it can be seen that 
\begin{align*} 
\begin{split}
\| x \| \lesssim \Bigg\{ \int_0^\infty \big\| \lambda^{ \theta \alpha } A^\alpha ( \lambda + A )^{ - \alpha } x \big\|^q \, \frac{ d \lambda }{ \lambda } \Bigg\}^{ 1 / q } 
\end{split}
\end{align*} 
by observing that 
\begin{align*} 
\begin{split}
\int_0^\infty \big\| \lambda^{ \alpha ( 1 - \theta ) } ( \lambda + A )^{ - \alpha } \big\|^{ q' } \, \frac{ d \lambda }{ \lambda } < \infty 
\end{split}
\end{align*} 
due to the estimates (\ref{e:uniform boundedness compositions-M}) and (\ref{e:uniform boundedness compositions-L}) and the fact that $0 \in \rho ( A )$. Again, applying (\ref{e:interpolation-proof-Besov<interpolation-R-start-2}) yields the desired (\ref{e:interpolation-proof-Besov<interpolation-x}), immediately.

Step IV. By using (\ref{e:interpolation-proof-Besov<interpolation-R}) and (\ref{e:interpolation-proof-Besov<interpolation-x}), we obtain the desired (\ref{e:interpolation-proof-Besov<interpolation}), immediately. The proof is complete. 
\end{proof}

\begin{rem}\label{r:interpolations}
Theorems \ref{t:lifting property-negative}, \ref{t:smoothness reiteration} and \ref{p:interpolation} are the main results of this paper, which are novel in the sense equivalent (quasi-)norms even in the case $1 \le q \le \infty$. More precisely, let $s > 0$ and $1 \le q \le \infty$. 
\begin{itemize}
\item[(i)] By Remark \ref{r:inhomogeneous Besov spaces-equivalent quasi-norms-continuity} (i) and Theorem \ref{t:lifting property-negative}, we obtain \cite[Theorem 2.6]{Komatsu1967}, immediatley. Moreover, Theorem \ref{t:lifting property-negative} also improves \cite[Theorem 2.6]{Komatsu1967} in the sense of equivalent norms for $1 \le q \le \infty$.
\item[(ii)] By Remark \ref{r:inhomogeneous Besov spaces-equivalent quasi-norms-continuity} (i) and Theorem \ref{p:interpolation} we obtain \cite[Theorem 3.1]{Komatsu1967}, immediately. Moreover, Theorem \ref{p:interpolation} also improves \cite[Theorem 3.1]{Komatsu1967} in the sense of equivalent norms for $1 \le q \le \infty$.
\item[(iii)] By Remark \ref{r:inhomogeneous Besov spaces-equivalent quasi-norms-continuity} (i) and Theorems \ref{p:interpolation} and \ref{t:smoothness reiteration} we obtain \cite[Corollary 7.3 (a)]{Haase2005}, immediately. Moreover, Theorems \ref{p:interpolation} and \ref{t:lifting property-negative} also improve \cite[Corollary 7.3 (b)]{Haase2005} in the sense of equivalent norms for $1 \le q \le \infty$.
\end{itemize}
\end{rem}

\smallskip

%
%
%
%
%
%
%
%
%
%
%
%
%
%
%
%
%
%
%


\section{Comparison of classical and new Besov spaces}\label{S:Comparison}

It is natural to compare our new Besov spaces with the classical Besov spaces and other known Besov spaces associated with concrete operators, for instance, Besov spaces constructed in the frame of extrapolation spaces \cite{Matsumoto and Ogawa2010} and Besov spaces associated with $0$-sectorial operators \cite{Kriegler and Weis2016} or heat kernels \cite{BDY2012}.

\smallskip


\subsection{Extrapolation method}

Let $A$ be a non-negative operator on a Banach space $( X, \left\| \cdot \right\|_X )$. 
Write $X_0 := \overline{ D ( A ) }$ endowed with the norm $\left\| \cdot \right\|_{ X_0 } := \left\| \cdot \right\|_X$, and let $A_0 := A |_{ X_0 }$, the part of $A$ in $X_0$.
It can be verified that $A_0$ is non-negative on $X_0$ with dense domain. Indeed, for each $x \in X_0$ fixed, write $x_n := n ( n + A )^{ - 1 } x$ with $n \in \mathbb{N}$. It is clear that $x_n \in D ( A )$ and 
\begin{align*}
A x_n = n A ( n + A )^{ - 1 } x = n ( A + n - n ) ( n + A )^{ - 1 } x = n x - n^2 ( n + A )^{ - 1 } x \in X_0 
\end{align*}
due to the fact that $n x \in \overline{ D ( A ) } = X_0$ and $n^2 ( n + A )^{ - 1 } x \in D ( A )\subset X_0$,  and hence, $x_n \in D ( A_0 )$ for each $n \in \mathbb{N}$. Since $x_n \rightarrow x$ in $X$ as $n \rightarrow \infty$ due to the fact that $x \in X_0 = \overline{ D ( A ) }$, it follows that $\overline{ D ( A_0 ) } = X_0$. Moreover, the non-negativity of $A_0$ follows from that of $A$, immediately. More precisely, 
\begin{align*}
\sup_{ \lambda > 0 } \| \lambda ( \lambda + A_0 )^{ - 1 } \| \leq M_A, 
\end{align*}
where $M_A$ is the non-negativity constant of $A$ given in (\ref{e:non-negative constant}).

Now let $X_{ - 1 }$ be the completion of $X_0$ with respect to the norm $\left\| \cdot \right\|_{ - 1 } := \left\| ( 1 + A_0 )^{ - 1 } \cdot \right\|_{ X_0 }$. It is clear that, for each $\lambda > 0$, $( \lambda + A_0 )^{ - 1 }$ admits a bounded extension $J_{ \lambda, - 1 }$ on $X_{ - 1 }$ and 
\begin{align*}
\sup_{ \lambda > 0 } \left\| \lambda J_{ \lambda, - 1 } \right\| \leq M_A.
\end{align*}
Furthermore, by the definition of the bounded extension, it is easy to verify that the operator $J_{ \lambda, - 1 }$ is injective for each $\lambda > 0$, and hence, $A_{ - 1 } := J_{ \lambda, - 1 }^{ - 1 } - \lambda$ is non-negative $X_{ - 1 }$ with dense domain. By recursion, one can define an extrapolation space $X_{ - l }$ and a non-negative operator $A_{ - l }$ with dense domain in $X_{ - l }$ for each $l \geq 2$. Thus, we have obtained a series of Banach spaces $\{  X_{ - l } \}_{ l \in \mathbb{N}_0 }$, where $X_{ - l }$ is densely embedded in $X_{ - l - 1 }$ for each $l \in \mathbb{N}_0$. 

Note that $( 1 + A_0 )^{ - 1 }$ is bounded from $X_0$ to $R ( ( 1 + A_0 )^{ - 1 } ) = ( 1 + A_0 )^{ - 1 } X_0 \subset X_0$ and that $X_0$ is the completion of $R ( ( 1 + A_0 )^{ - 1 } )$ with respect to $\left\| \cdot \right\|_{ - 1 }$. Thus, the extension $J_{ - 1, - 1 }$ of $( 1 + A_0 )^{ - 1 }$ is bounded from $X_{ - 1 }$ to $X_{ 0 }$, and hence, 
\begin{align*}
D ( A_{ - 1 } ) = D ( 1 + A_{ - 1 } ) = D ( J_{ - 1, - 1 }^{ - 1 } ) = R ( J_{ - 1, - 1 } ) = X_0.
\end{align*}
In general, by the recursive construction, it can be seen that $D ( A_{ - l } ) = X_{ - l + 1 }$ for each $ l \in \mathbb{ N }$ and that $A_{ - l } = A_{ - l - 1 } |_{ D ( A_{ - l } ) }$ for each $l \in \mathbb{ N }_0$. 

Observe that the operators $A_{ - n }$ on $X_{ - n}$ are consistent with each other, and therefore there is a well defined operator $\mathcal{A}$ on 
\begin{align*}
X_{ - \infty } := \bigcup_{ n \in \mathbb{N}_0 } X_{ - n }.
\end{align*}
Clearly, $\mathcal{A} |_{ X_{ - n } } = A_{ - n }$ for each $n \in \mathbb{N}_0$.
Moreover, for given $t > 0$, it can be seen that $\mathcal{A}^m ( t + \mathcal{ A } )^{ - m - l } \tilde{x} \in X_0$ for $\tilde{x} \in X_{ - l }$ by observing that $( t + \mathcal{ A } )^{ - l }$ maps $X_{ - l }$ into $X_0$. More precisely,   
\begin{align}\label{e:Besov spaces-inhomogeneous-extrapolaiton-Rquasinorms}
\mathcal{A}^m ( t + \mathcal{ A } )^{ - m - l } \tilde{x} = A^m ( t + A )^{ - m } ( t + \mathcal{A} )^{ - l } \tilde{x}, \quad \tilde{x} \in X_{ - l }.
\end{align}

We can now give an extrapolation version of abstract Besov spaces in the following way. 

\begin{defn}\label{d:Besov spaces-inhomogeneous-extrapolaiton}
Let $0 < q \leq \infty$ and $s \in \mathbb{R}$. Let $k \in \mathbb{Z}$ and let $l$ and $m$ be
two non-negative integers such that $- l < s < m$. The inhomogeneous Besov space $\mathcal{B}^{ s, A }_{ q, X }$ is
defined by
\begin{align*}
\mathcal{B}^{ s, A }_{ q, X } := \big\{ \tilde{x} \in X_{ - l }: |
\tilde{x} |_{ \mathcal{R}^{ s, A }_{ q, X } ( k, l, m ) } < \infty \big\} 
\end{align*}
endowed with the quasi-norm 
\begin{align*}
\| \tilde{x} \|_{ \mathcal{B}^{ s, A }_{ q, X } } := \big\| (
2^k + \mathcal{A} )^{ - l } \tilde{x} \big\|_X + | \tilde{x}
|_{ \mathcal{R}^{ s, A }_{ q, X } ( k, l, m ) },
\end{align*}
where 
\begin{align*}
| \tilde{x} |_{ \mathcal{R}^{ s, A }_{ q, X } ( k, l, m ) } :=
\Bigg\{ \sum_{ j = k }^\infty \big\| 2^{ j s } 2^{ j l }
\mathcal{A}^m ( 2^j + \mathcal{A} )^{ - m - l } \tilde{x}
\big\|_X^q \Bigg\}^{ 1 / q }
\end{align*}
(with the usual modification if $q = \infty$). 
\end{defn}

A homogeneous version of abstract Besov spaces via the extrapolation is given as follows. 

\begin{defn}\label{d:Besov spaces-homogeneous-extrapolaiton}
Let $0 < q \leq \infty$ and $s \in \mathbb{R}$, and let $l$ and $m$ be
two non-negative integers such that $- l < s < m$. The homogeneous Besov space $\mathcal{\dot{B}}^{ s, A }_{ q, X }$ is
defined by
\begin{align*}
\mathcal{\dot{B}}^{ s, A }_{ q, X } := \Bigg\{ \tilde{x} \in X_{ - l }: \sum_{ j = - \infty }^\infty \big\| 2^{ j s } 2^{ j l } \mathcal{A}^m ( 2^j + \mathcal{A} )^{ - m - l } \tilde{x}
\big\|_X^q < \infty \Bigg\},
\end{align*}
endowed with the quasi-norm 
\begin{align*}
\| \tilde{x} \|_{ \mathcal{\dot{B}}^{ s, A }_{ q, X } } := 
\Bigg\{ \sum_{ j = - \infty }^\infty \big\| 2^{ j s } 2^{ j l }
\mathcal{A}^m ( 2^j + \mathcal{A} )^{ - m - l } \tilde{x}
\big\|_X^q \Bigg\}^{ 1 / q }
\end{align*}
(with the usual modification if $q = \infty$). 
\end{defn}

\begin{rem}
Analogous to Theorem \ref{t:inhomogeneous Besov spaces-equivalent quasi-norms-continuity} (more precisely, Remark \ref{r:inhomogeneous Besov spaces-equivalent quasi-norms-continuity} above), it can be verified that $\mathcal{\dot{B}}^{ s, A }_{ q, X }$ admits an equivalent quasi-norm of continuous type, and therefore $\mathcal{\dot{B}}^{ s, A }_{ q, X }$ coincides with $\dot{B}^\phi_{ X, s }$ with $\phi ( \lambda ) = \lambda^s$ ($\lambda > 0$) for $1 \le q \le \infty$ and $s \in \mathbb{R}$, where $\dot{B}^\phi_{ X, s }$ is the homogeneous Besov space due to T. Matsumoto and T. Ogawa \cite[Definitions 2.8 and 2.10]{Matsumoto and Ogawa2010}.
\end{rem}

\begin{prop}\label{p:Besov=extrapolation-inhomogeneous}
Let $0 < q \leq \infty$ and $s \in \mathbb{R}$. Then $\mathcal{B}^{
s, A }_{ q, X } = B^{ s, A }_{ q, X }$ in the sense of equivalent
quasi-norms.
\end{prop}

\begin{proof}
%

We first verify that 
\begin{align}\label{e:Besov=extrapolation-inhomogeneous-proof-Besov-in-extrapolation}
B^{ s, A }_{ q, X } 
\subset \mathcal{B}^{ s, A }_{ q, X }, 
\end{align}
and
\begin{align}\label{e:Besov=extrapolation-inhomogeneous-proof-Besov-in-extrapolation-quasinorm}
\left\| \tilde{x} \right\|_{ \mathcal{B}^{ s, A }_{ q, X } } \simeq \left\| \tilde{x} \right\|_{ B^{ s, A }_{ q, X } }, \quad x \in B^{ s, A }_{ q, X }.
\end{align}
To this end, let $\tilde{x} \in B^{ s, A }_{ q, X }$. Fix $k \in \mathbb{Z}$ and let $m$ and $l$ be non-negative integers such that $- l < s < m$.  By Definitions \ref{d:inhomogeneous Besov spaces} and \ref{d:Besov type-inhomogeneous}, $\tilde{x} = \{ x_n \}$ for some $\{ x_n \} \subset R^{ s, A }_{ q, X } ( k, l, m )$ which is Cauchy with respect to the Besov quasi-norm 
$\left\| \cdot \right\|_{ B^{ s, A }_{ q, X } ( k, l, m ) }$ given by (\ref{e:quasi-norm of inhomogeneous Besov spaces}). 
This implies that, for each $\epsilon > 0$ fixed, there is an $N \in \mathbb{N}$ such that 
\begin{align}\label{e:Besov=extrapolation-inhomogeneous-proof-Besov-in-extrapolation-<e}
\begin{split}
\left\| x_n - x_{ n' } \right\|_{ B^{ s, A }_{ q, X } ( k, l, m ) } = {} & \| ( 2^k + A )^{ - l } ( x_n - x_{ n' } ) \|_{ X } \\
+ {} & \Bigg\{ \sum_{ j = k }^{ \infty } \big\| 2^{ j ( s + l ) } A^m ( 2^j + A )^{ - m - l } ( x_n - x_{ n' } ) \big\|_X^q \Bigg\}^{ 1 / q } \\
< {} & \epsilon, \quad n, n' \ge N.
\end{split}
\end{align}
From (\ref{e:Besov=extrapolation-inhomogeneous-proof-Besov-in-extrapolation-<e}) it is clear that $\{ x_n \}$ is Cauchy with respect to the norm $\left\| ( 2^k + A )^{ - l } \cdot \right\|_X$, and hence, $x \in X_{ - l }$ by observing that $X_{ - l }$ coincides with the completion of $X_0 = \overline{ D ( A ) }$ with respect to the norm $\| ( 2^k + A )^{ - l } \cdot \|_{ X_0 } = \| ( 2^k + A )^{ - l } \cdot \|_{ X }$. 

We now verify that 
\begin{align}\label{e:Besov=extrapolation-inhomogeneous-proof-Besov-in-extrapolation-R<infty}
| \tilde{x} |_{ \mathcal{R}^{ s, A }_{ q, X } ( k, l, m ) } < \infty.
\end{align}
Fix $\epsilon > 0$. From (\ref{e:Besov spaces-inhomogeneous-extrapolaiton-Rquasinorms}) and (\ref{e:Besov=extrapolation-inhomogeneous-proof-Besov-in-extrapolation-<e}) it follows that 
\begin{align*}
& \Bigg\{ \sum_{ j = k }^{ K } \big\| 2^{ j ( s + l ) } \mathcal{A}^m ( 2^j + \mathcal{A} )^{ - m - l } ( x_n - x_{ n' } ) \big\|_X^q \Bigg\}^{ 1 / q } \\
= {} & \Bigg\{ \sum_{ j = k }^{ K } \big\| 2^{ j ( s + l ) } A^m ( 2^j + A )^{ - m } ( 2^j + \mathcal{A} )^{ - l } ( x_n - x_{ n' } ) \big\|_X^q \bigg\}^{ 1 / q } \\
= {} & \Bigg\{ \sum_{ j = k }^{ K } \big\| 2^{ j ( s + l ) } A^m ( 2^j + A )^{ - m - l } ( x_n - x_{ n' } ) \big\|_X^q \Bigg\}^{ 1 / q } < \epsilon, \quad n, n' \ge N,
\end{align*}
for each integer $K \ge k$. Applying $n' \to \infty$ yields 
\begin{align*}
& \Bigg\{ \sum_{ j = k }^{ K } \big\| 2^{ j ( s + l ) } \mathcal{A}^m ( 2^j + \mathcal{A} )^{ - m - l } ( x_n - \tilde{x} ) \big\|_X^q \Bigg\}^{ 1 / q } \le \epsilon, \quad n \ge N,
\end{align*}
for each integer $K \ge k$ since $( 2^j + \mathcal{A} )^{ - l } ( x_n - x_{ n' } )$ converges to $( 2^j + \mathcal{A} )^{ - l } ( x_n - \tilde{x} )$ as $n' \to \infty$ due to (\ref{e:Besov=extrapolation-inhomogeneous-proof-Besov-in-extrapolation-<e}). Applying $K \to \infty$ yields 
\begin{align}\label{e:Besov=extrapolation-inhomogeneous-proof-Besov-in-extrapolation-R-xn-x}
| x_n - \tilde{x} |_{ \mathcal{R}^{ s, A }_{ q, X } ( k, l, m ) } =
 \Bigg\{ \sum_{ j = k }^{ \infty } \big\| 2^{ j ( s + l ) } \mathcal{A}^m ( 2^j + \mathcal{A} )^{ - m - l } ( x_n - \tilde{x} ) \big\|_X^q \Bigg\}^{ 1 / q } \le \epsilon 
\end{align}
for $n \ge N$. From (\ref{e: q-inequality}) and (\ref{e:Besov=extrapolation-inhomogeneous-proof-Besov-in-extrapolation-R-xn-x}) it follows that 
\begin{align*}
| \tilde{x} |_{ \mathcal{R}^{ s, A }_{ q, X } ( k, l, m ) } \lesssim | \tilde{x} - x_N |_{ \mathcal{R}^{ s, A }_{ q, X } ( k, l, m ) } + | x_N |_{ \mathcal{R}^{ s, A }_{ q, X } ( k, l, m ) } < \infty.
\end{align*}
which is the desired (\ref{e:Besov=extrapolation-inhomogeneous-proof-Besov-in-extrapolation-R<infty}). Thus, we have verified (\ref{e:Besov=extrapolation-inhomogeneous-proof-Besov-in-extrapolation}).

Moreover, observing that $x_n \to \tilde{x}$ as $n \to \infty$ in $X_{ - l }$ as mentioned above and that $x_n \to \tilde{x}$ as $n \to \infty$ in $\mathcal{R}^{ s, A }_{ q, X } ( k, l, m )$ due to (\ref{e:Besov=extrapolation-inhomogeneous-proof-Besov-in-extrapolation-R-xn-x}), we conclude that
\begin{align*} 
\left\| \tilde{x} \right\|_{ \mathcal{B}^{ s, A }_{ q, X } } = \lim_{ n \to \infty } \left\| x_n \right\|_{ \mathcal{B}^{ s, A }_{ q, X } } = \lim_{ n \to \infty } \left\| x_n \right\|_{ B^{ s, A }_{ q, X } } \simeq \left\| \tilde{x} \right\|_{ B^{ s, A }_{ q, X } },
\end{align*}
which is the desired (\ref{e:Besov=extrapolation-inhomogeneous-proof-Besov-in-extrapolation-quasinorm}).

In order to complete the proof, it remains to verify that
\begin{align}\label{e:Besov=extrapolation-inhomogeneous-proof-Besov-in-extrapolation-inverse}
\mathcal{B}^{ s, A }_{ q, X } 
\subset B^{ s, A }_{ q, X }.
\end{align}
Cleaarly, it suffices to verify that $R^{ s, A }_{ q, X } ( k, l, m )$ is dense in $\mathcal{B}^{ s, A }_{ q, X }$. Indeed, let $\tilde{x} \in \mathcal{B}^{ s, A }_{ q, X }$, i.e., $\tilde{x} \in X_{ - l }$ and 
\begin{align}\label{e:Besov=extrapolation-inhomogeneous-proof-Besov-in-extrapolation-inverse<infty}
| \tilde{x} |_{ \mathcal{R}^{ s, A }_{ q, X } ( k, l, m ) } = \Bigg\{ \sum_{ j = k }^{ \infty } \big\| 2^{ j ( s + l ) } \mathcal{A}^m ( 2^j + \mathcal{A} )^{ - m - l } \tilde{x} \big\|_X^q \Bigg\}^{ 1 / q } < \infty. 
\end{align}
Write $P_n := n ( n + \mathcal{A} )^{ - 1 }$ with $n \in \mathbb{N}$. It can be seen that $P_n^l \tilde{x} \in \overline{ D ( A ) }$ for each $n \in \mathbb{N}$ by observing that $( t + A )^{ - l }$ maps $X_{ - l }$ into $X_0 = \overline{ D ( A ) }$ for each $t > 0$, and therefore $P_n^l \tilde{x} \in R^{ s, A }_{ q, X } ( k, l, m )$ for each $n \in \mathbb{N}$
due to (\ref{e:Besov=extrapolation-inhomogeneous-proof-Besov-in-extrapolation-inverse<infty}) and the uniform boundedness of $\{ n^l ( n + A )^{ - l }\}_{ n \in \mathbb{N} }$. Thanks to Lemma \ref{l:ergodicity} (i), it can also be seen that 
\begin{align}\label{e:Besov=extrapolation-inhomogeneous-proof-Besov-in-extrapolation-inverse-k}
\begin{split}
\| ( 2^k + \mathcal{A} )^{ - l } ( P_n^l \tilde{x} - \tilde{x} ) \|_X = \| ( P_n^l - I ) ( 2^k + \mathcal{A} )^{ - l } \tilde{x} \|_X \to 0 
\end{split}
\end{align}
as $n \to \infty$ by observing the fact that $( 2^k + \mathcal{A} )^{ - l } \tilde{x} \in \overline { D ( A ) }$. Also, by Lemma \ref{l:ergodicity} (i) we conclude that 
\begin{align}\label{e:Besov=extrapolation-inhomogeneous-proof-Besov-in-extrapolation-inverse-q}
\begin{split}
& \Bigg\{ \sum_{ j = k }^{ \infty } \big\| 2^{ j ( s + l ) } \mathcal{A}^m ( 2^j + \mathcal{A} )^{ - m - l } ( P_n \tilde{x} - \tilde{x} ) \big\|_X^q \Bigg\}^{ 1 / q } \to 0, \quad n \to \infty.  
\end{split}
\end{align}
More precisely, fix $\epsilon > 0$. On the one hand, thanks to  (\ref{e:Besov=extrapolation-inhomogeneous-proof-Besov-in-extrapolation-inverse<infty}), there is an integer $K$ large enough such that 
\begin{align*}
\Bigg\{ \sum_{ j = K }^{ \infty } \big\| 2^{ j ( s + l ) } \mathcal{A}^m ( 2^j + \mathcal{A} )^{ - m - l } \tilde{x} \big\|_X^q \Bigg\}^{ 1 / q } < \frac{ \epsilon }{ 2 C C_{ 1 / q } },
\end{align*}
where $C = \sup_{ n \in \mathbb{N} } \| P_n^l - I \|$ and $C_{ 1 / q }$ is the constant given in (\ref{e: q-inequality}). On the other hand, thanks to Lemma \ref{l:ergodicity} (i), there is an integer $N$ such that  
\begin{align*}
\Bigg\{ \sum_{ j = 0 }^{ K - 1 } \big\| ( P_n^l - I ) 2^{ j ( s + l ) } \mathcal{A}^m ( 2^j + \mathcal{A} )^{ - m - l } \tilde{x} \big\|_X^q \Bigg\}^{ 1 / q } < \frac{ \epsilon }{ 2 C_{ 1 / q } }, \quad n \ge N,
\end{align*}
due to the fact that $2^{ j ( s + l ) } \mathcal{A}^m ( 2^j + \mathcal{A} )^{ - m - l } \tilde{x} \in \overline{D ( A ) }$ for each $j = 0, 1, \cdots, K-1$. Therefore, from (\ref{e: q-inequality}) it follows that 
\begin{align*}
& \Bigg\{ \sum_{ j = k }^{ \infty } \big\| 2^{ j ( s + l ) } \mathcal{A}^m ( 2^j + \mathcal{A} )^{ - m - l } ( P_n \tilde{x} - \tilde{x} ) \big\|_X^q \Bigg\}^{ 1 / q } \\
= {} & \Bigg\{ \sum_{ j = 0 }^{ \infty } \big\| ( P_n^l - I ) 2^{ j ( s + l ) } \mathcal{A}^m ( 2^j + \mathcal{A} )^{ - m - l } \tilde{x} \big\|_X^q \Bigg\}^{ 1 / q } \\
\le {} & C_{ 1 / q } \Bigg\{ \sum_{ j = 0 }^{ K - 1 } \big\| ( P_n^l - I ) 2^{ j ( s + l ) } \mathcal{A}^m ( 2^j + \mathcal{A} )^{ - m - l } \tilde{x} \big\|_X^q \Bigg\}^{ 1 / q } \\
+ {} & C_{ 1 / q } \Bigg\{ \sum_{ j = K }^{ \infty } \big\| ( P_n^l - I ) 2^{ j ( s + l ) } \mathcal{A}^m ( 2^j + \mathcal{A} )^{ - m - l } \tilde{x} \big\|_X^q \Bigg\}^{ 1 / q }.
\end{align*}
Applying the uniform boundedness of $\{ P_n^l - I \}_{ n \in \mathbb{N} }$ yields 
\begin{align*}
& \Bigg\{ \sum_{ j = k }^{ \infty } \big\| 2^{ j ( s + l ) } \mathcal{A}^m ( 2^j + \mathcal{A} )^{ - m - l } ( P_n \tilde{x} - \tilde{x} ) \big\|_X^q \Bigg\}^{ 1 / q } \\
\le {} & C_{ 1 / q } \Bigg\{ \sum_{ j = 0 }^{ K - 1 } \big\| ( P_n^l - I ) 2^{ j ( s + l ) } \mathcal{A}^m ( 2^j + \mathcal{A} )^{ - m - l } \tilde{x} \big\|_X^q \Bigg\}^{ 1 / q } \\
+ {} & C_{ 1 / q } \sup_{ n \in \mathbb{ N } } \| P_n^l - I \| \Bigg\{ \sum_{ j = K }^{ \infty } \big\| 2^{ j ( s + l ) } \mathcal{A}^m ( 2^j + \mathcal{A} )^{ - m - l } \tilde{x} \big\|_X^q \Bigg\}^{ 1 / q } \\
< {} & \frac{ \epsilon }{ 2 } + \frac{ \epsilon }{ 2 } = \epsilon, \quad n \ge N,
\end{align*}
from which (\ref{e:Besov=extrapolation-inhomogeneous-proof-Besov-in-extrapolation-inverse-q}) follows immediately. From (\ref{e:Besov=extrapolation-inhomogeneous-proof-Besov-in-extrapolation-inverse-k}) and (\ref{e:Besov=extrapolation-inhomogeneous-proof-Besov-in-extrapolation-inverse-q})
it can be seen that $P_n^l \tilde{x} \to \tilde{x}$ as $n \to \infty$, and hence, $R^{ s, A }_{ q, X } ( k, l, m )$ is dense in $\mathcal{B}^{ s, A }_{ q, X }$. Thus, we have verified (\ref{e:Besov=extrapolation-inhomogeneous-proof-Besov-in-extrapolation-inverse}).
The proof is complete.
\end{proof}

\smallskip


\subsection{Functional calculus approach}\label{Sub:Functional calculus}


Let $X$ be a Banach space and let $0 \leq \omega < \pi$. Recall that a closed linear operator $A : D ( A ) \subset X \rightarrow X$ is sectorial of angle $\omega$ if $\sigma ( A ) \subset \overline{ \Sigma_\omega }$ and 
\begin{align*}
\sup \big\{ \| z R ( z, A ) \|: z \in \mathbb{C} \setminus \overline{ \Sigma_{ \omega' } } \big\} < \infty
\end{align*}
for each $\omega' \in ( \omega, \pi )$. If this is the case, the value 
\begin{align*}
\omega_A := \min \{ \omega \in [ 0, \pi ): A \in S ( \omega ) \}
\end{align*}
is called the spectral angle of $A$. 

Let $A$ be a closed linear operator on a Banach space $X$. Recall that $A$ is non-negative if and only if it is sectorial (with spectral angle $\omega_A \leq \pi - \arcsin 1 / M_A$, where $M_A$ is the non-negativity constant of operator $A$) (see, \cite[Proposition 1.2.1]{MartinezM2001}). Thus, the theory of Besov spaces associated with non-negative operators on Banach spaces can also be developed by using the approach of functional calculus. 

Let's recall some preliminaries of the so-called primary functional calculus of sectorial operators. Let $\omega \in ( 0, \pi )$ and $A$ be sectorial with spectral angle $\omega_A < \omega$. Let $H^\infty ( \Sigma_\omega )$ be the space of bounded holomorphic function on $\Sigma_\omega$ and write 
\begin{align*}
H_0^\infty ( \Sigma_\omega ) := \big\{ f \in H^\infty ( \Sigma_\omega ): \exists \, C, \epsilon > 0 \mbox{ s.t. } | f ( z ) | \leq C \min\{ |z|^\epsilon, |z|^{ - \epsilon } \} \big\}. 
\end{align*}
For each $\psi \in H_0^\infty ( \Sigma_\omega )$, define 
\begin{align*}
\psi ( A ) := \frac{ 1 }{ 2 \pi i } \int_{ \partial \Sigma_{ \omega' } } \psi ( z ) ( z - A )^{ - 1 } \, d z, 
\end{align*}
where $\partial \Sigma_{ \omega' }$ is the boundary of a sector $\Sigma_{ \omega' }$ with $\omega' \in ( \omega, \pi )$, oriented counterclockwise. By the Cauchy integral theorem, the integral in the right-hand side of the last equality is independent of the choice of $\omega'$, and hence, $\psi ( A )$ is well defined as a bounded linear operator on the Banach space $X$. More precisely, the calculus 
\begin{align*}
\Phi : H_0^\infty ( \Sigma_\omega ) \rightarrow B ( X ) 
\end{align*}
is a homomorphism of algebras. Furthermore, one can enlarge the algebra $H_0^\infty ( \Sigma_\omega )$ as 
\begin{align*}
\mathcal{E} ( \Sigma_\omega ) := H_0^\infty ( \Sigma_\omega ) \oplus \mathbb{C} \frac{ 1 }{ 1 + z } \oplus \mathbb{C} \mathbf{1},
\end{align*}
and define 
\begin{align*}
f ( A ) := \psi ( A ) + c ( 1 + A )^{ - 1 } + d 
\end{align*}
for each function $f ( z ) := \psi ( z ) + \frac{ c }{ 1 + z } + d \in \mathcal{E} ( \Sigma_\omega )$. Such an extended calculus 
\begin{align*}
\Phi : \mathcal{E} ( \Sigma_\omega ) \rightarrow B ( X ) 
\end{align*}
is also a homomorphism of algebras \cite[Page 34, (2.7)]{HaaseM2006}, and we refer to \cite[Chapters 1 and 2]{HaaseM2006} for more information on the 
primary functional calculus for sectorial operators. 

\begin{prop}\cite[Theorem 1]{Haase2005}\label{p:Besov of Haase}
Let $A$ be a sectorial operator of angle $\omega$, and let $\omega' \in ( \omega, \pi )$ and $\mathrm{Re} \, \alpha > 0$. Take a function $0 \neq \psi \in \mathcal{E} ( \Sigma_{ \omega' } )$ such that $z^{ - \alpha } \psi ( z ) \in \mathcal{E} ( \Sigma_{ \omega' } )$. Then 
\begin{align*}
( X, D ( A ) )_{ \theta, q } = \left\{ x \in X: \int_0^\infty \big\| t^{ - \theta \alpha } \psi ( t A ) x \big\|^q \, \frac{ d t }{ t } < \infty \right\}
\end{align*}
with the equivalence of norms
\begin{align*}
\| x \|_{ ( X, D ( A ) )_{ \theta, q } } \simeq \| x \| + \left\{ \int_0^\infty \big\| t^{ - \theta \alpha } \psi ( t A ) x \big\|^q \, \frac{ d t }{ t } \right\}^{ 1 / q } 
\end{align*}
for all $\theta \in ( 0, 1 )$ and $1 \leq q \leq \infty$.
\end{prop}

In particular, for given $\theta \in ( 0, 1 )$ and $1 \leq q \leq \infty$, applying $\psi ( z ) = z^\alpha ( 1 + z )^{ - \beta }$ with $0 < \mathrm{Re} \alpha \leq \mathrm{Re} \beta < \infty$ to Proposition \ref{p:Besov of Haase} above yields 
\begin{align*}
( X, D ( A ) )_{ \theta, q } = \bigg\{ x \in X: \int_0^\infty \big\| t^{ \theta \alpha } \cdot t^{ \beta - \alpha } A^\alpha ( t + A )^{ - \beta } x \big\|^q \frac{ d t }{ t } < \infty \bigg\}
\end{align*}
with the equivalence of norms
\begin{align}\label{e:Besov norm of Haase}
\| x \|_{ ( X, D ( A ) )_{ \theta, q } } \simeq \| x \|_X + \bigg\{ \int_0^\infty \big\| t^{ \theta \alpha } \cdot t^{ \beta - \alpha } A^{ \alpha } ( t + A )^{ - \beta } x \big\|^q \frac{ d t }{ t } \bigg\}^{ 1 / q }  
\end{align}
(see, \cite[subsection 7.3]{Haase2005}). Obviously, the norm given in the right-hand side of (\ref{e:Besov norm of Haase}) is equivalent to our Besov norm $\left\| \cdot \right\|_{ B^{ \theta \alpha, A }_{ q, X } (k, \beta - \alpha, \alpha ) }$ as that given in (\ref{e:quasi-norm of inhomogeneous Besov spaces}) due to Theorem \ref{t:inhomogeneous Besov spaces-equivalent quasi-norms-continuity} above.


An alternative approach to Besov spaces associated with operators is the so-called Mihlin functional calculus \cite[Definition 3.3]{Kriegler and Weis2016}, which is constructed in the frame of sectorial operators with particular angle $\omega = 0$ (i.e., $0$-sectorial operator for short) and bounded holomorphic functions $f$ satisfying the following Mihlin condition, 
\begin{align*}
\sup_{ t > 0 } \big| t^k f^{ ( k ) } ( t ) \big| < \infty, \quad k = 0, 1, \cdots, N.
\end{align*}
More precisely, let $\alpha > 0$ and write 
\begin{align*}
\mathcal{M}^\alpha := \{ f: \mathbb{R}_+ \mapsto \mathbb{C}: f ( e^x ) \in B^\alpha_{ \infty, 1 } \},
\end{align*}
endowned with the norm $\| f \|_{ \mathcal{M}^\alpha } := \| f_e \|_{ B^\alpha_{ \infty, 1 } }$.
Let $A$ be a sectorial operator of angle $\omega = 0$. Recall that $A$ has a bounded $\mathcal{M}^\alpha$ calculus if there is a constant $C > 0$ such that 
\begin{align*}
\| f ( A ) \| \le C \| f \|_{ \mathcal{M}^\alpha }, \quad f \in \bigcap_{ 0 < \omega < \pi } H^\infty ( \Sigma_\omega ) \cap \mathcal{M}^\alpha,
\end{align*}
See \cite[Definition 3.3]{Kriegler and Weis2016}. Moreover, define 
\begin{align*}
\mathcal{M}^\alpha_1 := \Bigg\{ f \in \mathcal{M}^\alpha: \| f \|_{ \mathcal{M}^\alpha_1 } := \sum_{ n \in \mathbb{Z} } \| f \dot{ \varphi }_n \|_{ \mathcal{M}^\alpha } < \infty \Bigg\},
\end{align*}
where $\{ \dot{ \varphi }_n \}_{ n \in \mathbb{Z} }$ is a homogeneous dyadic partition of unity on $\mathbb{R}_+$ (see \cite[Definition 2.2]{Kriegler and Weis2016}). And recall that $A$ has a bounded $\mathcal{M}^\alpha_1$ calculus if there is a constant $C > 0$ such that 
\begin{align*}
\| f ( A ) \| \le C \| f \|_{ \mathcal{M}^\alpha_1 }, \quad f \in H_*^\infty ( \Sigma_\omega ),  
\end{align*}
for some $\omega \in (0, \pi)$, where 
\begin{align*}
H_*^\infty ( \Sigma_\omega ) = \{ f \in H^\infty ( \Sigma_\omega ): \exists \, C, \epsilon > 0 \mbox{ s.t. } | f ( z ) | \le C ( 1 + | \log z |^{ - 1 - \epsilon } ) \}.  
\end{align*}
See \cite[Definition 3.8]{Kriegler and Weis2016}.
 
Suppose that the $0$-sectorial operator $A$ has a bounded $\mathcal{M}^\alpha_1$ calculus for some $\alpha > 0$, and let $\{ \dot{\varphi}_n \}_{ n \in \mathbb{Z} }$ and $\{ \varphi_n \}_{ n \in \mathbb{N}_0 }$ be a homogeneous and an inhomogeneous partition of unity on $\mathbb{R}_+$, respectively (see \cite[Definition 2.2]{Kriegler and Weis2016}). Let $s \in \mathbb{R}$ and $1 \leq q \leq \infty$. Recall that the homogeneous Besov space $\dot{B}^s_q ( A )$ and inhomogeneous Besov space $B^s_q ( A )$ are defined by  
\begin{align*}
 \dot{B}^s_q ( A ) := \left\{ x \in \dot{X}_{ - N } + \dot{X}_{ N }: \| x \|_{ \dot{B}^s_q ( A ) } := \bigg( \sum_{ n \in \mathbb{Z} } 2^{ n s q } \| \dot{\varphi}_n ( A ) x \|^q \bigg)^{ 1 / q } < \infty \right\}
\end{align*}
and 
\begin{align*}
 B^s_q ( A ) := \Bigg\{ x \in \dot{X}_{ - N } + \dot{X}_{ N }: \| x \|_{ B^s_q ( A ) } := \bigg( \sum_{ n \in \mathbb{N}_0 } 2^{ n s q } \| \varphi_n ( A ) x \|^q \bigg)^{ 1 / q } < \infty \Bigg\}
\end{align*}
(with standard modification if $q = \infty$), where $ - N < s < N$. These spaces are independent of the choice of partition of unity and index $N$, so that they are well defined as Banach spaces. We refer to the reader to \cite[Section 5]{Kriegler and Weis2016} for more information. 

\begin{prop}\cite[Theorems 5.2 and 5.3 and Remark 5.4]{Kriegler and Weis2016}\label{p:Besov of Kriegler and Weis}
Let $A$ be a sectorial operator of angle $\omega = 0$ with an $\mathcal{M}^\alpha_1$ calculus, 
and let $1 \leq q \leq \infty$. Furthermore, let $f : ( 0, \infty ) \rightarrow \mathbb{C}$ be a function satisfying $\sum_{ k \in \mathbb{Z} } \| f \dot{\varphi}_0 ( 2^{ - k } \cdot ) \|_{ \mathcal{M}^\alpha_1 } 2^{ - k s } < \infty$ and $f^{ - 1 } \dot{\varphi}_0 \in \mathcal{M}^\alpha_1$. Then, with standard modification for $q = \infty$, the following two statements hold. 
\begin{itemize}
\item[1.] For the homogeneous Besov space $\dot{B}^s_q ( A )$ with $s \in \mathbb{R}$, we have the norm equivalence
\begin{align*}
\| x \|_{ \dot{B}^s_q ( A ) } \simeq \left\{ \int_0^\infty t^{ - s q } \big\| f ( t A ) x \big\|^q \, \frac{ d t }{ t } \right\}^{ 1 / q }.
\end{align*}
\item[2.] For the inhomogeneous Besov space $B^s_q ( A )$ with $s > 0$, we have the norm equivalence
\begin{align*}
\| x \|_{ B^s_q ( A ) } \simeq \| x \| + \left\{ \int_0^\infty t^{ - s q } \big\| f ( t A ) x \big\|^q \, \frac{ d t }{ t } \right\}^{ 1 / q }.
\end{align*}
\end{itemize}
\end{prop}

Let $s \in \mathbb{R}$ and fix $0 \leq \alpha \leq \beta < \infty$ such that $0 < \alpha - s < \beta < \infty$. Applying $f ( t ) = t^\alpha ( 1 + t )^{ - \beta }$ in Proposition \ref{p:Besov of Kriegler and Weis} yields 
\begin{align*}
 \| x \|_{ \dot{B}^s_q ( A ) } & \simeq \left\{ \int_0^\infty t^{ - s q } \big\| t^\alpha A^\alpha ( 1 + t A )^{ - \beta } x \big\|^q \, \frac{ d t }{ t } \right\}^{ 1 / q } \\
& = \left\{ \int_0^\infty t^{ s q } \big\| t^{ \beta - \alpha } A^\alpha ( t + A )^{ - \beta } x \big\|^q \, \frac{ d t }{ t } \right\}^{ 1 / q },
\end{align*}
which coincides with our Besov norm $\| x \|_{ \dot{B}^{ s, A }_{ q, X } ( \beta - \alpha, \alpha ) }$ as that given in (\ref{e:Besov quasi-norms-homogeneous}). On the other hand, if $s > 0$, again taking $0 \leq \alpha \leq \beta < \infty$ such that $0 < \alpha - s < \beta < \infty$ and applying $f ( t ) = t^\alpha ( 1 + t )^{ - \beta }$ in Proposition \ref{p:Besov of Kriegler and Weis} yields 
\begin{align*}
 \| x \|_{ B^s_q ( A ) } & \simeq \| x \| + \left\{ \int_0^\infty t^{ - s q } \big\| t^\alpha A^\alpha ( 1 + t A )^{ - \beta } x \big\|^q \frac{ d t }{ t } \right\}^{ 1 / q } \\
& = \| x \| + \left\{ \int_0^\infty t^{ s q } \big\| t^{ \beta - \alpha } A^\alpha ( t + A )^{ - \beta } x \big\|^q \frac{ d t }{ t } \right\}^{ 1 / q }, 
\end{align*}
which coincides with our Besov norm $\| x \|_{ B^{ s, A }_{ q, X } ( 0, \beta - \alpha, \alpha ) }$ given by (\ref{e:quasi-norm of inhomogeneous Besov spaces}).

\smallskip


\subsection{Semigroups revisited}

Let $0 < \theta \le \pi / 2$, and let $A$ be the negative generator of a bounded analytic semigroup $\mathcal{T}$ of angle $\theta$ on $X$. It is well known that $A$ is non-negative (more precisely, sectorial of angle $\pi / 2 - \theta$) on $X$ with dense domain.
By a routine calculation of complex integrals, one can verify that 
\begin{align}\label{e:uniform boundedness-powers-semigroups}
\sup_{ t > 0 } \left\| ( t A )^\alpha \mathcal{T} ( t ) \right\| < \infty 
\end{align}
for $\alpha \in \mathbb{C}_+$ (see \cite[Theorem 12.1]{Komatsu1966}).
Let $\alpha \in \mathbb{C}_+$ and $s > 0$, and write $C_\mathcal{T} := \sup_{ t \ge 0 } \| \mathcal{T} ( t ) \|$. From (\ref{e:uniform boundedness-powers-semigroups}) it can be
seen that 
\begin{align}\label{e:dyadic estimates-semigroups}
\left\| ( t A )^\alpha \mathcal{T} ( t ) \right\| \le 2^\alpha C_\mathcal{T} \left\| ( s A )^\alpha \mathcal{T} ( s ) \right\|, \quad s \le t \le 2 s,
\end{align}
by observing the well known semigroup property, i.e., $\mathcal{T} ( t + s ) = \mathcal{T} ( t ) \mathcal{T} ( s )$ for $t, s \ge 0$. Moreover, by a standard calculation of real integrals, one can also verify a semigroup version of Lemma \ref{l:representation of fractional powers-beta-weak convergence} (see \cite[Theorem 5.4]{Komatsu1967}).
In particular, one has 
\begin{align}\label{e:fractional powers-semigroups}
A^\alpha x = \frac{ 1 }{ \Gamma ( \beta - \alpha ) } \int_0^\infty t^{ - \alpha } ( t A )^\beta \mathcal{T} ( t ) x \, \frac{ d t }{ t }
\end{align}
for $x \in D ( A^{ \alpha + \epsilon } )$ with $\epsilon > 0$.

Let $s > 0$ and $0 < q \leq \infty$. For $k \in \mathbb{Z}$ and $\beta \in \mathbb{C}_+$ with $s < \RE \beta$, we define $\tilde{B}^{ s, \mathcal{T} }_{ q, X } ( k, \beta )$ to be the completion of $X$ with respect to the quasi-norm 
\begin{align}\label{e:Besov quasinorm-inhomogeneous-semigroups}
\| x \|_{ \tilde{B}^{ s, \mathcal{T} }_{ q, X } ( k, \beta ) } := \| x \| + | x |_{ \tilde{R}^{ s, \mathcal{T} }_{ q, X } ( k, \beta ) }, 
\end{align}
where 
\begin{align*}
| x |_{ \tilde{R}^{ s, \mathcal{T} }_{ q, X } ( k, \beta ) } := \Bigg\{ \sum_{ j = k }^\infty \big\| 2^{ j ( s - \beta ) } A^{ \beta } \mathcal{T} ( 2^{ - j } ) x \big\|^q \Bigg\}^{ 1 / q }.
\end{align*}


The following two propositions state that the resolvent of the underlying operator $A$ can be replaced by the semigroup generated by $A$ in the description of abstract Besov spaces.

\begin{prop}\label{p:bounded analytic semigroups-inhomogeneous}
Let $s > 0$ and $0 < q \leq \infty$ and let $A$ be the negative generator of a bounded analytic semigroup $\mathcal{T}$ on $X$. Fix $k \in \mathbb{Z}$ and $\beta \in \mathbb{C}_+$ such that $s < \RE \beta$. Then, for $x \in X$,
\begin{align*}
\| x \|_{ \tilde{B}^{ s, \mathcal{T} }_{ q, X } } := \| x \|_{ \tilde{B}^{ s, \mathcal{T} }_{ q, X } ( k, \beta ) } \simeq \| x \|_{ B^{ s, A }_{ q, X } },
\end{align*}
where $\left\| \cdot \right\|_{ \tilde{B}^{ s, \mathcal{T} }_{ q, X } ( k, \beta ) }$ is the quasi-norm on $X$ given by (\ref{e:Besov quasinorm-inhomogeneous-semigroups}).
Therefore, the quasi-Banach space $\tilde{B}^{ s, \mathcal{T} }_{ q, X } ( k, \beta )$ is independent of the choice of $k$
and $\beta$ and 
\begin{align*}
\tilde{B}^{ s, \mathcal{T} }_{ q, X } := \tilde{B}^{ s, \mathcal{T} }_{ q, X } ( k, \beta ) = B^{ s, A }_{ q, X }
\end{align*}
in the sense of equivalent quasi-norms.
\end{prop}

\begin{proof}
First we verify that 
\begin{align}\label{e:bounded analytic semigroups-inhomogeneous-T<B}
\| x \|_{ \tilde{B}^{ s, \mathcal{T} }_{ q, X } ( k, \beta ) } \lesssim \| x \|_{ B^{ s, A }_{ q, X } }. 
\end{align}
Let $x \in X$ such that $\| x \|_{ B^{ s, A }_{ q, X } } < \infty$ and fix an integer $m > \RE \beta$. Applying 
(\ref{e:representation of fractional powers-beta-domain of A-alpha+e}) to $A^\beta$ yields  
\begin{align*}
| x |_{ \tilde{R}^{ s, \mathcal{T} }_{ q, X } ( k, \beta ) } 
= {} & C \Bigg\{ \sum_{ j = k }^\infty \bigg\| \int_0^\infty 2^{ j ( s - \beta ) } \lambda^\beta
A^m ( \lambda + A )^{ - m } \mathcal{T} (2^{ - j } ) x \, \frac{ d \lambda}{ \lambda } \bigg\|^q \Bigg\}^{ 1 / q },
\end{align*}
where $C = \frac{ | \Gamma( m ) | }{ | \Gamma ( \beta ) \Gamma ( m - \beta ) | }$.
Rewriting the integral $\int_0^\infty = \sum_{ i = - \infty }^{ \infty } \int_{ 2^i }^{ 2^{ i +1 } }$ yields 
\begin{align*}
| x |_{ \tilde{R}^{ s, \mathcal{T} }_{ q, X } ( k, \beta ) } 
\lesssim \Bigg\{ \sum_{ j = k }^\infty \bigg[ \sum_{ i = - \infty }^{ \infty } \int_{ 2^i }^{ 2^{ i + 1 } } \big\| 2^{ j ( s - \beta ) } \lambda^\beta A^m ( \lambda + A )^{ - m } \mathcal{T} (2^{ - j } ) x \big\| \, \frac{ d \lambda}{ \lambda } \bigg]^q \Bigg\}^{ 1 / q }.
\end{align*}
Applying the estimate (\ref{e:resolvent estimate-discrete}) with $c = 1$ yields 
\begin{align*}
| x |_{ \tilde{R}^{ s, \mathcal{T} }_{ q, X } ( k, \beta ) } \lesssim J,
\end{align*}
where 
\begin{align*}
J = \Bigg\{ \sum_{ j = k }^\infty \bigg[ \sum_{ i = - \infty }^{ \infty } \big\| 2^{ j ( s - \beta ) } 2^{ i \beta } A^m ( 2^i + A )^{ - m } \mathcal{T} (2^{ - j } ) x \big\| \bigg]^q \Bigg\}^{ 1 / q }.
\end{align*}
It can be verified that
\begin{align}\label{e:bounded analytic semigroups-inhomogeneous-proof1}
J \lesssim \| x \|_{ B^{ s, A }_{ q, X } }.
\end{align}
We merely verify (\ref{e:bounded analytic semigroups-inhomogeneous-proof1}) in the case $0 < q \leq 1$ and the case when $1 < q \leq \infty$ can be verified analogously. To this end, let $0 < q \leq 1$. It follows from (\ref{e: q-inequality-series<1}) that 
\begin{align*}
J & \leq \Bigg\{ \sum_{ j = k }^\infty \sum_{ i = - \infty }^{ \infty } \big\| 2^{ j ( s - \beta ) } 2^{ i \beta } A^m ( 2^i + A )^{ - m } \mathcal{T} (2^{ - j } ) x \big\|^q \Bigg\}^{ 1 / q } \\
& = J_1 + J_2 + J_3,
\end{align*}
where 
\begin{align*}
J_1 & = \Bigg\{ \sum_{ j = k }^\infty \sum_{ i = - \infty }^{ k - 1 } \big\| 2^{ j ( s - \beta ) } 2^{ i \beta } A^m ( 2^i + A )^{ - m } \mathcal{T} (2^{ - j } ) x \big\|^q \Bigg\}^{ 1 / q }, \\
J_2 & = \Bigg\{ \sum_{ j = k }^\infty \sum_{ i = k }^{ j - 1 } \big\| 2^{ j ( s - \beta ) } 2^{ i  \beta } A^m ( 2^i + A )^{ - m } \mathcal{T} (2^{ - j } ) x \big\|^q \Bigg\}^{ 1 / q }, \\
J_3 & = \Bigg\{ \sum_{ j = k }^\infty \sum_{ i = j }^{ \infty } \big\| 2^{ j ( s - \beta ) } 2^{ i  \beta } A^m ( 2^i + A )^{ - m } \mathcal{T} (2^{ - j } ) x \big\|^q \Bigg\}^{ 1 / q }.
\end{align*}
As for the part $J_1$, applying (\ref{e:uniform boundedness compositions-L}) to $A^m ( 2^i + A )^{ - m }$ yields 
\begin{align*}
J_1 \lesssim {} & \Bigg\{ \sum_{ j = k }^\infty \sum_{ i = - \infty }^{ k - 1 } \big\| 2^{ j ( s - \beta ) } 2^{ i \beta } \mathcal{T} (2^{ - j } ) x \big\|^q \Bigg\}^{ 1 / q } \lesssim \| x \|,
\end{align*}
where the last inequality follows from the uniform boundedness of $\{ \mathcal{T} ( t ) \}_{ t \ge 0 }$ and the fact that $0 < s < \RE \beta$. 
As for the part $J_2$, exchanging the order of summation yields 
\begin{align*}
J_2 & = \Bigg\{ \sum_{ i = k }^\infty \sum_{ j = i + 1 }^{ \infty } \big\| 2^{ j ( s - \beta ) } 2^{ i \beta } A^m ( 2^i + A )^{ - m } \mathcal{T} (2^{ - j } ) x \big\|^q \Bigg\}^{ 1 / q } \\
& \lesssim \Bigg\{ \sum_{ i = k }^\infty \bigg[ \sum_{ j = i + 1 }^{ \infty } 2^{ j ( s - \RE \beta ) q } \bigg] \big\| 2^{ i \beta } A^m ( 2^i + A )^{ - m } x \big\|^q \Bigg\}^{ 1 / q } \\
& \simeq \Bigg\{ \sum_{ i = k }^\infty \big\| 2^{ j s } A^m ( 2^i + A )^{ - m } x \big\|^q \Bigg\}^{ 1 / q } = | x |_{ R^{ s, A }_{ q, X } ( k, 0, m ) },
\end{align*}
And as for the part $J_3$, exchanging the order of summation yields
\begin{align*}
J_3 & = \Bigg\{ \sum_{ i = k }^{ \infty } \sum_{ j = k }^i \big\| 2^{ j ( s - \beta ) } 2^{ i \beta } A^m ( 2^i + A )^{ - m } \mathcal{T} (2^{ - j } ) x \big\|^q \Bigg\}^{ 1 / q } \\
& = \Bigg\{ \sum_{ i = k }^{ \infty } \sum_{ j = k }^i \big\| 2^{ j s } 2^{ i \beta } ( 2^{ - j } A )^{ \beta } \mathcal{T} ( 2^{ - j } ) A^{ m - \beta } ( 2^i + A )^{ - m } x \big\|^q \Bigg\}^{ 1 / q }.
\end{align*}
Thanks to (\ref{e:uniform boundedness-powers-semigroups}), applying the uniform boundedness of $( 2^{ - j } A )^{ \beta } \mathcal{T} ( 2^{ - j } )$ yields 
\begin{align*}
J_3 & \lesssim \Bigg\{ \sum_{ i = k }^{ \infty } \Bigg( \sum_{ j = k }^i 2^{ j s q } \Bigg) \big\| 2^{ i \beta } A^{ m - \beta } ( 2^i + A )^{ - m } x \big\|^q \Bigg\}^{ 1 / q } \\
& \lesssim \Bigg\{ \sum_{ i = k }^{ \infty } \big\| 2^{ i ( s + \beta ) } A^{ m - \beta } ( 2^i + A )^{ - m } x \big\|^q \Bigg\}^{ 1 / q } = | x |_{ R^{ s, A }_{ q, X } ( k, \beta, m - \beta ) }.
\end{align*}
This implies that  
\begin{align*}
J \leq J_1 + J_2 + J_3 \lesssim \| x \|_{ B^{ s, A }_{ q, X } },
\end{align*}
which is the desired (\ref{e:bounded analytic semigroups-inhomogeneous-proof1}). Thus, we have verified that (\ref{e:bounded analytic semigroups-inhomogeneous-T<B}).
%
%
%

Next we verify that 
\begin{align}\label{e:bounded analytic semigroups-inhomogeneous-B<T}
| x |_{ R^{ s, A }_{ q, X } ( k, 0, m ) } \lesssim \| x \| + | x |_{ \tilde{R}^{ s, \mathcal{T} }_{ q, X } ( k, m ) } + | x |_{ \tilde{R}^{ s, \mathcal{T} }_{ q, X } ( k, 2m ) } 
\end{align}
for each integer $m > \RE \beta$. 
To this end, fix an integer $m > \RE \beta$ and let $x \in X$ such that $| x |_{ \tilde{R}^{ s, \mathcal{T} }_{ q, X } ( k, m ) }$ and $| x |_{ \tilde{R}^{ s, \mathcal{T} }_{ q, X } ( k, 2m ) }$ are both finite.
From (\ref{e:fractional powers-semigroups}) it follows that 
\begin{align*}
| x |_{ R^{ s, A }_{ q, X } ( k, 0, m ) }
\le {} & \frac{ 1 }{ \Gamma ( m ) } \Bigg\{ \sum_{ j = k }^{ \infty } \bigg[ \int_0^\infty \big\| 2^{ j s } t^m A^{ 2 m } \mathcal{T} ( t ) ( 2^j + A )^{ - m } x \big\| \, \frac{ d t }{ t } \bigg]^q \Bigg\}^{ 1 / q } \\
= {} & \frac{ 1 }{ \Gamma ( m ) } \Bigg\{ \sum_{ j = k }^{ \infty } \bigg[ \int_0^\infty \big\| 2^{ j s } t^{ - m } A^{ 2 m } \mathcal{T} ( t^{ - 1 } ) ( 2^j + A )^{ - m } x \big\| \, \frac{ d t }{ t } \bigg]^q \Bigg\}^{ 1 / q }.
\end{align*}
Rewriting $\int_0^\infty = \sum_{ i = - \infty }^{ \infty } \int_{ 2^{ i - 1 } }^{ 2^i }$ and applying (\ref{e:dyadic estimates-semigroups}) to $t^{ - m } A ^m \mathcal{T} ( 2^{ - j } )$ yields 
\begin{align*}
| x |_{ R^{ s, A }_{ q, X } }
\lesssim {} & K,
\end{align*}
where 
\begin{align*}
K = \Bigg\{ \sum_{ j = k }^{ \infty } \bigg[ \sum_{ i = - \infty }^{ \infty } \big\| 2^{ j s }2^{ - i m } A^{ 2 m } \mathcal{T} ( 2^{ - i } ) ( 2^j + A )^{ - m } x \big\| \bigg]^q \Bigg\}^{ 1 / q }.
\end{align*}
In order to verify (\ref{e:bounded analytic semigroups-inhomogeneous-B<T}), it needs merely to verify that 
\begin{align}\label{e:bounded analytic semigroups-inhomogeneous-proof2}
K \lesssim \| x \| + | x |_{ \tilde{R}^{ s, \mathcal{T} }_{ q, X } ( k, m ) } + | x |_{ \tilde{R}^{ s, \mathcal{T} }_{ q, X } ( k, 2 m ) }.
\end{align}
We merely verify (\ref{e:bounded analytic semigroups-inhomogeneous-proof2}) in the case $0 < q \leq 1$ and the case when $1 < q \leq \infty$ can be verified analogously. To this end, let $0 < q \leq 1$. It follows from (\ref{e: q-inequality-series<1}) that 
\begin{align*}
K \le {} & \Bigg\{ \sum_{ j = k }^{ \infty } \sum_{ i = - \infty }^{ \infty } \big\| 2^{ j s }2^{ - i m } A^{ 2 m } \mathcal{T} ( 2^{ - i } ) ( 2^j + A )^{ - m } x \big\|^q \Bigg\}^{ 1 / q } \\
= {} & K_1 + K_2 + K_3, 
\end{align*}
where 
\begin{align*}
K_1 = {} & \Bigg\{ \sum_{ j = k }^{ \infty } \sum_{ i = - \infty }^{ k - 1 } \big\| 2^{ j s }2^{ - i m } A^{ 2 m } \mathcal{T} ( 2^{ - i } ) ( 2^j + A )^{ - m } x \big\|^q \Bigg\}^{ 1 / q }, \\
K_2 = {} & \Bigg\{ \sum_{ j = k }^{ \infty } \sum_{ i = k }^{ j - 1 } \big\| 2^{ j s }2^{ - i m } A^{ 2 m } \mathcal{T} ( 2^{ - i } ) ( 2^j + A )^{ - m } x \big\|^q \Bigg\}^{ 1 / q }, \\
K_3 = {} & \Bigg\{ \sum_{ j = k }^{ \infty } \sum_{ i = j }^{ \infty } \big\| 2^{ j s }2^{ - i m } A^{ 2 m } \mathcal{T} ( 2^{ - i } ) ( 2^j + A )^{ - m } x \big\|^q \Bigg\}^{ 1 / q }.   
\end{align*}
As for the part $K_1$, applying (\ref{e:uniform boundedness-powers-semigroups}) to $( 2^{ - i } A )^{ 2 m } \mathcal{T} ( 2^{ - i } )$ yields 
\begin{align*}
 K_1 = {} & \Bigg\{ \sum_{ j = k }^{ \infty } \sum_{ i = - \infty }^{ k - 1 } \big\| 2^{ j s } 2^{ i m } ( 2^{ - i } A )^{ 2 m } \mathcal{T} ( 2^{ - i } ) ( 2^j + A )^{ - m } x \big\|^q \Bigg\}^{ 1 / q } \\
\lesssim {} & \Bigg\{ \sum_{ j = k }^{ \infty } \sum_{ i = - \infty }^{ j - 1 } \big\| 2^{ j ( s - m ) } 2^{ i m } 2^{ j m } ( 2^j + A )^{ - m } x \big\|^q \Bigg\}^{ 1 / q } \lesssim \| x \|.
\end{align*}
where the last inequality follows from the uniform boundedness of $2^{ j m } ( 2^j + A )^{ - m }$ due to (\ref{e:uniform boundedness compositions-M}).
As for the pret $K_2$, exchanging the order of summation yields and applying (\ref{e:uniform boundedness compositions-M}) to $2^{ j m } ( 2^j + A )^{ - m }$ yields
\begin{align*}
K_2 = {} & \Bigg\{ \sum_{ i = k }^{ \infty } \sum_{ j = i + 1 }^{ \infty } \big\| 2^{ j ( s - m ) }2^{ - i m } A^{ 2 m } \mathcal{T} ( 2^{ - i } ) 2^{ j m } ( 2^j + A )^{ - m } x \big\|^q \Bigg\}^{ 1 / q } \\
\lesssim {} & \Bigg\{ \sum_{ i = k }^{ \infty } \bigg[ \sum_{ j = i + 1 }^{ \infty } 2^{ j ( s - m ) q } \bigg] \big\|  2^{ - i m } A^{ 2 m } \mathcal{T} ( 2^{ - i } ) x \big\|^q \Bigg\}^{ 1 / q } \\
\simeq {} & \Bigg\{ \sum_{ i = k }^{ \infty } \big\|  2^{ i ( s - 2 m ) } A^{ 2 m } \mathcal{T} ( 2^{ - i } ) x \big\|^q \Bigg\}^{ 1 / q } \le | x |_{ \tilde{R}^{ s, \mathcal{T} }_{ q, X } ( k, 2cm ) }.
\end{align*}
And as for the part $K_3$, exchanging the order of summation yields and applying (\ref{e:uniform boundedness compositions-L}) to $A^m ( 2^j + A )^{ - m }$ yields
\begin{align*}
K_3 = {} & \Bigg\{ \sum_{ i = k }^{ \infty } \sum_{ j = k }^{ i } \big\| 2^{ j s }2^{ - i m } A^{ 2 m } \mathcal{T} ( 2^{ - i } ) ( 2^j + A )^{ - m } x \big\|^q \Bigg\}^{ 1 / q } \\
\lesssim {} & \Bigg\{ \sum_{ i = k }^{ \infty } \bigg( \sum_{ j = k }^{ i } 2^{ j s q } \bigg) \big\| 2^{ - i m } A^{ m } \mathcal{T} ( 2^{ - i } ) x \big\|^q \Bigg\}^{ 1 / q } \lesssim | x |_{ \tilde{R}^{ s, \mathcal{T} }_{ q, X } ( k, m ) }.  
\end{align*}
This implies that 
\begin{align*}
K \le K_1 + K_2 + K_3 \lesssim \| x \| + | x |_{ \tilde{R}^{ s, \mathcal{T} }_{ q, X } ( k, m ) } + | x |_{ \tilde{R}^{ s, \mathcal{T} }_{ q, X } ( k, 2 m ) },   
\end{align*}
which is the desired (\ref{e:bounded analytic semigroups-inhomogeneous-proof2}). Thus, we have verified (\ref{e:bounded analytic semigroups-inhomogeneous-B<T}).

Finally, from (\ref{e:bounded analytic semigroups-inhomogeneous-B<T}) and (\ref{e:bounded analytic semigroups-inhomogeneous-T<B}) it follows that 
\begin{align*}
\| x \|_{ B^{ s, A }_{ q, X } } \simeq {} & \| x \| + | x |_{ R^{ s, A }_{ q, X } ( k, 0, m ) } \\
\lesssim {} & \| x \| + | x |_{ \tilde{R}^{ s, \mathcal{T} }_{ q, X } ( k, m ) } + | x |_{ \tilde{R}^{ s, \mathcal{T} }_{ q, X } ( k, 2 m ) } \\
\lesssim {} & \| x \|_{ \tilde{B}^{ s, \mathcal{T} }_{ q, X } ( k, m ) } + \| x \|_{ \tilde{B}^{ s, \mathcal{T} }_{ q, X } ( k, m ) } \lesssim \| x \|_{ B^{ s, A }_{ q, X } },
\end{align*}
from which the desired conclusion follows immediately.
The proof is complete.
\end{proof}

A homogeneous version of Proposition \ref{p:bounded analytic semigroups-inhomogeneous} is given as follows. 

\begin{prop}\label{p:bounded analytic semigroups-homogeneous}
Let $s > 0$ and $0 < q \leq \infty$ and let $A$ be the negative generator of a bounded analytic semigroup $\mathcal{T}$ on $X$. Fix $\beta \in \mathbb{C}_+$ such that $s < \RE \beta$ and let $x \in \overline{ R ( A ) }$. If $A$ is injective, then 
\begin{align*}
| x |_{ \dot{R}^{ s, A }_{ q, X } ( 0, \beta ) } 
& \simeq \Bigg\{ \sum_{ j = - \infty }^\infty \big\| 2^{ j s } (
2^{ - j } A )^{ \beta } \mathcal{T} ( 2^{ - j } ) x \big\|_X^q \Bigg\}^{
1 / q } \\
& = \Bigg\{ \sum_{ j = - \infty }^\infty \big\| 2^{ - j s } ( 2^{
j } A )^{ \beta } \mathcal{T} ( 2^{ j } ) x \big\|_X^q \Bigg\}^{ 1 / q }
:= | x |_{ \dot{ \tilde{R} }^{ s, \mathcal{T} }_{ q, X } ( \beta ) }.
\end{align*}
If this is the case, $\dot{ B }^{ s, A }_{ q, X }$ is the completion of $\dot{ R }^{ s, A }_{ q, X } ( 0, \beta )$ with respect to the quasi-norm $\left| \cdot \right|_{ \dot{ \tilde{R} }^{ s, \mathcal{T} }_{ q, X } ( \beta ) }$.
\end{prop}

\begin{proof}
Assume that $A$ is injective.  Let $x \in \overline{ R ( A ) }$ and fix an integer $m > \RE \beta$. Let $| x |_{ \dot{R}^{ s, A }_{ q, X } ( 0, \beta ) } < \infty$. Analogous to (\ref{e:bounded analytic semigroups-inhomogeneous-proof1}), by using (\ref{e:representation of fractional powers-beta-domain of A-alpha+e}) to $A^\beta$ and (\ref{e:resolvent estimate-discrete}) with $c = 1$, 
we conclude that   
\begin{align*}
| x |_{ \dot{ R }^{ s, A }_{ q, X } ( \beta ) }^\mathcal{T} \lesssim J,
\end{align*}
where 
\begin{align*}
J = \Bigg\{ \sum_{ j = - \infty }^\infty \bigg[ \sum_{ i = - \infty }^{ \infty } \big\| 2^{ j ( s - \beta ) } 2^{ i \beta } A^m ( 2^i + A )^{ - m } \mathcal{T} (2^{ - j } ) x \big\| \bigg]^q \Bigg\}^{ 1 / q }.
\end{align*}
By using the decomposition $\sum_{ i = - \infty }^{ \infty } = \sum_{ i = - \infty }^{ j - 1 } + \sum_{ i = j }^{ \infty }$ and appplying the classical inequality (\ref{e: q-inequality-series<1}) for $0 < q \le 1$ and the H\"{o}lder inequality for $1 < q < \infty$ (the case when $q = \infty$ can be verified directly by use of the corresponding estimates) and we can conclude that 
\begin{align*}
J \lesssim  | x |_{ \dot{R}^{ s, A }_{ q, X } ( 0, \beta ) }. 
\end{align*}
This implies that 
\begin{align*}
| x |_{ \dot{ R }^{ s, A }_{ q, X } ( \beta ) }^\mathcal{T} \lesssim | x |_{ \dot{R}^{ s, A }_{ q, X } ( 0, \beta ) }. 
\end{align*}

Conversely, let $| x |_{ \dot{ R }^{ s, A }_{ q, X } ( \beta ) }^\mathcal{T} < \infty$. By using  (\ref{e:fractional powers-semigroups}) and (\ref{e:dyadic estimates-semigroups}) we have 
\begin{align*}
| x |_{ \dot{R}^{ s, A }_{ q, X } ( 0, \beta ) } \lesssim K,
\end{align*}
where 
\begin{align*}
K = {} & \Bigg\{ \sum_{ j = - \infty }^{ \infty } \bigg[ \sum_{ i = - \infty }^{ \infty } \big\| 2^{ j s }2^{ - i m } A^{ 2 m } \mathcal{T} ( 2^{ - i } ) ( 2^j + A )^{ - m } x \big\| \bigg]^q \Bigg\}^{ 1 / q }.
\end{align*}
Analogously, by using the decomposition $\sum_{ i = - \infty }^{ \infty } = \sum_{ i = - \infty }^{ j - 1 } + \sum_{ i = j }^{ \infty }$ and applying the classical inequality (\ref{e: q-inequality-series<1}) for $0 < q \le 1$ and the H\"{o}lder inequality for $1 < q < \infty$ (the case when $q = \infty$ can be verified directly by use of the corresponding estimates) and we can conclude that 
\begin{align*}
K \lesssim | x |_{ \dot{ R }^{ s, A }_{ q, X } ( \beta ) }^\mathcal{T}.
\end{align*}
This implies that 
\begin{align*}
| x |_{ \dot{R}^{ s, A }_{ q, X } ( 0, \beta ) } \lesssim | x |_{ \dot{ R }^{ s, A }_{ q, X } ( \beta ) }^\mathcal{T}.
\end{align*}
The proof is complete.
\end{proof}

Let $A$ be the negative generator of a bounded $C_0$-semigroup $T$ on $X$. It is well know that $- A^\alpha$ generates a bounded analytic semigroup $\mathcal{T}_\alpha$ of angle $\frac{ \pi }{ 2 } ( 1 - \alpha )$ for each $0 < \alpha < 1$ (sometimes we call $\mathcal{T}_\alpha$ the subordinated semigroup of $T$). More precisely, 
\begin{align*}
\mathcal{T}_\alpha ( t ) = \int_0^\infty k_\alpha ( t, s ) T ( s ) \, d s, \quad t > 0, 
\end{align*}
where $k_\alpha \left( \cdot, \cdot \right)$ is the so-called subordinated kernel given by 
\begin{align*}
k_\alpha ( s, t ) = \frac{ 1 }{ \pi } \int_0^\infty \sin ( r s \sin \pi \alpha ) \exp ( - r t - r^\alpha s \cos \pi \alpha ) \, d r, \quad s, t > 0.
\end{align*}
See \cite[Corollary 3.3 (a)]{LiandChen2010}. In particular, if $T$ is bounded analytic of angle $\theta$ for some $0 < \theta \le \pi / 2$ then $- A^\alpha$ also generates a bounded analytic semigroup $\mathcal{T}_{ \alpha }$ of angle $\frac{ \pi }{ 2 } - \alpha ( \frac{ \pi }{ 2 } - \theta ) $ for each $1 < \alpha < \frac{ \pi }{ \pi - 2 \theta }$ (see \cite[Proposition 3.5 (a)]{LiandChen2010}).
The following corollary is a direct consequence of Theorem \ref{t:smoothness reiteration} and Proposition \ref{p:bounded analytic semigroups-inhomogeneous}.

\begin{cor}\label{c:bounded analytic semigroups-inhomogeneous}
Let $s > 0$ and $0 < q \le \infty$ and let $A$ be the negative generator of a bounded $C_0$-semigroup $T$ on $X$. The following statements hold.
\begin{itemize}
\item[(i)] For $0 < \alpha < 1$, $B^{ s, A }_{ q, X }$ is the completion of $X$ with respect to $\left\| \cdot \right\|_{ \tilde{B}^{ s / \alpha, \mathcal{T}_\alpha }_{ q, X } }$, where $\mathcal{T}_\alpha$ is the bounded analytic semigroup generated by $- A^\alpha$.
\item[(ii)] If $T$ is bounded analytic of angle $\theta$ for some $0 < \theta \le \pi / 2$, then $B^{ s, A }_{ q, X }$ is the completion of $X$ with respect to $\left\| \cdot \right\|_{ \tilde{B}^{ s / \alpha, \mathcal{T}_\alpha }_{ q, X } }$ for each $0 < \alpha < \frac{ \pi }{ \pi - 2 \theta }$, where $\mathcal{T}_\alpha$ is the bounded analytic semigroup generated by $- A^\alpha$.
\end{itemize}
\end{cor}

\begin{rem}
By Proposition \ref{p:bounded analytic semigroups-inhomogeneous} and Remark \ref{r:inhomogeneous Besov spaces-equivalent quasi-norms-continuity} (i), we obtain directly \cite[Theorem 5.3]{Komatsu1967} for $s > 0$ and $1 \le q \le \infty$. Also, see \cite[Chapter 1, Section 4]{AshyralyevM1994} for a homogeneous version for abstract Besov norms with $0 < s < 1$ and $1 \le q \le \infty$. Moreover, thanks to Corollary \ref{c:bounded analytic semigroups-inhomogeneous}, one can obtain a variety of abstract Besov spaces by specifying bounded analytic semigroups or their subordinated semigroups on concrete function spaces. See Examples \ref{E:Gaussian semigroup} and \ref{E:Poisson semigroup} below for the classical Besov spaces associated with the fractional Laplacians. We also refer the reader to \cite{ChenZM2017} for some applications of the classical Besov spaces associated with the fractional Laplacians, \cite{DeLeon and Torrea2019,Bui2020,Cao2020} for Besov spaces associated with Schr\"{o}dinger operators,  \cite[Definition 4.1 and Theorem 4.3]{Komatsu1967} for Besov spaces associated with bounded $C_0$-semigroups on Banach spaces, and \cite{Alonso-Ruiz2020} and references therein for Besov spaces associated with heat semigroup on Dirichlet spaces.
\end{rem}

Finally,  
we present some explicit examples to illustrate that, to some extent, 
the fractional powers of (unbounded) operators are indeed a generalization of 
some singular integrals operators (for example, the Riesz potential
and Bessel potential) 
on function spaces,  
and therefore we achieve the
classical Besov spaces $B^{ s }_{ p, q } ( \mathbb{R}^n )$ by applying the Laplacian to abstract Besov spaces. 


Let $n \in \mathbb{N}$ and $1 < p < \infty$, and let $\Delta_p$ be the 
Laplacian on $L^p ( \mathbb{R}^n )$ with maximal domain, i.e., 
\begin{align*}
\Delta_p f := {} &\Delta f, \\
D ( \Delta_p ) := {} & \left\{ f \in L^p ( \mathbb{R}^n ): \Delta f \in L^p ( \mathbb{R}^n ) \right\}, 
\end{align*}
where $\Delta = \sum_{ j = 1 }^{ n } \frac{ \partial^2 }{ \partial x_j^2 }$ is understood in the distributional sense. 
It is well known that $\Delta_p$ in injective and non-negative on $L^p ( \mathbb{R}^n )$, and therefore $\Delta_p$ has dense domain and dense range  
due to the fact that $L^p ( \mathbb{R}^n )$ is reflexive (see \cite[Corollary 1.1.4 (v)]{MartinezM2001} or \cite[Proposition 2.1.1 (h)]{HaaseM2006}). Moreover, the inverse $\Delta_p^{ - 1 }$ of the Laplacian $\Delta_p$ is an unbounded operator on $L^p ( \mathbb{R}^n )$, 
so that $0 \in \sigma ( \Delta_p )$ (see the Riesz potential below).
%



Thanks to the non-negativity and injectivity of
$- \Delta_p$,
the fractional power $( - \Delta_p )^\alpha$ of $- \Delta_p$, i.e., the so-called fractional Laplacian, can be defined for each $\alpha \in \mathbb{C}$ via the Balakrishnan-Komatsu operator as shown in Section \ref{sub:Fractional powers of operators} above. 
Alternatively, the fractional Laplacian also can be characterized via the Fourier transform (see, for example, \cite[Chapter V]{SteinM1970}) or suitable harmonic extension \cite{Caffarelli2007}.
For more information on representations and estimates of the fractional Laplacian, we refer the reader to, for example, \cite{Samko1998,Hytonen2003,Cordoba2005,Frank2008,Miao2008,Abdellaoui2016,BucurM2016,Kwasnicki2017,Abatangelo2020,Case2020,Lischke2020} and references therein.


Now we turn to explicit representations of the fractional Laplacian via singular integral operators.
The following representations (i), (ii) and (iii) are quite standard and we refer the reader to \cite[Chapter V]{SteinM1970} and \cite[Section 12.2]{MartinezM2001} for more information on the Riesz potentials and Bessel potentials.
\begin{itemize}
\item[(i)] Let $0 < s < 2$. The fractional Laplacian $( - \Delta_p )^{ s / 2 }$ 
can be written as a singular integral operator in the following way: 
\begin{align*}
\big( ( - \Delta_p )^{ s / 2 } f \big) ( x ) = c_{ n, s } \mbox{ p.v.} \int_{ \mathbb{R}^n } \frac{ f ( x ) - f ( y ) }{ | x - y |^{ n + s } } \, d y, \quad x \in \mathbb{R}^n, 
\end{align*} 
where $f \in \mathcal{S} ( \mathbb{R}^n )$ and $c_{ n, s } = \frac{ 2^s \Gamma( \frac{ n + s }{ 2 } ) }{ \pi^{ \frac{ n }{ 2 } } | \Gamma ( - \frac{ s }{ 2 } ) | }$ is a normalization constant. Moreover, $( - \Delta_p )^{ s / 2 }$ admits a closed extension to $L^p ( \mathbb{R}^n )$ for $1 < p < \infty$.
\item[(ii)] Let $0 < s < n$.
It can be verified that $( - \Delta_p )^{ - s / 2 } = I_s$, i.e., the so-called Riesz potential which is given by 
\begin{align*}
( I_s f ) ( x ) = d_{ n, s } \int_{ \mathbb{R}^n } \frac{ f ( y ) }{ | x - y |^{ n - s } } \, d y, \quad x \in \mathbb{R}^n, 
\end{align*}
where $f \in \mathcal{S} ( \mathbb{R}^n )$ and $d_{ n, s } = \frac{ \Gamma ( \frac{ n - s }{ 2 } ) }{ 2^s \pi^{ \frac{ n }{ 2 } } \Gamma ( \frac{ s }{ 2 } ) }$ is a normalization constant. Moreover, $( - \Delta_p )^{ - s / 2 }$ admits a bounded $L^p$-$L^q$ extension
for $1 < p < n / s$ and $1 / q = 1 / p - s / n$.
\item[(iii)] It is clear that the translation $I - \Delta_p$ of $- \Delta_p$ is positive, so that $( I - \Delta_p )^{ - s / 2 }$ is bounded on $L^p ( \mathbb{R}^n )$ for each $s \in \mathbb{C}_+$. More precisely, $( I - \Delta_p )^{ - s / 2 } = B_s$ for $s \in \mathbb{C}_+$, where $B_s$ is the so-called Bessel potential defined by $B_s f = G_s * f$ with the Bessel kernel $G_s$ given by 
\begin{align*}
G_s ( x ) = \frac{ 1 }{ ( 4 \pi )^{ s / 2 } \Gamma ( s / 2 ) } \int_{ 0 }^\infty e^{ - \frac{ \pi | x |^2 }{ y } - \frac{ y }{ 4 \pi } } y^{ \frac{ s - n }{ 2 } - 1 } \, d y, \quad x \in \mathbb{R}^n \setminus \{ 0 \}.
\end{align*}
\end{itemize}
Moreover, applying $A = - \Delta_p$ to (\ref{e:representation of fractional powers-all}) yields a unified resolvent representation of the fractional Laplacian, as shown in (iv) below.
\begin{itemize}
\item[(iv)] Let $\alpha, \beta \ge 0$ such that $- n < - \alpha < s < \beta$. From (\ref{e:representation of fractional powers-all}) it follows that 
\begin{align*}
( - \Delta_p )^{ s / 2 } f = C \int_0^\infty \lambda^{ s / 2 } \lambda^\alpha ( - \Delta_p )^\beta ( \lambda - \Delta_p )^{ - \alpha - \beta } f \, \frac{ d \lambda }{ \lambda }
\end{align*}
for $f \in \mathcal{S} ( \mathbb{R}^n )$, where $C = \frac{ \Gamma ( \alpha + \beta ) }{ \Gamma ( \alpha + s ) \Gamma ( \beta - s ) }$. In particular, we can obtain (i) and (ii) above by specifying indices $\alpha$ and $\beta$. 
\end{itemize}

 


It is well known that $D ( \Delta_p ) = W^{ 2, p } ( \mathbb{R}^n )$, i.e., the Sobolev space on $\mathbb{R}^n$. In general, it is also well known that $D \left( ( - \Delta_p )^{ \alpha / 2 } \right) = W^{ \alpha, p } ( \mathbb{R}^n ) = L^p_\alpha ( \mathbb{R}^n )$ for $\alpha \in \mathbb{C}_+$, where $W^{ \alpha, p } ( \mathbb{R}^n )$ and $L^p_\alpha ( \mathbb{R}^n )$ are the so-called fractional Sobolev space and Bessel potential space, respectively (see \cite[Section 12.3]{MartinezM2001}). 
Moreover, applying the negative
Laplacian and fractional Laplacian to
our abstract Besov spaces gives the classical Besov spaces, as shown in Examples \ref{E:Gaussian semigroup} and \ref{E:Poisson semigroup} below.

\begin{example}[Gaussian semigroup]\label{E:Gaussian semigroup}
Let $s > 0$, $0 < q \le \infty$ and $1 < p < \infty$. 
It is well known that 
$\Delta_p$ is the generator of the Gaussian semigroup, a bounded analytic semigroup of angle $\pi / 2$, on $L^p ( \mathbb{R}^n )$. 
Applying $- \Delta_p$ to Theorem \ref{t:smoothness reiteration} and Proposition \ref{p:bounded analytic semigroups-inhomogeneous} yields the classical Besov spaces on $\mathbb{R}^n$, i.e., 
\begin{align*}
B^{ s, (- \Delta_p)^\alpha }_{ q, L^p ( \mathbb{R}^n ) } = B^{ 2 \alpha s }_{ p, q } ( \mathbb{R}^n ), \quad \alpha > 0,
\end{align*}
(in the sense of equivalent quasi-norms)
due to the fact that $B^{ s, - \Delta_p }_{ q, L^p ( \mathbb{R}^n ) } = B^{ 2 s }_{ p, q } ( \mathbb{R}^n )$ (see \cite[Section 2.12.2, Theorem (i)]{TriebelM1983-2010}).
\end{example}

In particular, applying $\alpha = 1 / 2$ in Example \ref{E:Gaussian semigroup} yields the well known description of $B^{ s }_{ p, q } ( \mathbb{R}^n )$ via the Poisson semigroup (also, see
\cite[Section 2.12.2, Theorem (ii)]{TriebelM1983-2010}). Moreover, observe that

\begin{example}[Poisson semigroup]\label{E:Poisson semigroup}
Let $s > 0$, $0 < q \le \infty$ and $1 < p < \infty$. Write $A_p := - H \frac{ \partial }{ \partial x }$, where $H$ is the Hilbert transform on $L^p ( \mathbb{R}^n )$. It is known that $A_p$ is the square root of the Laplacian, i.e., $A_p = - ( - \Delta_p )^{ 1 / 2 }$ with $D ( A_p ) = W^{ 1, p } ( \mathbb{R}^n )$ and that $A_p$ is the generator of the Poisson semigroup on $L^p ( \mathbb{R}^n )$, a bounded analytic semigroup of angle $\pi / 2$. Applying $- A_p$ to Theorem \ref{t:smoothness reiteration} and Example \ref{E:Gaussian semigroup} also yields the classical Besov spaces on $\mathbb{R}^n$, i.e.,
\begin{align*}
B^{ s, - A_p }_{ q, L^p ( \mathbb{R}^n ) } = B^{ s, ( - \Delta_p )^{ 1 / 2 } }_{ q, L^p ( \mathbb{R}^n ) } = B^{ s / 2, - \Delta_p }_{ q, L^p ( \mathbb{R}^n ) } = B^{ s }_{ p, q } ( \mathbb{R}^n ).
\end{align*}
\end{example}

\end{document}